\documentclass{amsart}
\usepackage[utf8]{inputenc}

\usepackage{amsmath, amssymb, amsfonts, amsthm}
\usepackage{enumerate}
\usepackage{bbold}
\usepackage{color}
\usepackage{xcolor}
\usepackage{esvect}
\usepackage{mathtools}
\usepackage{graphicx}
\usepackage{longtable}
\usepackage{fullpage}
\usepackage{rotating}
\usepackage{tikz}
\usepackage{tikz-cd}
\usepackage[boxsize = .3em]{ytableau}
\theoremstyle{plain}
\newtheorem{thm}{Theorem}[section]
\newtheorem{lem}[thm]{Lemma}
\newtheorem{prop}[thm]{Proposition}
\newtheorem{cor}[thm]{Corollary}
\newcommand{\att}[2]{\raisebox{-.5\height}{ \includegraphics[scale = #2]{diagrams/#1.pdf}}}
\theoremstyle{definition}
\newtheorem{defn}[thm]{Definition}
\newtheorem{ex}[thm]{Example}
\newtheorem{rmk}[thm]{Remark}

\newcommand{\bbZ}{\mathbb{Z}}

\newtheoremstyle{named}{}{}{\itshape}{}{\bfseries}{.}{.5em}{#1 \thmnote{(#3) }}
\theoremstyle{named}

\title{Type $\textrm{II}$ quantum subgroups for quantum $\mathfrak{sl}_N$. $\textrm{II}$: Classification}
\author{Cain Edie-Michell and Terry Gannon}
\date{}

\begin{document}

\maketitle
\begin{abstract}
In this paper we study the indecomposable module categories over $\mathcal{C}(\mathfrak{sl}_N, k)$, the category of integrable level-$k$ respresentations of affine Kac-Moody $\mathfrak{sl}_N$. Our first main result classifies these module categories in the case of generic $k$, i.e. $k$ is sufficiently large relative to $N$. As $\mathcal{C}(\mathfrak{sl}_N, k)$ is a braided tensor category, there is a relative tensor product structure on its category of module categories. In the generic setting we obtain a formula for the relative tensor product rules between the indecomposable module categories. Our second main result classifies the indecomposable module categories over $\mathcal{C}(\mathfrak{sl}_N, k)$ for $N\leq 7$, with no restrictions on $k$. In this non-generic setting, exceptional module categories are obtained. This work relies heavily on previous results by the two authors. In previous literature, module category classification results were known only for $\mathfrak{sl}_2$ and $\mathfrak{sl}_3$.
\end{abstract}

\section{Introduction}

Given a tensor category $\mathcal{C}$, it is an interesting question to study its representation theory. More precisely, we want to study $\operatorname{Mod}(\mathcal{C})$, the module categories over $\mathcal{C}$ in the sense of \cite{OstMod}. This is a difficult problem in general, and there has been significant literature dedicated to this problem e.g. \cite{Kril,ETMod,HaagMod,Emily,EH,Gannon,dan}.

Beyond being an interesting algebraic question in its own right, the representation theory of a tensor category $\mathcal{C}$ has several interesting applications. In the case that $\mathcal{C}$ is the category of bifinite bimodules over a type $\textrm{II}_1$ factor $M$, then a $\mathcal{C}$-module category (along with a choice of module object) can be used to construct a subfactor $M\subset N$ \cite{LongSub}. If $\mathcal{C}$ is braided, then $\operatorname{Mod}( \mathcal{C})$ has the structure of a fusion 2-category \cite{fus2}. It is shown by D\'ecoppet that every fusion 2-category is Morita equivalent to $\operatorname{Mod}( \mathcal{C})$ for some braided tensor category $\mathcal{C}$ \cite[Theorem 4.2.2]{thibo}. Hence the categories $\operatorname{Mod}( \mathcal{C})$ are key examples in the theory of fusion 2-categories. Also, a rational conformal field theory can be identified with chiral data such as a vertex operator algebra, together with a module category of the associated modular tensor category \cite{CFTmod}; the module category describes the boundary data and captures how the two chiral halves splice together.

A large class of non-degenerate braided fusion categories come from the representation theory of the affine Lie algebras. Given a simple Lie algebra $\mathfrak{g}$, and a positive integer level $k$, then the category of integrable level-$k$ representations of $\widehat{\mathfrak{g}}$ can be equipped with a tensor product which makes it non-degenerately braided \cite{ten1,KacBook}. These categories are typically denoted $\mathcal{C}(\mathfrak{g}, k)$. Alternatively, the categories $\mathcal{C}(\mathfrak{g}, k)$ can be realised as semisimplifications of the category of tilting modules $\operatorname{Rep}(U_q(\mathfrak{g}))$, where $q$ is a root of unity depending on $k$. Details of this category can be found in \cite{QG1,QG2}, and details of the equivalence can be found in \cite{KL1,KL3,KL4,MR1384612}. These categories have physical relevance, as they are the representation categories of the Wess-Zumino-Witten chiral conformal field theories $\mathcal{V}(\mathfrak{g}, k)$ \cite{WZW,hua}.

The problem of determining $\operatorname{Mod}(\mathcal{C}(\mathfrak{g}, k))$ was initially studied in the physics literature \cite{Wit,ADE}, where the case $\mathfrak{sl}_2$ was considered. Here an $ADE$ classification was obtained. A fully rigorous mathematical treatment of the $\mathfrak{sl}_2$ case was obtained in \cite{OstMod}. This classification result was particularly striking, as it contained a mix of infinite families of module categories, along with three sporadic examples at $k\in \{10,16,28\}$. In \cite{GanSU3,Ocneanu,SU3} the case of $\mathfrak{g} = \mathfrak{sl}_3$ was considered. Again an interesting mix of infinite families, along with sporadic examples at $k \in \{5,9,21\}$ was found. Before this paper, results were not known for larger $\mathfrak{sl}_N$. We remark that the $\mathfrak{sl}_N$ cases can be considered a form of ``higher rank $ADE$''.

The problem of determining the objects of $\operatorname{Mod}(\mathcal{C}(\mathfrak{g}, k))$ (or more generally, the irreducible module categories over any non-degenerately braided tensor category) can be naturally broken up into two pieces. This is due to the results of \cite{triples}, which classify such modules in terms of \textit{\'etale} algebra objects in $\mathcal{C}(\mathfrak{g}, k)$, and braided equivalences between their categories of local modules. 

Recent progress has been made towards the classification of \textit{\'etale} algebra objects in $\mathcal{C}(\mathfrak{g}, k)$. The first groundbreaking result was obtained by Schopieray in \cite{LevelAndy}, where for the rank 2 simple Lie algebras, a bound $N(\mathfrak{g})$ was obtained such that for $k > N(\mathfrak{g})$, the only \textit{\'etale} algebra objects in $\mathcal{C}(\mathfrak{g}, k)$ are pointed (and hence fully understood). This was built on by the second author in \cite{Gannon}, where a bound was obtained for all simple Lie algebras. Furthermore, this bound was effective in that $N(\mathfrak{g})$ grows approximately like the rank of $\mathfrak{g}$ cubed. This effective bound allowed the second author to classify all \textit{\'etale} algebra objects in $\mathcal{C}(\mathfrak{g}, k)$ for many of the low rank examples \cite{newGan}.

The complimentary problem of studying the braided equivalences between the categories of local modules was studied by the first author in the prequel to this paper \cite{ModPt1}. In the case that the \textit{\'etale} algebra is non-pointed, the category of local modules tend to be very simple, and the desired equivalences can be deduced via elementary means. However in the case where the algebra is pointed, the category of local modules is much more complicated. The main contribution of the prequel was to completely classify the braided autoequivalences of these categories of local modules. The surprising result was that for 
\[(N,k) \notin \{(2,16),(3,9),(4,8),(5,5),(8,4),(9,3),(16,2)\},\] there are no exceptional braided autoequivalences. This work built on the first author's work \cite{autos}, which classified the braided autoequivalences in the case of the trivial \textit{\'etale} algebra $\mathbf{1}$.

The purpose of this paper is to combine the results of the authors in order to apply the general classification theory of \cite{triples} to give classification results for $\operatorname{Mod}(\mathcal{C}(\mathfrak{sl}_N, k))$. It should be noted that the methods of \cite{triples} are non-constructive in our setting, and so our classification results are abstract. The complimentary program of explicitly constructing $\mathcal{C}(\mathfrak{sl}_N, k)$ module categories is in progress, with results found in \cite{ HansMod, HansSp,dan}.

Our main result classifies the indecomposable module categories over $\mathcal{C}(\mathfrak{sl}_N, k)$ in the case where all \textit{\'etale} algebra objects are pointed. We also determine decomposition formulae for the relative tensor products between these module categories. In the language of fusion 2-categories, we determine the simple objects of $\operatorname{Mod}(\mathcal{C}\left(\mathfrak{sl}_N, k  \right))$, and the $\boxtimes$-product rules for these simples. From the results of the second author \cite{Gannon}, there are only finitely many $k$ for each $N$ where we have non-pointed \textit{\'etale} algebra objects in $\mathcal{C}(\mathfrak{sl}_N, k)$. Hence we consider this classification theorem as the generic result on $\operatorname{Mod}(\mathcal{C}(\mathfrak{sl}_N, k))$.

    


\begin{thm}\label{thm:bigmain}
    Let $N,k\in \mathbb{N}_{\geq 4}$ be such that the only connected \textit{\'etale} algebra objects in $\mathcal{C}(\mathfrak{sl}_N, k)$ are pointed. Then there is a bijection
    \[  (d, \varepsilon) \mapsto \mathcal{M}_{d,\varepsilon}        \]
    between pairs $(d,\varepsilon)$ where $\varepsilon\in \{+, -\}$, and $d$ is a divisor of $\frac{N}{2}$ if $N$ is even and $k$ is odd, or a divisor of $N$ otherwise, and indecomposable module categories $\mathcal{M}$ of $\mathcal{C}(\mathfrak{sl}_N, k)$ up to equivalence.

    The decomposition of the relative tensor products between these indecomposable module categories is given by the following rule:
        \[  \mathcal{M}_{d_1,\varepsilon_1} \boxtimes_{\mathcal{C}(\mathfrak{sl}_N, k)} \mathcal{M}_{d_2,\varepsilon_2}   \simeq \mathcal{M}_{d_2,\varepsilon_2} \boxtimes_{\mathcal{C}(\mathfrak{sl}_N, k)} \mathcal{M}_{d_1,\varepsilon_1} \simeq \mathcal{M}_{d,\varepsilon_1\varepsilon_2} ^{\boxplus \gcd(m_{d_1}, m_{d_2}  )}    \]
        where $d = d_{\operatorname{lcm}  (m_{d_1}, m_{d_2}  ),  \vv{a}(d_1)  \times \vv{a}(d_2)      }$.
\end{thm}
The definitions of the quantities $d_{m, \vv{a}}$, $m_d$, and $\vv{a}(d)$ can be found in Definition~\ref{def:ll} and Remark~\ref{rmk:bij}.

\begin{rmk}
    The above Theorem actually holds more generally than just for $N,k \geq 4$, with this restriction simply chosen to provide the cleanest possible statement. The bijection between pairs $(d,\varepsilon)$ and indecomposable module caregories holds for all $N,k \geq 3$ satisfying the \textit{\'etale} algebra condition. For either $N=2$ or $k=2$ with $(N,k) \notin \{(2,16), (16,2)\}$ the bijection holds if we restrict to $\varepsilon = +$. For $N=2$ and $k=2$ the bijection holds if we restrict to $d=1$ and $\varepsilon = +$. The $N=2$ version of the classification theorem is known from previous results in the literature \cite{ADE, OstMod}. The relative tensor product rules hold for all values except for $(N,k) \in \{(3,3),(3,6),(6,3)\}$. We expect that the formula is still true in these three cases. It is also true that the relative tensor product rules between the module categories $\mathcal{M}_{d,\varepsilon}$ hold even when the category has non-pointed \textit{\'etale} algebra objects (that is, also in the cases where there are exceptional module categories beyond the $\mathcal{M}_{d,\varepsilon}$) in all cases apart from $(N,k) \in \{(3,3),(3,6),(4,4), (5,5),(6,3)\}$. We have not included the proof of this fact in this paper.
\end{rmk}

Our classification of indecomposable module categories is achieved via the classification of central Lagrangian algebras. While we have a full understanding of the combinatorial structure of these Lagrangian algebras (i.e. the corresponding modular invariant), we do not have a sufficient description of the algebra map required to reconstruct the corresponding indecomposable module category. This makes our classification result somewhat non-constructive.

We are able to give explicit constructions of some of our module categories via alternate means. The details of this can be found in Section~\ref{sec:cons}, though it is worth highlighting the results here. The categories $\mathcal{M}_{d,+}$ are the \textit{simple current} module categories over $\mathcal{C}(\mathfrak{sl}_N, k)$. These are the categories corresponding to pointed algebra objects in $\mathcal{C}(\mathfrak{sl}_N, k)$ (which are classified by divisors of $N$ or $\frac{N}{2}$). These simple current module categories are fully understood. The module category $\mathcal{M}_{1,-}$ is typically called the \textit{charge conjugation} module category. From the general theory, we know that this module category is invertible, and has rank the number of self-dual objects in $\mathcal{C}(\mathfrak{sl}_N, k)$. There has been some work on determining the explicit structure of this module category (see \cite{gannonCharge,pet, beh}). We understand that Wenzl is currently working on constructing the module category as a quantum deformation of the subgroups
\[   Sp(N-1) \subset SL(N)\quad  \text{: $N$ odd }\qquad  Sp(N) \subset SL(N)  \quad  \text{: $N$ even.}   \]
At the time of writing, Wenzl has a construction of a projective version in the $N$ even case \cite{HansMod}, which we expect corresponds to the module category $\mathcal{M}_{\frac{N}{2},-}$. We expect that all module categories $\mathcal{M}_{d,-}$ with $d$ a divisor of $\frac{N}{2}$ can be constructed as a projective version of $\mathcal{M}_{1,-}$. In an alternate direction, one could also use the theory developed in \cite{dan} to extend the quantum $\mathfrak{gl}_N$ module categories constructed in \cite{beh}, up to $\mathcal{C}(\mathfrak{sl}_N, k)$ module categories. This latter approach was taken in \cite[Section 6]{dan} to construct the module categories $\mathcal{M}_{1,-},\mathcal{M}_{2,-}$, and $\mathcal{M}_{4,-}$ in the $N=4$ case.

We can interpret Theorem~\ref{thm:bigmain} as saying that the module categories $\mathcal{M}_{d,\varepsilon}$ are the generic module categories over $\mathcal{C}(\mathfrak{sl}_N,k)$. This is due to results of the second author, which show that there are only finitely many $k$ for each $N$ where the category $\mathcal{C}(\mathfrak{sl}_N,k)$ has non-pointed \textit{\'etale} algebra objects. Hence, running over all $k$, there are only finitely many indecomposable module categories over $\mathcal{C}(\mathfrak{sl}_N,k)$ other than the module categories $\mathcal{M}_{d,\varepsilon}$.

In the works \cite{GanSU3, Gannon,newGan} an explicit classification of \textit{\'etale} algebra objects in $\mathcal{C}(\mathfrak{sl}_N, k)$ is given for all $N\leq 7$. In particular, there are several levels $k$ where non-pointed \textit{\'etale} algebra objects exist. For each $3\leq N\leq 7$ there are exactly three of these exceptional levels. Furthermore, we expect that for $N>7$ there will always be only three exceptional levels (that is $k=N-2, N, N+2$).   Using the theory of \cite{triples}, we are able to classify the number of module categories over $\mathcal{C}(\mathfrak{sl}_N, k)$ for these exceptional levels, where the results of Theorem~\ref{thm:bigmain} do not hold.

\begin{thm}\label{thm:spec}
For the following exceptional pairs $(N,k)$ we have the following number of indecomposable module categories over $\mathcal{C}(\mathfrak{sl}_N, k)$:
    \begin{longtable}{c|c}
        $\mathcal{C}$ & $| \operatorname{Mod}(\mathcal{C})|$  \\ \hline \hline
        $\mathcal{C}(\mathfrak{sl}_3, 5)$  & $6$\\
        $\mathcal{C}(\mathfrak{sl}_3, 9)$  & $8$\\
        $\mathcal{C}(\mathfrak{sl}_3, 21)$  & $5$\\ \hline
        $\mathcal{C}(\mathfrak{sl}_4, 4)$  & $7$\\
        $\mathcal{C}(\mathfrak{sl}_4, 6)$  & $8$\\
        $\mathcal{C}(\mathfrak{sl}_4, 8)$  & $9$\\ \hline
        $\mathcal{C}(\mathfrak{sl}_5, 3)$  & $6$\\
        $\mathcal{C}(\mathfrak{sl}_5, 5)$  & $12$\\
        $\mathcal{C}(\mathfrak{sl}_5, 7)$  & $8$\\ \hline
        $\mathcal{C}(\mathfrak{sl}_6, 4)$  & $12$\\
        $\mathcal{C}(\mathfrak{sl}_6, 6)$  & $16$\\
        $\mathcal{C}(\mathfrak{sl}_6, 8)$  & $12$\\ \hline
        $\mathcal{C}(\mathfrak{sl}_7, 5)$  & $8$\\
        $\mathcal{C}(\mathfrak{sl}_7, 7)$  & $10$\\
        $\mathcal{C}(\mathfrak{sl}_7, 9)$  & $8$ \\ \hline
    \end{longtable}
\end{thm}


As in Theorem~\ref{thm:bigmain}, this construction is via central Lagrangian algebra objects. Hence we do not have a constructive classification of the indecomposable module categories. There has been some complementary work on the construction of these module categories. In the case of $\mathfrak{sl}_3$, the 7 exceptional module categories are constructed in \cite{SU3, lance}. For $\mathfrak{sl}_4$, the construction of the 6 exceptional module categories was obtained in \cite{dan}. From the results of \cite{LiuYB, LiuRing} and \cite{pres} the structure of the module categories corresponding to the conformal embeddings (the definition of which can be found in Section~\ref{sec:pre})
\[  \mathcal{V}\left(\mathfrak{sl}_N, N \pm 2    \right)\subset \mathcal{V}\left(\mathfrak{sl}_{ \frac{N(N\pm 2)}{2} },1   \right)\quad \text{and} \quad \mathcal{V}\left(\mathfrak{sl}_N, N   \right)\subset \mathcal{V}\left(\mathfrak{so}_{ N^2-1},1   \right)    \]
can be explicitly determined. This leaves thirteen exceptional module categories which remain to be explicitly constructed. Using the general framework of \cite{dan}, we expect to be able to provide these constructions in future work. 

In the cases where the modular invariants of the module category (i.e. the isomorphism class of the corresponding Lagrangian algebra) are linearly independent over $\mathbb{Z}$, we can determine the relative tensor products of the module categories in these exceptional cases. We have linear independence (by direct computation) in all exceptional cases except $\mathcal{C}(\mathfrak{sl}_3,9)$, $\mathcal{C}(\mathfrak{sl}_4,4)$, $\mathcal{C}(\mathfrak{sl}_5,5)$, and $\mathcal{C}(\mathfrak{sl}_7,7)$.

\begin{ex}
From Theorem~\ref{thm:spec} the category $\mathcal{C}(\mathfrak{sl}_6,6)$ has 16 indecomposable module categories up to equivalence. From the results of Subsection~\ref{sub:6} we have the modular invariant matrices of these module categories, and we verify that they are linearly independent. We label these module categories by 
\[\{ \mathcal{M}_{d,\pm} : d\in \{1,2,3,6\}\}\cup \{\mathcal{M}_i : 9 \leq i \leq 16\}.\] 
The last eight module categories are the exceptional ones. The first four of these eight correspond to the four non-pointed \textit{\'etale} algebra objects in $\mathcal{C}(\mathfrak{sl}_6,6)$, the two after that are the heterotic pairings of the \textit{\'etale} algebra objects corresponding to $\mathcal{M}_9$ and $\mathcal{M}_{11}$, and the final two correspond to type $\textrm{II}$ twists of the module categories $\mathcal{M}_{10}$ and $\mathcal{M}_{12}$. We refer the reader to \cite[Remark 3.10]{triples} for explanations of this terminology. We compute the following relative tensor product rules of these module categories.

\[
\resizebox{\textwidth}{!}{$
\begin{array}{c|ccccccccccccccccc}
\boxtimes_{\mathcal{C}(\mathfrak{sl}_6,6)} &\mathcal{M}_{1,+}&\mathcal{M}_{2,+} & \mathcal{M}_{3,+} &\mathcal{M}_{6,+} & \mathcal{M}_{1,-} & \mathcal{M}_{2,-} &\mathcal{M}_{3,-} &\mathcal{M}_{6,-} & \mathcal{M}_9 & \mathcal{M}_{10} & \mathcal{M}_{11} & \mathcal{M}_{12} & \mathcal{M}_{13} & \mathcal{M}_{14} &\mathcal{M}_{15} & \mathcal{M}_{16}\\ \hline
 \mathcal{M}_{1,+}& \mathcal{M}_{1,+}&\mathcal{M}_{2,+} & \mathcal{M}_{3,+} &\mathcal{M}_{6,+} & \mathcal{M}_{1,-} & \mathcal{M}_{2,-} &\mathcal{M}_{3,-} &\mathcal{M}_{6,-} & \mathcal{M}_9 & \mathcal{M}_{10} & \mathcal{M}_{11} & \mathcal{M}_{12} & \mathcal{M}_{13} & \mathcal{M}_{14} &\mathcal{M}_{15} & \mathcal{M}_{16}\\
\mathcal{M}_{2,+} &\mathcal{M}_{2,+} & \mathcal{M}_{1,+}&\mathcal{M}_{6,+} & \mathcal{M}_{3,+} & \mathcal{M}_{2,-} & \mathcal{M}_{1,-} &\mathcal{M}_{6,-} &\mathcal{M}_{3,-} & \mathcal{M}_9 &\mathcal{M}_{15} & \mathcal{M}_{11} & \mathcal{M}_{16}& \mathcal{M}_{13} & \mathcal{M}_{14} & \mathcal{M}_{10} & \mathcal{M}_{12} \\
 \mathcal{M}_{3,+} &\mathcal{M}_{3,+} &\mathcal{M}_{6,+} & 3 \mathcal{M}_{3,+} & 3\mathcal{M}_{6,+} &\mathcal{M}_{3,-} &\mathcal{M}_{6,-} & 3\mathcal{M}_{3,-} & 3\mathcal{M}_{6,-} & 3 \mathcal{M}_9 & 3 \mathcal{M}_{10} & 3 \mathcal{M}_{11} & 3 \mathcal{M}_{12} & 3 \mathcal{M}_{13} & 3 \mathcal{M}_{14} & 3\mathcal{M}_{15} & 3 \mathcal{M}_{16}\\
\mathcal{M}_{6,+} &\mathcal{M}_{6,+} & \mathcal{M}_{3,+} & 3\mathcal{M}_{6,+} & 3 \mathcal{M}_{3,+} &\mathcal{M}_{6,-} &\mathcal{M}_{3,-} & 3\mathcal{M}_{6,-} & 3\mathcal{M}_{3,-} & 3 \mathcal{M}_9 & 3\mathcal{M}_{15} & 3 \mathcal{M}_{11} & 3 \mathcal{M}_{16}& 3 \mathcal{M}_{13} & 3 \mathcal{M}_{14} & 3 \mathcal{M}_{10} & 3 \mathcal{M}_{12} \\
 \mathcal{M}_{1,-} &\mathcal{M}_{1,-} & \mathcal{M}_{2,-} &\mathcal{M}_{3,-} &\mathcal{M}_{6,-} & \mathcal{M}_{1,+}&\mathcal{M}_{2,+} & \mathcal{M}_{3,+} &\mathcal{M}_{6,+} & \mathcal{M}_9 & \mathcal{M}_{10} & \mathcal{M}_{11} & \mathcal{M}_{16}& \mathcal{M}_{13} & \mathcal{M}_{14} &\mathcal{M}_{15} & \mathcal{M}_{12} \\
 \mathcal{M}_{2,-} &\mathcal{M}_{2,-} & \mathcal{M}_{1,-} &\mathcal{M}_{6,-} &\mathcal{M}_{3,-} &\mathcal{M}_{2,+} & \mathcal{M}_{1,+}&\mathcal{M}_{6,+} & \mathcal{M}_{3,+} & \mathcal{M}_9 &\mathcal{M}_{15} & \mathcal{M}_{11} & \mathcal{M}_{12} & \mathcal{M}_{13} & \mathcal{M}_{14} & \mathcal{M}_{10} & \mathcal{M}_{16}\\
 \mathcal{M}_{3,-} &\mathcal{M}_{3,-} &\mathcal{M}_{6,-} & 3\mathcal{M}_{3,-} & 3\mathcal{M}_{6,-} & \mathcal{M}_{3,+} &\mathcal{M}_{6,+} & 3 \mathcal{M}_{3,+} & 3\mathcal{M}_{6,+} & 3 \mathcal{M}_9 & 3 \mathcal{M}_{10} & 3 \mathcal{M}_{11} & 3 \mathcal{M}_{16}& 3 \mathcal{M}_{13} & 3 \mathcal{M}_{14} & 3\mathcal{M}_{15} & 3 \mathcal{M}_{12} \\
\mathcal{M}_{6,-} &\mathcal{M}_{6,-} &\mathcal{M}_{3,-} & 3\mathcal{M}_{6,-} & 3\mathcal{M}_{3,-} &\mathcal{M}_{6,+} & \mathcal{M}_{3,+} & 3\mathcal{M}_{6,+} & 3 \mathcal{M}_{3,+} & 3 \mathcal{M}_9 & 3\mathcal{M}_{15} & 3 \mathcal{M}_{11} & 3 \mathcal{M}_{12} & 3 \mathcal{M}_{13} & 3 \mathcal{M}_{14} & 3 \mathcal{M}_{10} & 3 \mathcal{M}_{16}\\
\mathcal{M}_9 & \mathcal{M}_9 & \mathcal{M}_9 & 3 \mathcal{M}_9 & 3 \mathcal{M}_9 & \mathcal{M}_9 & \mathcal{M}_9 & 3 \mathcal{M}_9 & 3 \mathcal{M}_9 & 16 \mathcal{M}_9 & 4 \mathcal{M}_{14} & 8 \mathcal{M}_{14} & 4 \mathcal{M}_{14} & 8 \mathcal{M}_9 & 16 \mathcal{M}_{14} & 4 \mathcal{M}_{14} & 4 \mathcal{M}_{14} \\
 \mathcal{M}_{10} & \mathcal{M}_{10} &\mathcal{M}_{15} & 3 \mathcal{M}_{10} & 3\mathcal{M}_{15} & \mathcal{M}_{10} &\mathcal{M}_{15} & 3 \mathcal{M}_{10} & 3\mathcal{M}_{15} & 4 \mathcal{M}_{13} & 6 \mathcal{M}_{10}\boxplus \mathcal{M}_{11}  & 8 \mathcal{M}_{11} & 4 \mathcal{M}_{11} & 8 \mathcal{M}_{13} & 4 \mathcal{M}_{11} &  \mathcal{M}_{11} \boxplus 6\mathcal{M}_{15} & 4 \mathcal{M}_{11} \\
 \mathcal{M}_{11} &\mathcal{M}_{11} & \mathcal{M}_{11} & 3 \mathcal{M}_{11} & 3 \mathcal{M}_{11} & \mathcal{M}_{11} & \mathcal{M}_{11} & 3 \mathcal{M}_{11} & 3 \mathcal{M}_{11} & 8 \mathcal{M}_{13} & 8 \mathcal{M}_{11} & 16 \mathcal{M}_{11} & 8 \mathcal{M}_{11} & 16 \mathcal{M}_{13} & 8 \mathcal{M}_{11} & 8 \mathcal{M}_{11} & 8 \mathcal{M}_{11} \\
 \mathcal{M}_{12} &\mathcal{M}_{12} & \mathcal{M}_{16}& 3 \mathcal{M}_{12} & 3 \mathcal{M}_{16}& \mathcal{M}_{16}& \mathcal{M}_{12} & 3 \mathcal{M}_{16}& 3 \mathcal{M}_{12} & 4 \mathcal{M}_{13} & 4 \mathcal{M}_{11} & 8 \mathcal{M}_{11} & \mathcal{M}_{11} \boxplus 6\mathcal{M}_{12}& 8 \mathcal{M}_{13} & 4 \mathcal{M}_{11} & 4 \mathcal{M}_{11} & \mathcal{M}_{11}\boxplus 6\mathcal{M}_{16} \\
 \mathcal{M}_{13} &\mathcal{M}_{13} & \mathcal{M}_{13} & 3 \mathcal{M}_{13} & 3 \mathcal{M}_{13} & \mathcal{M}_{13} & \mathcal{M}_{13} & 3 \mathcal{M}_{13} & 3 \mathcal{M}_{13} & 16 \mathcal{M}_{13} & 4 \mathcal{M}_{11} & 8 \mathcal{M}_{11} & 4 \mathcal{M}_{11} & 8 \mathcal{M}_{13} & 16 \mathcal{M}_{11} & 4 \mathcal{M}_{11} & 4 \mathcal{M}_{11} \\
 \mathcal{M}_{14} & \mathcal{M}_{14} & \mathcal{M}_{14} & 3 \mathcal{M}_{14} & 3 \mathcal{M}_{14} & \mathcal{M}_{14} & \mathcal{M}_{14} & 3 \mathcal{M}_{14} & 3 \mathcal{M}_{14} & 8 \mathcal{M}_9 & 8 \mathcal{M}_{14} & 16 \mathcal{M}_{14} & 8 \mathcal{M}_{14} & 16 \mathcal{M}_9 & 8 \mathcal{M}_{14} & 8 \mathcal{M}_{14} & 8 \mathcal{M}_{14} \\
 \mathcal{M}_{15} &\mathcal{M}_{15} & \mathcal{M}_{10} & 3\mathcal{M}_{15} & 3 \mathcal{M}_{10} &\mathcal{M}_{15} & \mathcal{M}_{10} & 3\mathcal{M}_{15} & 3 \mathcal{M}_{10} & 4 \mathcal{M}_{13} & \mathcal{M}_{11} \boxplus 6\mathcal{M}_{15} & 8 \mathcal{M}_{11} & 4 \mathcal{M}_{11} & 8 \mathcal{M}_{13} & 4 \mathcal{M}_{11} & 6\mathcal{M}_{10} \boxplus \mathcal{M}_{11} & 4 \mathcal{M}_{11} \\
 \mathcal{M}_{16}&\mathcal{M}_{16}& \mathcal{M}_{12} & 3 \mathcal{M}_{16}& 3 \mathcal{M}_{12} & \mathcal{M}_{12} & \mathcal{M}_{16}& 3 \mathcal{M}_{12} & 3 \mathcal{M}_{16}& 4 \mathcal{M}_{13} & 4 \mathcal{M}_{11} & 8 \mathcal{M}_{11} & \mathcal{M}_{11}\boxplus 6\mathcal{M}_{16} & 8 \mathcal{M}_{13} & 4 \mathcal{M}_{11} & 4 \mathcal{M}_{11} &\mathcal{M}_{11}\boxplus 6\mathcal{M}_{12} \\
\end{array}$
}
\]

\end{ex}

The paper is outlined as follows.

In Section~\ref{sec:pre} we introduce the required background on tensor categories required for this paper. In particular, we present the result \cite[Corollary 3.8]{triples} which is our key tool for proving Theorem~\ref{thm:bigmain} and Theorem~\ref{thm:spec}. We also summarise the key results of \cite{Gannon,newGan} and \cite{ModPt1} which are the base point for this paper.

In Section~\ref{sec:gen} we determine the number of indecomposable module categories over $\mathcal{C}(\mathfrak{sl}_N,k)$, for generic $k$. Using the classification results of \cite{ModPt1}, we are quickly able to apply the theory of \cite{triples} to obtain a non-closed formula for the rank of $   \operatorname{Mod}( \mathcal{C}(\mathfrak{sl}_N, k)   ) $ in the generic case. We finish this section by providing a surprisingly simple closed form expression for the rank of $   \operatorname{Mod}( \mathcal{C}(\mathfrak{sl}_N, k)   ) $.

In Section~\ref{sec:cons} we find enough indecomposable $\mathcal{C}(\mathfrak{sl}_N, k)$-module categories to realise the number of module categories abstractly classified in Section~\ref{sec:gen}. We determine the modular invariants of these module categories. Using these modular invariants, we deduce that the module categories we have obtained are distinct, and hence give a complete set of representatives. Furthermore, we show that in the generic setting these modular invariants are linearly independent. This allows us to determine the relative tensor product rules for these module categories in the generic setting.

In Section~\ref{sec:spec} we give the proof of Theorem~\ref{thm:spec}. As this theorem deals with the non-generic cases, we have to carefully consider the \textit{\'etale} algebra objects in each of the categories $\mathcal{C}(\mathfrak{sl}_N, k)$ in order to apply the theory of \cite[Corollary 3.8]{triples}. The classification of these \textit{\'etale} algebra objects for the relevant categories was obtained in \cite{Gannon,newGan}. In a case by case analysis we are able to provide the proof of Theorem~\ref{thm:spec}. The main technical difficulty we encounter in applying the theory of \cite[Corollary 3.8]{triples} is determining the image of $\operatorname{Aut}(A)$ in $\operatorname{EqBr}(\mathcal{C}(\mathfrak{sl}_N, k)_A^0)$ for the exceptional non-pointed \textit{\'etale} algebra objects $A$. We develop several tools for resolving this difficulty.

We finish the paper with Appendix~\ref{app:fix}, which contains a minor erratum to a lemma in the prequel to this paper \cite{ModPt1}. As we require the correct statement of the lemma in Section~\ref{sec:cons}, we include this erratum here.

\subsection*{Acknowledgements}
CE would like to thank Dmitri Nikshych and  Hans Wenzl for many useful conversations. CE was supported by NSF grants NSF DMS 2245935 and DMS 2400089. The research of TG was supported in part by NSERC. Both authors thank MSRI/SLMath, where part of this paper was written, for a great work environment.



\section{Preliminaries}\label{sec:pre}
We refer the reader to \cite{book} for the basics on tensor categories.

\subsection{  \textit{\'Etale} algebra objects}

In this subsection, we review the relevant background material on \textit{\'etale} algebra objects in braided tensor categories. We also prove some general results in the case that the algebra is pointed. Much of the following can be found in \cite[Section 3]{LagrangeUkraine}.

\begin{defn}
    Let $\mathcal{C}$ be a braided tensor category. We say an algebra object $A\in \mathcal{C}$ is \textit{\'etale} if it is both commutative and separable. 
\end{defn}
As the braided tensor categories we consider in this paper have pivotal structures, we must consider the twists of the \textit{\'etale} algebra objects in these categories.

Given an algebra object, we can construct the category of right $A$-module objects. 

\begin{defn}
Let $A\in \mathcal{C}$ be an algebra object. We write $\mathcal{C}_A$ for the category of right $A$-module objects internal to $\mathcal{C}$. 
\end{defn}
When $A$ is \textit{\'etale}, the category $\mathcal{C}_A$ has the structure of a semisimple tensor category via the following tensor product.

\begin{defn}
   Let $A\in \mathcal{C}$ be an \textit{\'etale} algebra object, and $M_1,M_2\in \mathcal{C}_A$. We define $M_1 \otimes_A M_2$ as the image of the projection:
   \[ \att{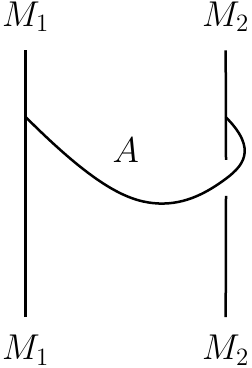}{.4}  \]
\end{defn}

For an \textit{\'etale} algebra object $A$, there is a distinguished subcategory of $\mathcal{C}_A$ consisting of local modules.

\begin{defn}
    Let $A\in \mathcal{C}$ be an \textit{\'etale} algebra object, and $M\in \mathcal{C}_A$. We say $M$ is a local module if
    \[  \att{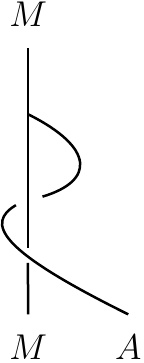}{.4} \quad=\quad \att{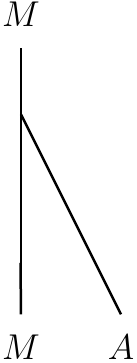}{.4}. \]
    We write $\mathcal{C}_A^0$ for the subcategory of local modules in $\mathcal{C}_A$.
\end{defn}

It is shown in \cite[Corollary 3.30]{LagrangeUkraine} that $\mathcal{C}_A^0$ is a braided tensor category.

An important invariant of an algebra object is its group of automorphisms. These are the invertible algebra maps from the algebra to itself. 
\begin{defn}
    Let $A,B\in \mathcal{C}$ be an algebra objects. A map of algebras is a morphism $f\in \operatorname{Hom}_\mathcal{C}(A\to B)$ satisfying
    \[    \att{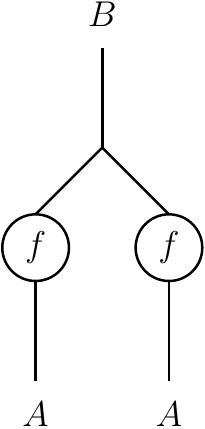}{.3} \quad = \quad \att{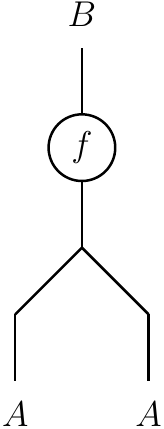}{.3}  .   \]
    We will write $\operatorname{Aut}(A)$ for the group of algebra automorphisms of $A$.
\end{defn}

Given an algebra isomorphism $f: A\to B$, we naturally obtain a monoidal equivalence $\mathcal{C}_B\to \mathcal{C}_A$. The details of this equivalence are important for this paper, so we include them below.

\begin{lem}\label{lem:induce}
Let $\mathcal{C}$ be a braided (pivotal) tensor category, and $A,B\in \mathcal{C}$ \textit{\'etale} algebra objects (with ribbon twists $\theta_A = 1 = \theta_B$), and $f:A\to B$ an algebra isomorphism. Then there is a (pivotal) monoidal equivalence $\operatorname{Ind}(f): \mathcal{C}_B \to \mathcal{C}_A$. This equivalence is defined on objects by 
\[   (X, r_X) \mapsto (X, r_X\circ(\operatorname{id}_X\otimes \phi))  .  \]
Further, we have that $\operatorname{Ind}(\phi)$ restricts to a braided (pivotal) equivalence $\mathcal{C}^0_B\to \mathcal{C}^0_A$. 
\end{lem}
\begin{proof}
    This is a direct computation.
\end{proof}

In particular, the image $\operatorname{Ind}(\operatorname{Aut}(A))$ gives a distinguished subgroup of $\operatorname{EqBr}(\mathcal{C}_A^0)$.

\begin{defn}\label{def:act}
  Let $\mathcal{C}$ be a braided tensor category, and $A\in \mathcal{C}$ an \textit{\'etale} algebra object. We will write
  \[         \operatorname{Aut}(A) /   \operatorname{EqBr}\left(\mathcal{C}_A^0\right) \backslash   \operatorname{Aut}(A)     \]
  for the set of double cosets with respect to the subgroup $\operatorname{Ind}\left( \operatorname{Aut}(A) \right)$.
\end{defn}

For many of the \textit{\'etale} algebra objects we study in this paper, we do not have direct knowledge of their automorphism groups $\operatorname{Aut}(A)$, but we do know the structure of $\operatorname{EqBr}(\mathcal{C}_A^0)$. The following observation gives us an indirect way to detect if an autoequivalence of $\mathcal{C}_A^0$ is in the image of $\operatorname{Ind}$.

\begin{cor}\label{cor:ind}
Let $A\in \mathcal{C}$ an \textit{\'etale} algebra object, and $\operatorname{For}_A: \mathcal{C}_A \to \mathcal{C}$ the forgetful functor, and $\phi \in \operatorname{Aut}(A)$. Then
\[    \operatorname{For}_A\circ \operatorname{Ind}(\phi)(M) \cong \operatorname{For}_A(M) \]
for all $M\in \mathcal{C}_A$.
\end{cor}
\begin{proof}
This is immediate from the definition of $\operatorname{Ind}(\phi)$ given in Lemma~\ref{lem:induce}.
\end{proof}
Note that this statement is not an exact characterisation of braided autoequivalences of $\mathcal{C}_A^0$ in the image of $\operatorname{Ind}$. Later in this paper we will encounter braided autoequivalences $\mathcal{F}$ of $\mathcal{C}_A^0$ satisfying the condition in Corollary~\ref{cor:ind}, but which are not in the image of $\operatorname{Ind}$. 

A special class of simple objects in a tensor category $\mathcal{C}$ are the invertible objects. These objects play a key role in this paper.
\begin{defn}
    Let $\mathcal{C}$ be a tensor category. We say a simple object $g\in \mathcal{C}$ is invertible if there exists an object $g^{-1}\in \mathcal{C}$ such that $g\otimes g^{-1} \cong \mathbf{1}$. We will write $\operatorname{Inv}(\mathcal{C})$ for the group of isomorphism classes of invertible objects in $\mathcal{C}$. The group operation is given by tensor product. 
\end{defn}

In the special case that the \textit{\'etale} algebra $A$ is \textit{pointed cyclic}, we can describe the category $\mathcal{C}_A$ in explicit detail.

\begin{defn}
    We say an algebra $A\in \mathcal{C}$ is pointed cyclic if it is of the form
    \[     A \cong \bigoplus_{0 \leq i <m} g^{\otimes i}   \]
for some invertible object $g\in \mathcal{C}$.
\end{defn}

Given a pointed cyclic algebra, we have that the set $\{g^{\otimes i} \mid 0 \leq i < m \}$ forms a subgroup of $\operatorname{Inv}(\mathcal{C})$ isomorphic to $\mathbb{Z}_m$. In fact this set generates a monoidal subcategory of $\mathcal{C}$ equivalent to $\operatorname{Vec}(\mathbb{Z}_m)$. In particular, a collection of morphisms 
\[      \att{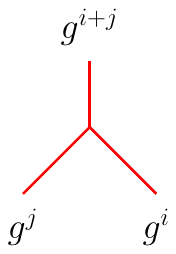}{.4} : g^{j} \otimes g^i \to g^{j+i}  \]
can be chosen with trivial 6-j symbols.

For $X \in \mathcal{C}$ we define $[X]$ to be the orbit of $X$ under the tensor action by $\mathbb{Z}_m$, and $\operatorname{Stab}(X)$ to be the subset of $\{g^{\otimes i} \mid 0 \leq i < m \}$ which fixes the isomorphism class of $X$. Note that $\operatorname{Stab}(X) \cong   \mathbb{Z}_{\frac{m}{|[X]| }}$. Finally, we let
\[  r_X := \att{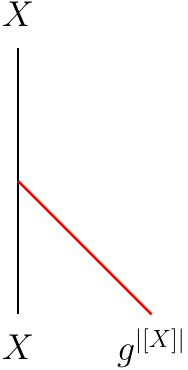}{.4} : X\otimes g^{|[X]|}\to X   \]
be an isomorphism, normalised so that
\[       \att{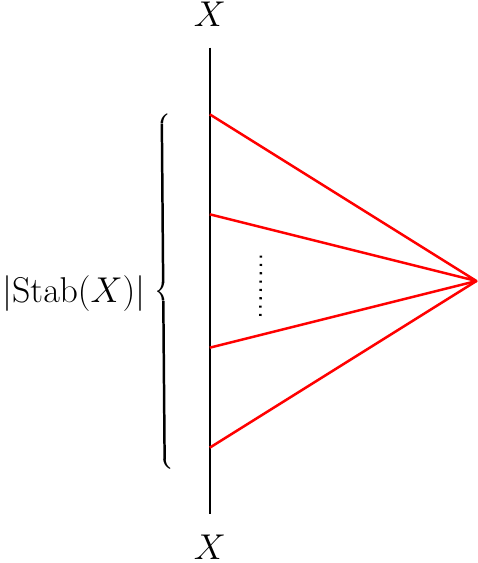}{.4} = \att{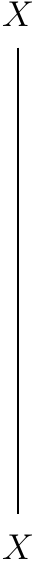}{.4}   \]

The following result is well known. See for example \cite[Corollary 5.3]{muger}. However, we are not aware of the below proposition appearing in its exact form in the literature. Hence we give a proof for completeness sake.
\begin{prop}\label{prop:12}\cite{muger}
    Let $\mathcal{C}$ be a braided tensor category, and 
    \[A\cong \bigoplus_{0\leq i < m} g^{\otimes i}\in \mathcal{C}\] a pointed cyclic \textit{\'etale} algebra. Then every simple $A$-module in $\mathcal{C}_A$ corresponds to a pair $(X, \chi)$ where $X$ is a simple object of $\mathcal{C}$, and $\chi$ is an element of $\widehat{\mathbb{Z}_m}$. The simple $A$-module corresponding to $(X,\chi)$ is $( M_{X}, r^\chi)$ where \[ M_{X} := \bigoplus_{j \in \mathbb{Z}_{|[X]|}} X\otimes g^{j} \]
    and $r^\chi: M_{X}\otimes A \to M_{X}$ is defined on simple components by
    \[r^\chi_{j,i} :=\chi(i) \att{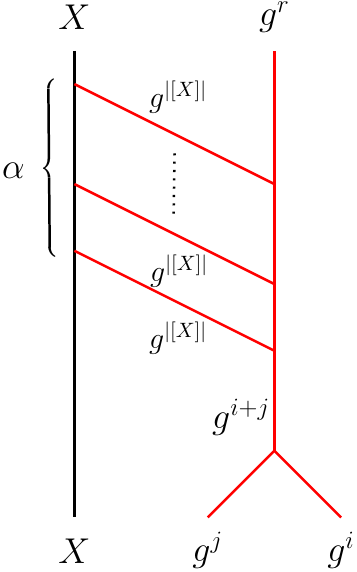}{.4}  : (X\otimes g^j) \otimes g^i \to g^r \]
    where $\alpha $ and $r$ are such that $i+j = \alpha |[X]| + r$ with $r \in \mathbb{Z}_{|[X]|}$. For two pairs $(X_1, \chi_1)$, $(X_2, \chi_2)$ we have that $( M_{X_1}, r^{\chi_1})\cong ( M_{X_2}, r^{\chi_2})$ as $A$-modules if and only if $[X_1]= [X_2]$, and $\chi_1 = \mu \cdot \chi_2$ for some $\mu\in  \widehat{\mathbb{Z}_m}$ such that $\mu|_{\operatorname{Stab}(X_1)} = \operatorname{id}$.
\end{prop}
\begin{proof}
    It is routine to verify that $( M_{X}, r^{\chi})$ is a right $A$-module, and that it is simple. 
    
    Let $( M_{X_1}, r^{\chi_1})$, and $( M_{X_2}, r^{\chi_2})$ be two such $A$-modules. Clearly if $[X_1] \neq [X_2]$, then there is no non-zero map between $M_{X_1}$ and $M_{X_2}$ in $\mathcal{C}$, and so $( M_{X_1}, r^{\chi_1}) \not\cong ( M_{X_2}, r^{\chi_2})$ as $A$-modules. If $[X_1] = [X_2]$, then we have that $M_{X_1} \cong M_{X_2}$ as objects in $\mathcal{C}$. Let $\mu: ( M_{X_1}, r^{\chi_1}) \to ( M_{X_1}, r^{\chi_2})$ be an $A$-module isomorphism. It follows that $\mu_{X_1\otimes g^{\otimes j}} = \mu_{j} \operatorname{id}_{X_1\otimes g^{\otimes j}}$ for complex scalars $\mu_j$. We rescale the isomorphism $\mu$ so that $\mu_0 = 1$. We then have by the definition of an $A$-module isomorphism that 
    \[     \chi_1(i) \mu_j  = \chi_2(i) \mu_{i+j} \]
    for all $i \in \mathbb{Z}_m$, and for all $j \in \mathbb{Z}_{|[X_1]|}$, where the $i+j$ is taken mod $|[X_1]|$. Taking $j = 0$ shows that $\mu = \frac{\chi_1}{\chi_2}\in \widehat{\mathbb{Z}_m}$, which implies that $\chi_1|_{\operatorname{Stab}(X_1)}=\chi_2|_{\operatorname{Stab}(X_1)}$.

    To show that $( M_{X}, r^{\chi})$ exhaust all simple $A$-modules, we observe that there is a natural $A$-module map $( M_{X}, r^{\chi})\to \mathcal{F}_A(X)$ given by inclusion. By the previous discussion, this gives $|\widehat{\mathbb{Z}_m} / \{ \mu : \mu|_{\operatorname{Stab}(X)} = \operatorname{id}  \}| = |    \operatorname{Stab}(X) |$ non-isomorphic summands of $\mathcal{F}_A(X)$. On the other hand, we have
    \[   \operatorname{dimEnd}_{\mathcal{C}_A}( \mathcal{F}_A(X)) = \operatorname{dimHom}_{\mathcal{C}}( A\otimes X \to X) = |    \operatorname{Stab}(X) |.   \]
    We thus have
    \[      \mathcal{F}_A(X) \cong \bigoplus_{\chi\in \widehat{\mathbb{Z}_m} / \{ \mu : \mu|_{\operatorname{Stab}(X)} = \operatorname{id}  \}}  ( M_{X}, r^{\chi}) .  \]
    As the functor $\mathcal{F}_A$ is dominant, it follows that every simple $A$-module is isomorphic to some $ ( M_{X}, r^{\chi})$.
\end{proof}

\begin{rmk}
    It is important to note that the classification of simples of $\mathcal{C}_A$ in the above Proposition is different than in the prequel \cite{ModPt1}. In the prequel, the simples were classified in terms of characters of the group $\mathbb{Z}_{\operatorname{Stab}(X)}$, while in this paper they are classified in terms of characters of $\mathbb{Z}_m$ modulo the characters which are trivial on the $\operatorname{Stab}(X)$ subgroup. These groups are isomorphic via the restriction map $\mathbb{Z}_m\to \mathbb{Z}_{\operatorname{Stab}(X)}$. The reason for this difference is that in the prequel, it was most natural to realise $\mathcal{C}_A$ as the Cauchy completion of a generators and relations category, while in this paper it is most natural to work with $\mathcal{C}_A$ directly.
\end{rmk}

We will be interested in the automorphism groups of pointed cyclic algebras in $\mathcal{C}$. As this algebra lives in a pointed subcategory of $\mathcal{C}$, the automorphism group can be explicitly computed.

\begin{prop}\label{prop:13}
    Let $\mathcal{C}$ be a modular tensor category, and $A\cong \bigoplus_{0\leq i < m} g^{\otimes i}\in \mathcal{C}$ a pointed cyclic \textit{\'etale} algebra. Then $\operatorname{Aut}_{\mathcal{C}}(A)\cong \widehat{\mathbb{Z}_m}$ with the isomorphism sending $\phi \in \widehat{\mathbb{Z}_m}$ to the automorphism $\eta^\phi$ defined on components by 
    \[    \eta^\phi_{g^j}:= \phi(j) \operatorname{id}_{g^j}: g^j \to g^j.       \]
\end{prop}
\begin{proof}
    As $A\cong \bigoplus_{0\leq i < m} g^{\otimes i}$ is an algebra, it follows that the pointed subcategory of $\mathcal{C}$ generated by $g$ is monoidally equivalent to $\operatorname{Vec}(\mathbb{Z}_m)$, and that the algebra $A$ lives in this subcategory. The result then follows by direct computation.
\end{proof}

With this explicit description of both $\mathcal{C}_A$ and $\operatorname{Aut}(A)$ in the cyclic pointed case, we can explicitly determine the map $\operatorname{Ind} : \operatorname{Aut}(A) \to \operatorname{Eq}(\mathcal{C}_A)$ from Lemma~\ref{lem:induce} in this setting.

\begin{prop}\label{prop:imgpt}
     Let $\mathcal{C}$ be a modular tensor category, and $A\cong \bigoplus_{0\leq i < m} g^{\otimes i}\in \mathcal{C}$ a pointed cyclic \textit{\'etale} algebra. Then for $\ell \in \widehat{\mathbb{Z}_m}$, and $(X,\chi)$ a simple object of $\mathcal{C}_A$, we have
     \[ \operatorname{Ind}(\eta^\ell)( X, \chi) = \left(X, \ell \cdot \chi\right).      \]
\end{prop}
\begin{proof}
    This follows directly from the classification of simple objects of $\mathcal{C}_A$ from Proposition~\ref{prop:12}, the explicit description of $\operatorname{Aut}(A)$ from Proposition~\ref{prop:13}, and the definition of the map $\operatorname{Ind}$ from Lemma~\ref{lem:induce}.
\end{proof}

\subsection{Indecomposable module categories from \textit{\'etale} algebra objects}\label{sub:modinv}

Given an \textit{\'etale} algebra object in a braided tensor category $\mathcal{C}$, the category $\mathcal{C}_A$ is an indecomposable module category over $\mathcal{C}$. However, not all indecomposable $\mathcal{C}$ module categories are of this form. Instead, one has to consider \textit{Lagrangian} algebras in the Drinfeld centre of $\mathcal{C}$ to obtain the full picture.

\begin{defn}\cite[Definition 4.6]{LagrangeUkraine}
    Let $\mathcal{C}$ be a modular tensor category. An algebra $A$ is called \textit{Lagrangian} if it is connected \textit{\'etale}, and if $ \operatorname{ FPDim}(A)^2 = \operatorname{ FPDim}(\mathcal{C})$.
\end{defn}

We then have the following bijection. It should be noted that for most examples of modular tensor categories, this bijection is for all intents and purposes non-constructive.

\begin{thm}\cite[Proposition 4.8]{LagrangeUkraine}
Let $\mathcal{C}$ be a modular tensor category. Then there is a bijection between indecomposable $\mathcal{C}$ module categories, and Lagrangian algebras in $\mathcal{Z}(\mathcal{C})$. This bijection sends
\[   \mathcal{M} \mapsto \mathcal{Z}(\mathcal{M})          \]
where $\mathcal{Z}(\mathcal{M})\in \mathcal{Z}(\mathcal{C})$ is the full centre of $\mathcal{M}$ \cite{LagrangeGerman,davy}.
\end{thm}

From the full centre of a module category $\mathcal{M}$, we can define the modular invariant of $\mathcal{M}$. The modular invariant encodes the combinatorial data of the full centre.

\begin{defn}
    Let $\mathcal{C}$ be a modular tensor category, and $\mathcal{M}$ a $\mathcal{C}$-module category. We define the \textit{modular invariant} of $\mathcal{M}$ as the $|\operatorname{Irr}(\mathcal{C})|\times|\operatorname{Irr}(\mathcal{C})|$ square matrix: 
    \[    [  \mathcal{Z}(\mathcal{M})    ]_{X,Y} := \operatorname{dimHom}_{ \mathcal{C}\boxtimes\mathcal{C}^\text{rev}    }(X\boxtimes Y \to      \mathcal{Z}(\mathcal{M}) ) .  \]
\end{defn}

In any braided tensor category $\mathcal{C}$, one can take the relative tensor product of two $\mathcal{C}$-module categories using the braiding. Details can be found in \cite[Section 3.4]{homo}. We will write $\mathcal{M}_1 \boxtimes_\mathcal{C} \mathcal{M}_2$ for this relative tensor product. The relative tensor product is rather difficult to compute in practice.

The modular invariant matrix $[ \mathcal{Z}(\mathcal{M})]$ can be thought of as a kind of ``character'' of the module category $\mathcal{M}$. In particular, this matrix plays nice with relative tensor product and direct sums of $\mathcal{C}$-module categories. Indeed from \cite[Proposition 5.3]{fuch} we have
\begin{align*}
    [Z(\mathcal{M}_1 \boxplus \mathcal{M}_1    )]  &=[Z(\mathcal{M}_1)] + [Z(\mathcal{M}_2)]\\
    [Z(\mathcal{M}_1 \boxtimes_{\mathcal{C}}\mathcal{M}_1    )]  &=[Z(\mathcal{M}_1)] \cdot  [Z(\mathcal{M}_2)],
\end{align*}
where $\cdot$ represents matrix multiplication. In the case where the modular invariant matrices are linearly independent over $\mathbb{Z}$, we can use the above formulae to find the relative tensor product of module categories over $\mathcal{C}$ in the same manner as one uses characters of Lie algebra representations to determine their tensor product decompositions. It should be pointed out that it is possible for distinct module categories to have the same modular invariant. Several examples will be seen later in this paper. Hence the modular invariants do not provide a universal method of deducing relative tensor product rules.

While it may seem difficult to classify Lagrangian algebras in $\mathcal{Z}(\mathcal{C})$ at first glance, it turns out this can be done explicitly by studying the \textit{\'etale} algebras in $\mathcal{C}$, and the braided equivalences between their categories of local modules.

\begin{thm}\label{thm:nik}\cite[Corollary 3.8]{triples}
Let $\mathcal{C}$ be a modular tensor category. There is a bijection between Lagrangian algebras in $\mathcal{Z}(\mathcal{C})$, and triples $(A_1, \mathcal{F}, A_2)$, where $A_1,A_2\in \mathcal{C}$ are \textit{\'etale} algebra objects, and $\mathcal{F}$ is a braided equivalence $\mathcal{C}^0_{A_1} \to \mathcal{C}^0_{A_2}$. Two triples $(A_1, \mathcal{F}_A, A_2)$ and $(B_1, \mathcal{F}_B, B_2)$ are considered equivalent if there exist algebra isomorphisms $\psi_1:A_1 \to B_1$ and $\psi_2:A_2 \to B_2$ such that the following diagram commutes:
\[\begin{tikzcd}
\mathcal{C}^0_{A_1}  \arrow[r, "\mathcal{F}_A"]   & \mathcal{C}^0_{A_2}  \\
    \mathcal{C}^0_{B_1} \arrow[u, "\operatorname{Ind}(\psi_1)"] \arrow[r, "\mathcal{F}_B"]&  \mathcal{C}^0_{B_2}\arrow[u,swap, "\operatorname{Ind}(\psi_2)"]
\end{tikzcd}\]
where $\operatorname{Ind}(\psi_i)$ are the braided equivalences $\mathcal{C}^0_{B_i} \to \mathcal{C}^0_{A_i}$ induced from the algebra isomorphisms $\psi_i: A_i \to B_i$. Explicitly, the Lagrangian algebra in $\mathcal{Z}(\mathcal{C})$ corresponding to the triple $(A_1, \mathcal{F}, A_2)$ is given by 
 \[\left(\operatorname{For}_{A_1}\boxtimes (\mathcal{F}\circ \operatorname{For}_{A_2})\right)\left(I\left(\mathbf{1}_{\mathcal{C}_{A_1}^0}\right)\right) \in \mathcal{C}\boxtimes \mathcal{C}^{rev} \simeq \mathcal{Z}(\mathcal{C}).\]
 Here $I: \mathcal{C}_{A_1}^0 \to  \mathcal{C}_{A_1}^0\boxtimes \left( \mathcal{C}_{A_1}^0\right)^{rev}$ is the induction functor, and $\operatorname{For}_{A_i}: \mathcal{C}_{A_i}^0 \to \mathcal{C}$ are the forgetful functors.
\end{thm}

\begin{rmk}\label{rmk:same}
In the case that $A_1 = A_2 = A$, we have that two triples $(A, \mathcal{F}_1, A)$ and $(A, \mathcal{F}_2, A)$ are equivalent if and only if $\mathcal{F}_1$ and $\mathcal{F}_2$ lie in the same double coset given in Definition~\ref{def:act}.
\end{rmk}
\begin{rmk}
    The above theorem allows the explicit computation of the modular invariant of the module category corresponding to the triple $(A_1, \mathcal{F}, A_2)$. Let $M_1$ and $M_2$ be the matrix of branching rules for the \textit{\'etale} algebra objects $A_1$ and $A_2$, and $F$ the matrix representing the combinatorics of the equivalence $\mathcal{F}$. The modular invariant of the module category is then given by 
    \[    M_1 \cdot F \cdot M_2^T.     \]
\end{rmk}

\subsection{The categories $\mathcal{C}(\mathfrak{g}, k)$}\label{sec:Cgk}

In this subsection we briefly review key information on the modular tensor categories $\mathcal{C}(\mathfrak{g}, k)$ where $\mathfrak{g}$ is a simple Lie algebra, and $k\in \mathbb{N}$. These are the categories of level-$k$ integrable representations of $\widehat{\mathfrak{g}}$ \cite{KacBook}. There exists a tensor product on this category given by level preserving fusion. With this tensor product, the category $\mathcal{C}(\mathfrak{g}, k)$ has the structure of a modular tensor category. We will present the relevant combinatorics for this paper. These results are taken from \cite[Subsection 2.1]{Gannon}.

We begin by describing the simple objects in this category. Let $\mathfrak{g}$ be a simple Lie algebra, and $\Gamma_\mathfrak{g}$ the associated Dynkin diagram. The simple objects of $\mathcal{C}(\mathfrak{g}, k)$ can be described in terms of the fundamental weights $\Lambda_i$ of $\mathfrak{g}$. These fundamental weights are in bijection with the vertices of $\Gamma_\mathfrak{g}$. We define the set
\[ P_+^k(\mathfrak{g}) := \left\{  \sum_{i=1}^{|\Gamma_\mathfrak{g}|  }  \lambda_i \Lambda_i  : \lambda_i \in \mathbb{N}, \quad \sum_{i=1}^{|\Gamma_\mathfrak{g}|  } a_i^\vee \lambda_i \leq k   \right\}  . \]
Here the integers $a_i^\vee$ are the \textit{co-labels} of $\mathfrak{g}$, which are associated to the fundamental weights. We list these for the relevant Lie algebras in this paper:
\begin{alignat*}{3}
    \mathfrak{sl}_N: &\quad \att{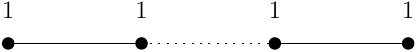}{.5} \qquad \qquad \qquad
      &\mathfrak{so}_{2N+1}: &\quad\att{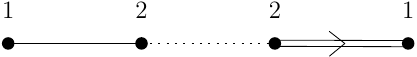}{.5}\\
  \mathfrak{sp}_{2N}: &\quad\att{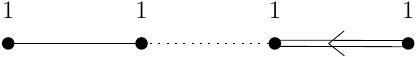}{.5}  &\mathfrak{so}_{2N}: &\quad\att{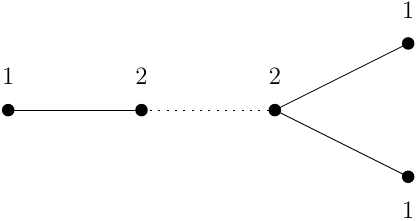}{.5}\\
    \mathfrak{e}_{6}: &\quad\att{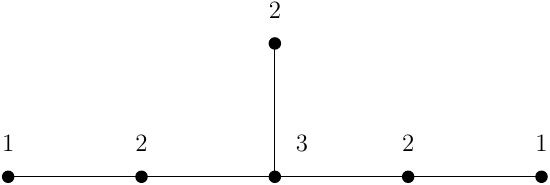}{.5}
    &\mathfrak{e}_{7}: &\quad\att{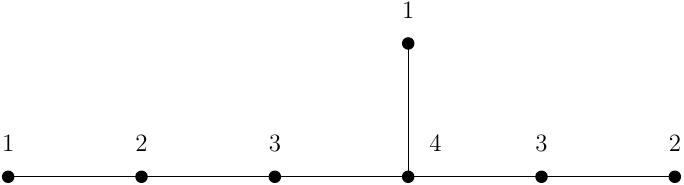}{.5}
\end{alignat*}
We then have a bijection between $P_+^k(\mathfrak{g})$ and the simple objects of $\mathcal{C}(\mathfrak{g}, k)$ sending $\lambda \mapsto V^{\lambda}$.

In the special case of $\mathfrak{g} = \mathfrak{sl}_N$, we have a bijection between the set $P_+^k(\mathfrak{sl}_N)$, and the set of Young diagrams which fit in an $(N-1)\times k$ box. This bijection is given by sending a Young diagram $\mu$ to the weight
\[     \sum_{i=1}^{N-1  }  (\mu_i - \mu_{i+1}) \Lambda_i.           \]
For a Young diagram $\mu$ we will write $V_\mu$ for the corresponding simple object of $\mathcal{C}(\mathfrak{sl}_N, k)$.

The category $\mathcal{C}(\mathfrak{sl}_N, k)$ is $\mathbb{Z}_N$-graded in the sense of \cite[Section 2.3]{homo}. We have that a simple object $V^\lambda$ lives in the 
\[   t(\lambda) := \sum_{i= 1}^{N-1} i \lambda_i     \]
graded piece, modulo $N$. In terms of Young diagrams, the quantity $t(\lambda)$ is the number of boxes in the corresponding Young diagram.

The category $\mathcal{C}(\mathfrak{sl}_N, k)$ has the distinguished set of (isomorphism classes of) invertible objects. These have been classified in \cite{currents}, and are the objects
\[    \operatorname{Inv}(\mathcal{C}(\mathfrak{sl}_N, k))   = \{   V_{[k^j]} : 0 \leq j < N \}.    \]
These objects have the following fusion rules
\[  V_{[k^{j_1}]} \otimes V_{[k^{j_2}]} \cong V_{[k^{j_1+j_2 \pmod N}]},      \]
and hence form a group isomorphic to $\mathbb{Z}_N$. To describe the action of the invertibles on the simple objects of $\mathcal{C}(\mathfrak{sl}_N, k)$, we introduce the following definition.
\begin{defn}\label{def:tau}
    Let $N,k \in \mathbb{N}$, and $\lambda$ a Young diagram which fits in an $(N-1) \times k$ box. We define $\tau(\lambda)$ as the Young diagram obtained by adding a row with $k$ boxes, to the top of $\lambda$, and then deleting any columns which contain $N$ boxes.
\end{defn}
We then have that
\[    V_{[k]} \otimes V_{\lambda} \cong V_{\tau(\lambda)} .    \]
It will be useful to observe that
\[t(\tau^j(\lambda))=k  j+t(\lambda) \pmod N.\]

\subsection{\textit{\'Etale} algebra objects in $\mathcal{C}(\mathfrak{sl}_N,k)$}

In this subsection we review several of the sources of \textit{\'etale} algebra objects in the categories $\mathcal{C}(\mathfrak{sl}_N,k)$. These have been fully classified up to $N=7$ in \cite{Gannon,newGan}. The main constructions of these algebras are through pointed subcategories and conformal embeddings. However there do exist exotic examples not related to either of these constructions e.g. \cite{A67}.

We first discuss the pointed \textit{\'etale} algebra objects in $\mathcal{C}(\mathfrak{sl}_N,k)$. 

The category $\mathcal{C}(\mathfrak{sl}_N,k)$ contains the distinguished subcategory $\mathcal{C}(\mathfrak{sl}_N,k)^{pt}$ generated by the invertible objects $V_{[k^j]}$ where $0 \leq j <N$. Explicit formulae for the braiding on this subcategory can be deduced from \cite[Equation 2.4]{Gannon},\cite[Section 2.5]{quinn} and implicit in \cite{KacBook}. From these formulae, it is seen that for $m$ a divisor of $N$, the subcategory generated by the invertibles
\[  \left \{ V_{\left[k^{\frac{N}{m} j}\right]} : 0 \leq j < m  \right\}      \]
is braided equivalent to $\operatorname{Rep}(\mathbb{Z}_m)$ if and only if $m^2 \mid Nk$ if $N$ is odd, and $2m^2 \mid Nk$ if $N$ is even. This implies the following result on the  pointed \textit{\'etale} algebra objects in $\mathcal{C}(\mathfrak{sl}_N,k)$.

\begin{prop}\label{prop:etpt}
    Let $N\in \mathbb{N}_{\geq 2}$ and $k\in \mathbb{N}_{\geq 1}$. Then there is a bijection between divisors $m$ of $N$ satisfying $m^2 \mid Nk$ if $N$ is odd, and $2m^2 \mid Nk$ if $N$ is even, and pointed \textit{\'etale} algebra objects $A\in \mathcal{C}(\mathfrak{sl}_N,k)$. This bijection sends such a divisor $m$ to the object
    \[    A_m := \bigoplus_{0 \leq j < m} V_{\tau^{\frac{N}{m} j     }(\emptyset)}.  \]
\end{prop}

Our second class of \textit{\'etale} algebra objects in $\mathcal{C}(\mathfrak{sl}_N,k)$ comes from the theory of conformal embeddings. This construction typically gives non-pointed algebra objects, apart from some specific cases. Let $\mathfrak{g}$ be a simple Lie algebra. For us, a conformal embedding is an embedding of WZW vertex operator algebras \cite{WZW}
\[  \mathcal{V}(\mathfrak{sl}_N, k) \subseteq  \mathcal{V}(\mathfrak{g}, 1)       \]
such that the restriction of the trivial $\mathcal{V}(\mathfrak{g}, 1)$ module to a $\mathcal{V}(\mathfrak{sl}_N, k)$ module decomposes as a finite direct sum of irreducible modules (equivalently an inclusion of the Lie algebras $\mathfrak{sl}_N\subseteq \mathfrak{g}$ with equal central charges $\frac{k(N^2-1)}{k+N}=\frac{dim(\mathfrak{g})}{h^\vee+1}$). A complete list of such inclusions can be found in \cite{LagrangeUkraine}. These inclusions were first found in \cite{Embeddings1,Embeddings2}. It is shown in \cite{Kril,HKL} that $\mathcal{V}(\mathfrak{g}, 1)$ has the structure of an \textit{\'etale} algebra object in $\operatorname{Rep}(   \mathcal{V}(\mathfrak{sl}_N, k)   )$. The category $\operatorname{Rep}(   \mathcal{V}(\mathfrak{sl}_N, k)   )$ is identified with $  \mathcal{C}(\mathfrak{sl}_N, k)$ by \cite{WZW}, hence giving us our construction for \textit{\'etale} algebra objects in these categories. 

\begin{defn}
    For a conformal embedding $\mathcal{V}(\mathfrak{sl}_N, k) \subseteq  \mathcal{V}(\mathfrak{g}, 1) $ we will write $A_{\mathfrak{g}}$ for the corresponding \textit{\'etale} algebra object in $  \mathcal{C}(\mathfrak{sl}_N, k)$ from the above construction.
\end{defn}

For the \textit{\'etale} algebras $A_\mathfrak{g}$, the structure of the local $A_\mathfrak{g}$-modules is fully understood. From \cite[Theorem 5.2]{Kril} we have
\[\mathcal{C}(\mathfrak{sl}_N,k)_{A_\mathfrak{g}}^0 \simeq \mathcal{C}(\mathfrak{g},1).\]

The restriction of the forgetful functor $\operatorname{For} : \mathcal{C}(\mathfrak{sl}_N,k)_{A_\mathfrak{g}}\to \mathcal{C}(\mathfrak{sl}_N,k)$ to the distinguished subcategory $\mathcal{C}(\mathfrak{sl}_N,k)_{A_\mathfrak{g}}^0 \simeq \mathcal{C}(\mathfrak{g},1)$ has been well studied in the literature under the name of branching rules. The data of these forgetful functors will be required later in this paper to apply Corollary~\ref{cor:ind}. Thus we review the data of these functors for the conformal embeddings relevant to this paper.

\subsubsection{The embedding $\mathcal{V}(\mathfrak{sl}_N, N+2) \subset \mathcal{V}(\mathfrak{sl}_{\frac{N(N+1)}{2}}, 1)$}\label{subsub:n+2}

For this conformal embedding, we have that \[\mathcal{C}(\mathfrak{sl}_N, N+2)_{A_{\mathfrak{sl}_{\frac{N(N+1)}{2}}}}^0 \simeq \mathcal{C}\left(\mathfrak{sl}_{\frac{N(N+1)}{2}}, 1\right).\] The simple objects of $\mathcal{C}(\mathfrak{sl}_{\frac{N(N+1)}{2}}, 1)$ are $\left\{V_{\Lambda_i} : 0 \leq i  < \frac{N(N+1)}{2}\right\}$. The combinatorial data of the forgetful functor $\operatorname{For}:\mathcal{C}\left(\mathfrak{sl}_{\frac{N(N+1)}{2}}, 1\right)\to \mathcal{C}(\mathfrak{sl}_{N}, N+2)$ was determined in \cite[Section 1]{branchA}. We summarise their results.

 We define the distinguished vector $\rho = (N,N-1, \cdots, 1)$. For $s \in \{-1,1\}^N$ we define $\sigma_s(\rho)$ as the unique reordering of 
\[   (s_1\times N, s_2 \times (N-1), \cdots, s_N\times 1)       \]
which leaves the vector strictly decreasing. It follows that $\sigma_s(\rho) - \rho$ is a decreasing vector, and hence can be identified with a Young diagram. Let $X$ be the set of pairs $(s,k)$ where $s \in  \{-1,1\}^N$, and $\ell\in \mathbb{Z}$ such that $0 \leq \ell <N$. For $(s,\ell) \in X$ we define
\[ j(s,\ell) := \ell(N+1) + \sum_{i: s_i = 1} N+1 - i. \]

We then have
\[ \dim \operatorname{Hom}(\operatorname{For}(V^{\Lambda_i}), V_\lambda) = \begin{cases}
    1 & \quad \text{ if $\lambda =\tau^\ell( \sigma_s(\rho) - \rho)$ for some $(s,\ell) \in X$ such that $ j(s,\ell)\equiv i \pmod {\frac{N(N+1)}{2}}$}\\
    0 & \quad \text{otherwise.}
\end{cases}
\]
Here $\tau^\ell$ represents $\ell$ applications of the function $\tau$ from Definition~\ref{def:tau}.

\subsubsection{The embedding $\mathcal{V}(\mathfrak{sl}_N, N-2) \subset \mathcal{V}(\mathfrak{sl}_{\frac{N(N-1)}{2}}, 1)$}\label{subsub:n-2}
This conformal embedding is similar to the previous example (and is in fact level-rank dual to it as mentioned in \cite[Section 4.3]{MirrorXu}). We have $\mathcal{C}(\mathfrak{sl}_N, N-2)_{A_{\mathfrak{sl}_{\frac{N(N-1)}{2}}}}^0 \simeq \mathcal{C}(\mathfrak{sl}_{\frac{N(N-1)}{2}}, 1)$ in this case. The simple objects of this category are $\left\{V_{\Lambda_i} : 0 \leq i  < \frac{N(N-1)}{2}\right\}$. The following is taken from \cite[Section 2]{branchA}.

We define the distinguished vector $\rho = (N-1,N-2, \cdots,1, 0)$. For $s \in \{-1,1\}^{N-1}$ we define $\sigma_s(\rho)$ as the unique reordering of 
\[   (s_1 \times (N-1), s_2 \times (N-2), \cdots, s_{N-1}\times 1, 0)       \]
which leaves the vector strictly decreasing. It follows that $\sigma_s(\rho) - \rho$ is decreasing, and can be identified with a Young diagram. Let $X$ be the set of pairs $(s,\ell)$ where $s \in  \{-1,1\}^N$ satisfies $\prod_{i=1}^N s_i = 1$, and $\ell\in \mathbb{Z}$ such that $0 \leq \ell <N$. For $(s,\ell) \in X$ we define
\[ j(s,\ell) := 2\ell(N-1) + \sum_{i: s_i = 1} N - i. \]
We then have
\[ \dim \operatorname{Hom}(\operatorname{For}(V^{\Lambda_i}), V_\lambda) = \begin{cases}
    1 & \quad \text{ if $\lambda =\tau^{2\ell}( \sigma_s(\rho) - \rho)$ for some $(s,\ell) \in X$ such that $ j(s,k)\equiv i \pmod {\frac{N(N-1)}{2}}$}\\
    0 & \quad \text{otherwise.}
\end{cases}
\]
\subsubsection{The embedding $\mathcal{V}(\mathfrak{sl}_N, N) \subset \mathcal{V}(\mathfrak{so}_{N^2-1}, 1)$}\label{subsub:n}
The branching rules for the adjoint embedding $\mathfrak{g} \subset \mathfrak{so}_\mathfrak{g}$ for general $\mathfrak{g}$ were determined in \cite{KacBranch}. In the specific case of $\mathfrak{g} = \mathfrak{sl}_N$, a refined formula was found in \cite{pres}. We summarise these results.

For this family, we have $\mathcal{C}(\mathfrak{sl}_N, N)_{A_{\mathfrak{so}_{N^2-1}}}^0 \simeq \mathcal{C}(\mathfrak{so}_{N^2-1}, 1)$. In the case of $N$ even, we have that there are three simple objects in $\mathcal{C}(\mathfrak{so}_{N^2-1}, 1)$, which we label for convenience $\mathbf{1}$, $V_{\Lambda_1}$, and $S$. When $N$ is odd, the category $\mathcal{C}(\mathfrak{so}_{N^2-1}, 1)$ has the four simple objects which we label $\mathbf{1}$, $V_{\Lambda_1}$, $S^+$, and $S^-$.

For a pair of partitions $\lambda, \mu$ with $l(\lambda) + l(\mu) \leq N$, we define the decreasing vector
\[ (\lambda, \mu) := (\lambda_1 , \lambda_2, \cdots ,-\mu_2,  -\mu_1 ). \]
It follows that this vector can be identified with a Young diagram. It was shown in \cite{pres} that 
\[     \operatorname{For}(\mathbf{1}) \cong \bigoplus_{\substack{H_\lambda(1,1) < N \\|\lambda| \text{ even }}} V_{(\lambda, \lambda^T)}\qquad \text{and}\qquad  \operatorname{For}(V) \cong \bigoplus_{\substack{H_\lambda(1,1) < N \\|\lambda| \text{ odd }}} V_{(\lambda, \lambda^T)}.   \]
The decompositions of the spinor representations were explicitly worked out in \cite{KacBranch}. In the case that $N$ is even, we have
\[     \operatorname{For}\left(S \right) \cong 2^{\frac{N-2}{2}} \cdot V_{(N-1, N-2, \cdots, 1)},   \]
and in the case of $N$ odd, we have
\[ \operatorname{For}\left(S^{\pm}\right)\cong 2^{\frac{N-3}{2}} \cdot V_{(N-1, N-2, \cdots, 1)} .  \]

\subsubsection{Sporadic Embeddings}\label{sub:spor} We will also need the branching rules for several sporadic conformal embeddings. Here we will write $[V]_{\mathbb{Z}_N}$ to represent the orbit of the object $V$ under the action of the invertibles of $\mathcal{C}(\mathfrak{sl}_N, k)$. i.e. the direct sum of the objects in the set $\{  \tau^\ell (V) : \ell \in \mathbb{Z}_N\}$.

In the case of the embedding $\mathcal{V}(\mathfrak{sl}_3, 9) \subset\mathcal{V}(\mathfrak{e}_6, 1)$ we have that $\mathcal{C}(\mathfrak{sl}_3, 9)_{A_{\mathfrak{e}_6}}^0$ has three simple objects which we label $\mathbf{1}$, $g$, and $g^2$. The forgetful functor sends:
\[\mathbf{1} \mapsto  [  V_{\emptyset}  ]_{\mathbb{Z}_3} \oplus \left[ V_{\ydiagram{5,1}  }\right]_{\mathbb{Z}_3},\qquad
   g \mapsto \left[  V_{\ydiagram{4,2} }\right]_{\mathbb{Z}_3},\qquad
   g^2 \mapsto  \left[  V_{\ydiagram{4,2} }\right]_{\mathbb{Z}_3}.\]

For the embedding $\mathcal{V}(\mathfrak{sl}_3, 21) \subset\mathcal{V}(\mathfrak{e}_7, 1)$  we have that $\mathcal{C}(\mathfrak{sl}_3, 21)_{A_{\mathfrak{e}_7}}^0$ has two simple objects which we label $\mathbf{1}$ and $g$. The forgetful functor sends:
\[
    \mathbf{1} \mapsto  [   V_{\emptyset}  ]_{\mathbb{Z}_3} \oplus \left[   V_{\ydiagram{8,4}} \right]_{\mathbb{Z}_3},\qquad
    g \mapsto \left[   V_{\ydiagram{5}} \right]_{\mathbb{Z}_3}\oplus  \left[   V_{\ydiagram{5,5}} \right]_{\mathbb{Z}_3} \oplus  \left[   V_{\ydiagram{11,7}} \right]_{\mathbb{Z}_3}\oplus  \left[   V_{\ydiagram{11,4}} \right]_{\mathbb{Z}_3}.
\]

For the embedding $\mathcal{V}(\mathfrak{sl}_4, 8) \subset\mathcal{V}(\mathfrak{so}_{20}, 1)$  we have that $\mathcal{C}(\mathfrak{sl}_4, 9)_{A_{\mathfrak{so}_{20}}}^0$ has four simple objects which we label $\mathbf{1}, V$ and $S^\pm$. The forgetful functor sends
\[
    \mathbf{1} \mapsto  [  V_{\emptyset } ]_{\mathbb{Z}_4} \oplus \left[   V_{\ydiagram{4,3,1}} \right]_{\mathbb{Z}_4},\qquad  V \mapsto \left[   V_{\ydiagram{2,2}} \right]_{\mathbb{Z}_4} \oplus \left[   V_{\ydiagram{5,3}} \right]_{\mathbb{Z}_4},\qquad 
    S^+ \mapsto \left[   V_{\ydiagram{5,2,1}} \right]_{\mathbb{Z}_4} ,\qquad
    S^- \mapsto \left[   V_{\ydiagram{5,2,1}} \right]_{\mathbb{Z}_4}.
\]

For the embedding $\mathcal{V}(\mathfrak{sl}_6, 6) \subset\mathcal{V}(\mathfrak{sp}_{20}, 1)$ we have that $\mathcal{C}(\mathfrak{sl}_6, 6)_{A_{\mathfrak{sp}_{20}}}^0$ has the 11 simple objects $\{V_{\Lambda_i} : 0 \leq i \leq 10\}$. Here we recall the map $\tau$ from Definition~\ref{def:tau}. The forgetful functors sends
\begin{align*}
    V^{\Lambda_0} &\mapsto \left[   V_{\emptyset} \right]_{\mathbb{Z}_3} \oplus \left[   V_{\ydiagram{2,2,2}} \right]_{\mathbb{Z}_3}  \oplus    V_{\ydiagram{4,4,2,2}} & V^{\Lambda_{10}} &\mapsto\tau(\operatorname{For}( V_{\emptyset} ))\\
    V^{\Lambda_1} &\mapsto  \left[   V_{\ydiagram{1,1,1}} \right]_{\mathbb{Z}_3}  \oplus \left[   V_{\ydiagram{3,3,2,1}} \right]_{\mathbb{Z}_3} & V^{\Lambda_{9}} &\mapsto\tau(\operatorname{For}( V_{\Lambda_1} ))\\
    V^{\Lambda_2} &\mapsto  \left[   V_{\ydiagram{2,2,1,1}} \right]_{\mathbb{Z}_3}  \oplus \left[   V_{\ydiagram{6,5,3,3,1}} \right]_{\mathbb{Z}_3}& V^{\Lambda_{8}} &\mapsto\tau(\operatorname{For}(  V_{\Lambda_2} ))\\
    V^{\Lambda_3} &\mapsto  \left[   V_{\ydiagram{3,2,2,2}} \right]_{\mathbb{Z}_3}  \oplus \left[   V_{\ydiagram{3,3,1,1,1}} \right]_{\mathbb{Z}_3}\oplus  V_{\ydiagram{5,4,3,2,1}}& V^{\Lambda_{7}} &\mapsto\tau(\operatorname{For}(  V_{\Lambda_3} ))\\
    V^{\Lambda_4} &\mapsto  \left[   V_{\ydiagram{6,6,3,3}} \right]_{\mathbb{Z}_3}  \oplus \left[   V_{\ydiagram{4,3,2,2,1}} \right]_{\mathbb{Z}_3}\oplus  V_{\ydiagram{6,4,4,2,2}}& V^{\Lambda_{6}} &\mapsto\tau(\operatorname{For}(  V_{\Lambda_4} )).\\
    V^{\Lambda_5} &\mapsto  \left[   V_{\ydiagram{5,3,3,2,2}} \right]_{\mathbb{Z}_3}  \oplus \left[   V_{\ydiagram{5,5,3,2}} \right]_{\mathbb{Z}_3}
\end{align*}

\subsection{Braided autoequivalences of $\mathcal{C}(\mathfrak{sl}_N, k 
  )^0_{A_m}$}\label{sec:braidauto}

  Here we summerise the main result of \cite{ModPt1} which gave a full classification of the braided autoequivalences of the categories $\mathcal{C}(\mathfrak{sl}_N, k 
  )^0_{A}$ when $A$ is a pointed \textit{\'etale} algebra. Recall these are the algebras of the form $A_m$ for $m$ a divisor of $N$ satisfying $m^2 \mid Nk$ if $N$ odd, and $2m^2 \mid Nk$ if $N$ even. The classification result \cite[Theorem 1.4]{ModPt1} states that if 
  \[(N,k,m) \not\in \{ (2,16,2),(3,9,3),(4,8,4), (5,5,5), (8,4,4), (9,3,3), (16,2,2) \}\]
  then we have the isomorphism
  \[       \operatorname{EqBr}\left(\mathcal{C}(\mathfrak{sl}_N, k)_{A_m}^0\right)   \cong \begin{cases}
\{e\} & \text{ if $N=2$ and $k=2$}\\
\mathbb{Z}_{m'}\times \mathbb{Z}_2^{p_m + t_m} & \text{ if either $N=2$ or $k=2$}\\
   D_{m'}\times \mathbb{Z}_2^{p_m + t_m} & \text{ otherwise}\\
   
\end{cases} . \]
Here $p_m$ and $t_m$ are defined as follows.
\begin{defn}\label{def:tp}
Let $m' = \gcd(m,k)$. We define $p_m$ as the number of odd primes dividing $N \frac{m'}{m^2}$ but not $\frac{k}{m'}$, and 
\[  t_m = \begin{cases}
0 & \text{ if $N \frac{m'}{m^2}$ is odd, or if $\frac{k}{m'} \equiv 0 \pmod 4$, or if both $N \frac{m'}{m^2}\equiv 2 \pmod 4$ and $\frac{k}{m'}$ is odd}\\
1 & \text{ otherwise.}
\end{cases}\]
\end{defn}

In order to determine the structure of the double cosets of the group $\operatorname{EqBr}\left(\mathcal{C}(\mathfrak{sl}_N, k)_{A_m}^0\right)$ with respect to the image of $\operatorname{Aut}(A_m)$ from Definition~\ref{def:act}, we need to describe the above isomorphism in some detail.

The $\mathbb{Z}_2^{p_m + t_m}$ factor corresponds to the so-called \textit{simple current autoequivalences} \cite{Simple}. The explicit structure of these autoequivalences is not required for this paper, and so we neglect to include the details here. A curious reader can find the general definition of these braided autoequivalences in \cite[Lemma 2.26]{ModPt1}, and the isomorphism between $\mathbb{Z}_2^{p_m + t_m}$ and these autoequivalences in \cite[Corollary 3.11]{ModPt1}.

The $\mathbb{Z}_{m'}$ factor in the $N=2$ or $k=2$ case, and the $\mathbb{Z}_{m'}$ subgroup of the $D_{m'}$ factor corresponds to the image of the group $\operatorname{Aut}(A_m)\cong \mathbb{Z}_m$ under the map $\operatorname{Ind}: \operatorname{Aut}(A) \to   \operatorname{EqBr}\left(\mathcal{C}_A^0\right)$ from Lemma~\ref{lem:induce}. The isomorphism is given by choosing an isomorphism $\ell_{\underline{\hspace{.3em}}}: \mathbb{Z}_m \to \widehat{\mathbb{Z}_m}$. The braided autoequivalence corresponding to $i \in \mathbb{Z}_m$ is the map defined on objects by 
\[    (V_\lambda, \chi) \mapsto (V_\lambda, \ell_i \cdot \chi).   \]
The kernel of this map is $m' \mathbb{Z}_m$, and so we have the isomorphism with the $\mathbb{Z}_{m'}$ subgroup of $\mathcal{C}(\mathfrak{sl}_N, k 
  )^0_{A_m}$.

The final component to describe is the $\mathbb{Z}_2$ subgroup of the $D_{m'}$ factor. This $\mathbb{Z}_2$ subgroup is generated by the braided autoequivalence which maps
\[ (V_\lambda, \chi) \mapsto (V_\lambda^*, \chi^{-1}).    \]

We also need to describe some of the exceptional cases excluded above. To describe the objects in these categories we again fix a choice of isomorphism $\ell_{\underline{\hspace{.3em}}}: \mathbb{Z}_m \to \widehat{\mathbb{Z}_m}$.

For $\mathcal{C}(\mathfrak{sl}_3, 9 
  )^0_{A_3}$ we have that the braided autoequivalence group is isomorphic to $S_4$. This group is generated by the $D_3$ subgroup described above, along with the exceptional autoequivalence defined on objects by 
  \[     \left( V_{\ydiagram{2,1}} , \ell_0    \right) 
 \leftrightarrow         \left( V_{\ydiagram{6,3}} , \ell_0    \right)  ,   \qquad\qquad     \left( V_{\ydiagram{6,3}} , \ell_1    \right) \leftrightarrow         \left( V_{\ydiagram{6,3}} , \ell_2    \right) .     \]
 Explicitly, the isomorphism $ \mathcal{C}(\mathfrak{sl}_3, 9 
  )^0_{A_3} \cong S_4$ sends the $D_3$ subgroup to the permutations fixing $1$, and the exceptional autoequivalence to the permutation $(12)(34)$.

 For $\mathcal{C}(\mathfrak{sl}_4, 8 )^0_{A_4}$ we have that the braided autoequivalence group is isomorphic to $S_4$. This group is generated by the $D_4$ subgroup described above, along with the exceptional autoequivalence defined on objects by 
  \[     \left( V_{\ydiagram{2,1,1}} , \ell_0    \right) 
 \leftrightarrow       \left( V_{\ydiagram{4,4}} , \ell_0    \right), \qquad \qquad          \left( V_{\ydiagram{3,1}} , \ell_0    \right) \mapsto          \left( V_{\ydiagram{6,4,2}} , \ell_0    \right)   \mapsto          \left( V_{\ydiagram{3,3,2}} , \ell_0    \right)\mapsto          \left( V_{\ydiagram{6,4,2}} , \ell_2    \right) \mapsto  \left( V_{\ydiagram{3,1}} , \ell_0    \right).      \]
 The explicit isomorphism $\mathcal{C}(\mathfrak{sl}_4, 8 )^0_{A_4}\cong S_4$ sends the $D_4$ subgroup to the natural action on the set $\{1,2,3,4\}$, and the exceptional autoequivalence to the permutation $(12)$.

  For $\mathcal{C}(\mathfrak{sl}_5, 5 )^0_{A_5}$ we have that the braided autoequivalence group is isomorphic to $\operatorname{Alt}(5)$. This group is generated by the $D_5$ subgroup described above, along with the exceptional autoequivalence defined on objects by 
  \[    \resizebox{\hsize}{!}{ $\left( V_{\ydiagram{2,1,1,1}} , \ell_0    \right) 
 \mapsto    \left( V_{\ydiagram{4,3,2,1}} , \ell_0    \right)\mapsto    \left( V_{\ydiagram{4,3,2,1}} , \ell_1    \right)\mapsto   \left( V_{\ydiagram{2,1,1,1}} , \ell_0    \right) , \qquad          \left( V_{\ydiagram{4,3,2,1}} , \ell_2    \right)\mapsto  \left( V_{\ydiagram{4,3,2,1}} , \ell_4    \right)\mapsto  \left( V_{\ydiagram{4,3,2,1}} , \ell_3    \right)\mapsto  \left( V_{\ydiagram{4,3,2,1}} , \ell_2    \right).   $ } \]
  The explicit isomorphism $\mathcal{C}(\mathfrak{sl}_5, 5 )^0_{A_5}\cong \operatorname{Alt}(5)$ sends the $D_5$ subgroup to the natural action on the set $\{1,2,3,4,5\}$, and the exceptional autoequivalence to the permutation $(123)$.

  \begin{rmk}
      We wish to point out that at the level of fusion ring automorphisms, the above exceptional braided autoequivalences have appeared previously in the physics literature. The cases of $\mathcal{C}(\mathfrak{sl}_3, 9 )^0_{A_3}$ and $\mathcal{C}(\mathfrak{sl}_4, 8 )^0_{A_4}$ first appeared (to our best knowledge) in \cite{LeBeak}. The case of $\mathcal{C}(\mathfrak{sl}_5, 5 )^0_{A_5}$ first appeared (to our best knowledge) in \cite{bert55}. These papers made no attempt to classify all such fusion ring automorphisms. The prequel paper of the first author \cite{ModPt1} was the first time at which these fusion ring automorphisms were shown to exist at the categorical level, as well as proving full classification.
  \end{rmk}

\section{Generic Classification}\label{sec:gen}

In this section we consider the case where every \textit{\'etale} alegbra in $\mathcal{C}(\mathfrak{sl}_N,k)$ is pointed. Note that this is the generic situation, as the results of \cite{Gannon} show that for each $N$, there are only finitely many $k$ where there are non-pointed \textit{\'etale} algebras in $\mathcal{C}(\mathfrak{sl}_N,k)$ (in fact $k << \frac{13}{6}N^3$ when $N$ is large). Our main result is the following abstract classification of the indecomposable module categories over $\mathcal{C}(\mathfrak{sl}_N,k)$ in this setting. The key lemma for obtaining this classification will be the following, which counts the size of the double cosets of $ \mathcal{C}(\mathfrak{sl}_N, k)^0_{A_m}$ with respect to the subgroup coming from the image of $\operatorname{Aut}(A_m)$ under the map $\operatorname{Ind}$ from Lemma~\ref{lem:induce}.

\begin{lem}\label{lem:cosets}
    Let $(N,k) \not \in \{ (2,16), (3,9), (4,8), (5,5), (8,4), (9,3), (16,2)\}$, and let $m \mid N$ such that $m^2 \mid Nk$ if $N$ is odd, and such that $2m^2 \mid Nk$ is $N$ is even. Then 
    \[    \left|   \operatorname{Aut}(A_m) \backslash \operatorname{EqBr}\left(  \mathcal{C}(\mathfrak{sl}_N, k)^0_{A_m} \right)  / \operatorname{Aut}(A_m)      \right|  = \begin{cases}
1 & \text{ if $N=2$ and $k=2$}\\
2^{p_m + t_m} & \text{ if either $N=2$ or $k=2$}\\
   2^{1+p_m + t_m} & \text{ otherwise}\\ 
\end{cases}  \]
where $m',p_m$, and $t_m$ are as in Definition~\ref{def:tp}.
\end{lem}
\begin{proof}
We recall from Subsection~\ref{sec:braidauto} the expression for $\operatorname{EqBr}\left(\mathcal{C}(\mathfrak{sl}_N, k)_{A_m}^0\right)$ in the case of \[(N,k) \notin \{(2,16),(3,9),(4,8),(5,5),(8,4), (9,3), (16,2)\}.\] The image of $\operatorname{Aut}(A_m) \cong \mathbb{Z}_m$ in $\operatorname{EqBr}\left(  \mathcal{C}(\mathfrak{sl}_N, k)^0_{A_m} \right)$ under the map $\operatorname{Ind}$ is the $\mathbb{Z}_{m'}$ factor in the $N=2$ or $k=2$ case, and the $\mathbb{Z}_{m'}$ subgroup of the $D_{m'}$ factor in the $N>2$ and $k> 2$ case.

 A direct computation shows that there are $2^{p_m + t_m}$ double cosets of $\mathbb{Z}_{m'}\times \mathbb{Z}_2^{p_m + t_m}$ with respect to the $\mathbb{Z}_{m'}$ subgroup, and $2^{1+p_m + t_m}$ double cosets of $D_{m'}\times \mathbb{Z}_2^{p_m + t_m}$ with respect to the $\mathbb{Z}_{m'}$ subgroup of $D_{m'}$.
\end{proof}

A direct application of Theorem~\ref{thm:nik} gives the following.

\begin{lem}\label{lem:firstcount}
    Let $(N,k) \not \in \{ (2,16), (16,2)\}$ be such that the only \textit{\'etale} algebra objects in $\mathcal{C}(\mathfrak{sl}_N, k)$ are pointed. Then there are 
    \[\sum_{m \mid N : m^2 \mid Nk } \begin{cases}
2^{p_m + t_m} & \text{ if $k=2$}\\
   2^{1+p_m + t_m} & \text{ otherwise}\\ 
\end{cases} \] indecomposable module categories over $\mathcal{C}(\mathfrak{sl}_N, k)$ if $N$ is odd, and 
\[\sum_{m \mid N : 2m^2 \mid Nk }\begin{cases}
1 & \text{ if $N=2$ and $k=2$}\\
2^{p_m + t_m} & \text{ if either $N=2$ or $k=2$}\\
   2^{1+p_m + t_m} & \text{ otherwise}\\ 
\end{cases} \] indecomposable module categories over $\mathcal{C}(\mathfrak{sl}_N, k)$ if $N$ is even.
\end{lem}
\begin{proof}
From \cite[Corollary 3.8]{triples} we have that equivalences classes of indecomposable module categories over $\mathcal{C}(\mathfrak{sl}_N, k)$ are in bijection with triples $(A,\mathcal{F}, A')$ where $A, A'\in \mathcal{C}(\mathfrak{sl}_N, k)$ are \textit{\'etale} algebra objects, and $\mathcal{F}: \mathcal{C}(\mathfrak{sl}_N, k)_{A}^0\to \mathcal{C}(\mathfrak{sl}_N, k)_{A'}^0$ is a braided equivalence. These braided equivalences are considered up to the equivalence relation in Theorem~\ref{thm:nik}.

By assumption, the only connected \textit{\'etale} algebra objects in $\mathcal{C}(\mathfrak{sl}_N, k)$ are pointed, and hence of the form $A_m$ for $m$ a divisor of $N$ such that $m^2$ divides $Nk$ if $N$ is odd, and $2m^2$ divides $Nk$ if $N$ is even. Thus $A = A_{m_1}$ and $A' = A_{m_2}$ for $m_1,m_2$ divisors of $N$ as above.

We first claim that $m_1 = m_2$. To see this, observe from \cite[Corrollary 3.32]{LagrangeUkraine} that 
\[\operatorname{dim}\left( \mathcal{C}(\mathfrak{sl}_N, k)_{A_{m_i}}^0\right) = \frac{\operatorname{dim}(\mathcal{C}(\mathfrak{sl}_N, k))}{\dim(A_{m_i})^2}= \frac{\operatorname{dim}(\mathcal{C}(\mathfrak{sl}_N, k))}{m_i^2}.\]
The existence of a braided equivalence 
\[\mathcal{F}:\mathcal{C}(\mathfrak{sl}_N, k)_{A_{m_1}}^0 \to \mathcal{C}(\mathfrak{sl}_N, k)_{A_{m_2}}^0,\] 
implies $\operatorname{dim}\left( \mathcal{C}(\mathfrak{sl}_N, k)_{A_{m_1}}^0\right)=\operatorname{dim}\left( \mathcal{C}(\mathfrak{sl}_N, k)_{A_{m_2}}^0\right)$. It follows that $m_1 = m_2$.

Our problem thus reduces to counting the number of double cosets of $\operatorname{EqBr}(\mathcal{C}(\mathfrak{sl}_N, k)_{A_m}^0)$ with respect to the image of the map $\operatorname{Aut}(A_m)$. We have from Lemma~\ref{lem:cosets} that for each $m$, there are exactly
\[ \begin{cases}
1 & \text{ if $N=2$ and $k=2$}\\
2^{p_m + t_m} & \text{ if either $N=2$ or $k=2$}\\
   2^{1+p_m + t_m} & \text{ otherwise}\\ 
\end{cases} \]
of these cosets. Hence we have the statement of the lemma.
\end{proof}

The summations appearing in the above Lemma can be drastically simplified. Our next combinatorial result does exactly this.

\begin{lem}\label{lem:count}
    Let $N,k\in \mathbb{N}_{\geq 1}$. Then
    \begin{align*}
\sum_{m \mid N : m^2 \mid Nk } 2^{p_m + t_m} &= \sigma(N) \quad &&\text{if $N$ is odd}\\
\sum_{m \mid N : 2m^2 \mid Nk } 2^{p_m + t_m} &= \begin{cases}
\sigma(N) &\quad \text{if $k$ is even}\\
\sigma(\frac{N}{2}) &\quad \text{if $k$ is odd}\\
\end{cases} \quad && \text{if $N$ is even}
\end{align*}
where $\sigma(\ell)$ is the number of divisors of $\ell$. 
\end{lem}
\begin{proof}
    Let $m$ be a divisor of $N$ such that either $m^2 \mid Nk$, or $2m^2\mid Nk$ depending on the parity of $N$, and let $d$ be a divisor of $N$, or $\frac{N}{2}$ depending on the parity of $N$ and $k$. Let us write the prime decompositions
    \[ N = \prod_p p^{\nu_p}\qquad k = \prod_p p^{\kappa_p}\qquad m = \prod_p p^{\mu_p} \quad \text{and} \quad d = \prod_p p^{\delta_p}. \]
    We then have that $p_m$ is the number of odd primes such that $\mu_p \geq \kappa_p$ and $\nu_p + \kappa_p > 2 \mu_p$. Let us write $\mu_2' = \min(\kappa_2, \mu_2)$. Then we have
    \[  t_m = \begin{cases}
            0 \quad &\text{if $\nu_2+ \mu_2'  = 2\mu_2$, or if $\kappa_2 \geq \mu_2'+2$, or if both $\kappa_2 = \mu_2' $ and $\nu_2 + \mu_2'  = 2\mu_2 + 1$  }\\
            1 \quad &\text{otherwise.}
    \end{cases}     \]
    Assume first that $N$ is odd, and let $d$ be a divisor of $N$. Define $m_d=\gcd\left(d, \frac{Nk}{d}   \right) = \prod_p p^{\mu_p}$ where $\mu_p=\mathrm{min}\{\delta_p,\nu_p+\kappa_p-\delta_p\}$. Note that $m_d$ divides $N$ (since $\mu_p\le\delta_p\le\nu_p$) and $m_d^2$ divides $Nk$ (since $2\mu_p\le \delta_p+(\nu_p+\kappa_p-\delta_p)=\nu_p+\kappa_p$). Conversely, if both $m|N$ and $m^2|Nk$, then for $d=m$ we have $m_d=m$. 
 
Fix such an $m$. We prove the statement of the lemma by showing that the number of divisors $d|N$ with $m_d=m$ is precisely $2^{p_m}$. Let $d|N$ have $m_d=m$. Then $\mu_p=\mathrm{min}\{\delta_p,\nu_p+\kappa_p-\delta_p\}$ so either $\delta_p=\mu_p$ or $\delta_p=\nu_p+\kappa_p-\mu_p$. But for the latter to work, we require $\nu_p+\kappa_p-\mu_p\le\nu_p$ (since $d|N$), i.e. $\kappa_p\le \mu_p$. And for $\delta_p\ne \mu_p$ (so this gives different divisors), we require $\nu_p+\kappa_p-\mu_p\ne \mu_p$, i.e. $\nu_p+\kappa_p>2\mu_p$. 

Thus the number of different  $d|N$ with $m_d=m$ equals $\prod_p2^{t'_p}$ where $t'_p=1$ if both $\nu_p+\kappa_p>2\mu_p$ and $\kappa_p\le \mu_p$; otherwise $t'_p=0$. This completes the proof in the $N$ odd case. 

\smallskip Now consider  $N$ and $k$ even, and let $d$ be a divisor of $N$. Define $m_d=\gcd\left(d, \frac{Nk}{2d}   \right) =\prod_pp^{\mu_p}$ with $\mu_p$ as before, except for $\mu_2=\mathrm{min}\{\delta_2,\nu_2+\kappa_2-\delta_2-1\}$. Then $m_d$ divides $N$ and $2m^2$ divides $Nk$. Conversely, when $d=m$ for such an $m$, we have $m_d=m$.

Fix such an $m$. Again we want to count the number of $d\mid N$ with $m_d=m$. Odd primes behave exactly as before.  Also, either $\delta_2=\mu_2$ or $\delta_2=\nu_2+\kappa_2-\mu_2-1$. But $\delta_2=\nu_2+\kappa_2-\mu_2-1$ requires that $\delta_2\le\nu_2$ (since $d|N$), i.e. $\kappa_2\le\mu_2+1$. And for $\delta_2\ne\mu_2$ (so this gives different divisors), $2\mu_2+1< \nu_2+\kappa_2$.

Thus in this case the number of different  divisors $d|N$ with $m_d=m$ will equal $\prod_p2^{t'_p}$ where $t'_2=1$ if  $\kappa_2\le\mu_2+1$ and  $2\mu_2+1< \nu_2+\kappa_2$; otherwise $t'_2=0$. To complete the proof in this case, it suffices to verify that $t'_2= t_m$.

There are three subcases here:\smallskip

\noindent\textbf{ (i)} $ \kappa_2=\mu_2+1$, so $\mu_2'=\mu_2$ and $\nu_2+\mu'_2-2\mu_2=\nu_2+\kappa_2-2\mu_2-1$. We find that $t_m=1$ iff $\nu_2+\mu'_2>2\mu_2$ iff $ \nu_2+\kappa_2>2\mu_2+1$ iff $t_2'=1$.\smallskip

\noindent\textbf{ (ii)} $ \kappa_2\le\mu_2$, so $\mu_2'=\kappa_2$. We find that $t_m=1$ iff $\nu_2+\mu'_2>2\mu_2+1$ iff $ \nu_2+\kappa_2>2\mu_2+1$ iff $t_2'=1$.\smallskip

\noindent\textbf{ (iii)} $\kappa_2>\mu_2+1$, so $\mu_2'=\mu_2$. Here $t_2'=0$ and $\kappa_2\ge \mu'_2+2$, so also $t_m=0$.\smallskip

Finally, the argument  for $N$ even and $k$ odd (so $\kappa_2=\mu'_2=0$) is similar but easier. Define $m_d=\mathrm{gcd}(d,\frac{Nk}{2d})$ as before. Let $d$ divide $N/2$ and have $m_d=m$. Again $\delta_2=\mu_2$ or $\delta_2=\nu_2+\kappa_2-\mu_2-1$; $\delta_2=\nu_2+\kappa_2-\mu_2-1$ now requires both $\kappa_2\le\mu_2$ (since $d|\frac{N}{2}$), which is automatically satisfied,  and $2\mu_2+1< \nu_2$ (so we get different divisors). So define $t'_2$ as before; we need to show that $t_2'=t_m$. But  $t_m=1$ iff $\nu_2\ge 2\mu_2+2$ iff $t'_2=1$. This completes the proof in this case.
\end{proof}

Putting the results of this section together, we obtain our main result.
\begin{thm}\label{thm:genabs}
    Let $N,k \in \mathbb{N}_{\geq 3}$ be such that the only \textit{\'etale} algebra objects in $\mathcal{C}(\mathfrak{sl}_N, k)$ are pointed. Then there are precisely 
    \[  \begin{cases}
                2\sigma\left( \frac{N}{2}\right) \quad & \text{ if $N$ even and $k$ odd}\\
                2\sigma(N) \quad &\text{ otherwise}
    \end{cases}  \]
    indecomposable module categories over $\mathcal{C}(\mathfrak{sl}_N, k)$, up to equivalence.
\end{thm}

\section{Module categories and modular invariants}\label{sec:cons}

The results of Section~\ref{sec:gen} show that there are precisely $2\sigma(N)$ (or $2\sigma\left(\frac{N}{2}\right)$ if $N$ even and $k$ odd) indecomposable module categories over $\mathcal{C}(\mathfrak{sl}_N, k)$ for generic $k$. We can immediately identify half of these module categories as \textit{simple current} module categories. This follows from the well known result on pointed cyclic algebra objects in $\mathcal{C}(\mathfrak{sl}_N, k)$.
\begin{prop}
    Let $N\geq 2$ and $k\geq 1$. Then there is a bijection between divisors $d$ of $N$ (or $\frac{N}{2}$ if $N$ is even and $k$ is odd), and pointed cyclic algebra objects in $\mathcal{C}(\mathfrak{sl}_N, k)$. This bijection sends such a divisor $d$ to the object
    \[   A_d := \bigoplus_{0 \leq j < m}  V_{\tau^{j \frac{N}{d}}(\emptyset}).  \]
\end{prop}
This proposition is simply the non-\textit{\'etale} version of Proposition~\ref{prop:etpt}, and is proved in a near identical fashion. We note that the general case of pointed algebra objects in tensor categories is worked out in \cite{vecG}. We can then take the category of $A_d$ modules in $\mathcal{C}(\mathfrak{sl}_N, k)$ to obtain half of our indecomposable module categories.
\begin{defn}\label{def:mod+}
For $d$ a divisor of $N$ (or $\frac{N}{2}$), we will write 
\[ \mathcal{M}_{d,+} :=   \mathcal{C}(\mathfrak{sl}_N, k)_{A_d}.   \]
\end{defn}
These module categories are completely understood. In particular, we have the following formula for their modular invariants (this formula was obtained in \cite[Equation 9]{xratedproof}, which applies in our setting by \cite[Section 3.5]{simps}).
\begin{equation}\label{eq:mod} Z(d,+)_{\lambda,\nu} := [ \mathcal{Z}(\mathcal{M}_{d,+})]_{V_\lambda,V_\nu} = \sum_{i=1}^d \delta^d\left(t(\lambda) + \frac{N i \hat{k}}{2d}\right)\delta_{\nu, \tau^{i \frac{N}{d}}(\lambda)}.      \end{equation}
Here $\delta^d(L) = 1$ if $d\mid L$, and $\delta^d(L) = 0$ otherwise, and \[ \hat{k} = \begin{cases}
    k + N \quad &\text{ if $Nk$ odd}\\
    k &\text{otherwise.}
\end{cases}\] To remind the reader, the definition of $t(\lambda)$ is given in Subsection~\ref{sec:Cgk}.

We will require a more explicit form for the matrix $Z(d,+)$. To do this, we make the following definitions.
\begin{defn}\label{def:ll}
    Let $N,k,d,m\in \mathbb{N}$. We will write their prime decompositions as
    \[     N = \prod_p p^{\nu_p},\qquad   k = \prod_p p^{\kappa_p},\qquad d = \prod_p p^{\delta_p},\quad\text{and}\quad  m = \prod_p p^{\mu_p}.     \]
    We will write 
    \[  m_d:=\begin{cases}
            \gcd\left( d, \frac{Nk}{d} \right) \quad &\text{ $N$ odd}\\
            \gcd\left( d, \frac{Nk}{2d} \right) \quad &\text{ $N$ even}\\
    \end{cases} \]
    as in the proof of Lemma~\ref{lem:count}.

    Let $P$ be the set of all primes dividing $N$, and write $2^P := \{-,+\}^P$. For each $m$, define an equivalence relation on $2^P$ by $\vv{a} \sim_m \vv{b}$ iff $\vv{a}_p = \vv{b}_p$ whenever $\kappa_p \leq \mu_p < \frac{\nu_p+\kappa_p}{2}$. Finally we define
    \[  \vv{a}(d)_p := \begin{cases}
    +1 \quad &\text{ if $\delta_p = \mu_p$}\\
    -1 \quad &\text{ otherwise}.
\end{cases}   \]
\end{defn}

\begin{rmk}\label{rmk:bij}
    Note that the proof of Lemma~\ref{lem:count} says that for each $m$, the set of $d$ satisfying $m_d = m$ is in bijection with $2^P / \sim_m$. This bijection in one direction sends $d\mapsto \vv{a}(d)$. In the reverse direction, we send $\vv{a}$ to the divisor $d_{m, \vv{a}}$ defined by
    \[     d_{m, \vv{a}}:= \begin{cases}
        m \prod_{p: \vv{a}_p =-1} p ^{  \nu_p +\kappa_p - 2\mu_p  } \quad &\text{$N$ odd, or $\vv{a}_2 = +1$}\\
        m\cdot  2^{\nu_2 + \kappa_2 - 2\mu_2 -1}\prod_{p\neq 2: \vv{a}_p =-1} p ^{  \nu_p +\kappa_p - 2\mu_p  } \quad &\text{$N$ even and $\vv{a}_2 = -1$}
    \end{cases}       \]
\end{rmk}
We will write $\mathbb{Z}_N$ for the group action of invertible objects of $\mathcal{C}(\mathfrak{sl}_N, k)$ on the set of isomorphism classes of simple objects, and $D_N$ for the extension of $\mathbb{Z}_N$ by the duality map. The map $\tau$ from Definition~\ref{def:tau} is the pullback of the action of the generator $V_{[k]}$ to the indexing set. In a similar fashion, we will write $\lambda^*$ for the pullback of the duality map (so $V_\lambda^* \cong V_{\lambda^*}$).

With these definitions, we can show the following.
\begin{prop}\label{prop:modSc}The modular invariant $Z(d,+)$ satisfies $Z(d,+)=Z(d,+)^t$ and
$$Z(d,+)_{{\lambda},{\mu}}=    \left\{\begin{array}{cc} |\mathrm{stab}_{\bbZ_{m_d}}(\lambda)|&\mathrm{ if}\ m_d|t(\lambda)\ \mathrm{and}\ \mu\in\bbZ_{m_d} \prod_p \tau^{h_pt(\lambda)\ell_pN_p}(\lambda)\cr 0&\mathrm{otherwise}\end{array}\right.$$
   where $N_p=N/p^{\nu_p+\kappa_p}\in \mathbb{Q}$, $h_p=0$ if $\vv{a}(d)_p=+1$, $h_p=-2$ if $\vv{a}(d)_p=-1$, and $\ell_p\equiv({N_pk})^{-1}$ (mod $p^{\nu_p}$).
  \end{prop}
    
\begin{proof}
By the Chinese remainder theorem, we can write any $i\in\bbZ_d$ in the form $i=\sum_{p|N}i_pd/{p^{\delta_p}}$ for unique $i_p\in\bbZ_{p^{\delta_p}}$.    Plugging this into the $\delta^d$ condition in Equation~\eqref{eq:mod}, we get (for odd $p$) that $-2t(\lambda)\equiv Nk/{p^{\delta_p}}i_p$ (mod $p^{\delta_p}$) and (for $p=2$) $-t(\lambda)\equiv Nk/{2^{\delta_2+1}}i_2$ (mod $2^{\delta_2}$). Note that (for odd $p$) $p^{\mu_p}$ divides $Nk/p^{\delta_p}$ since $\delta_p+\mu_p\le (\nu_p+\kappa_p-\mu_p)+\mu_p=\nu_p+\kappa_p$, and $2^{\mu_2}$ divides $Nk/2^{\delta_2+1}$ since $\delta_2+1+\mu_2\le (\nu_2+\kappa_2-\mu_2-1)+1+\mu_2=\nu_2+\kappa_2$. Thus there is a solution for $i$ only if $m_d$ divides $t(\lambda)$.  

Note that $N_pk$ is an integer coprime to $p$, so $\ell_p$ exists. Assume $m_d$ divides $t(\lambda)$.
If $\mu_p=\delta_p$ (i.e. $\vv{a}(d)_p=+1$), then any $i_p$ works. If  $\mu_p<\delta_p\le\nu_p$ (i.e. $\vv{a}(d)_p=-1$), then $i_p$ works iff $i_p\equiv -2\ell_pt(\lambda)/p^{\mu_p}$ (mod $p^{\delta_p-\mu_p}$) (when $p$ is odd), or iff  $i_2\equiv -\ell_2t(\lambda)/2^{\mu_2}$ (mod $2^{\delta_2-\mu_2}$) ($p=2$).  Hence if $Z(d,+)_{\lambda,\mu}\ne0$, then $Z(d,+)_{\lambda,\mu'}\ne0$ iff $\mu'\in\bbZ_{m_d}\mu$, in which case $Z(d,+)_{\lambda,\mu'}=Z(d,+)_{\lambda,\mu}$.

Finally, suppose  $\vv{a}(d)_p=-1$. Since $t(\tau^{rN/p^{\mu_p}}\tau^{-2t(\lambda)\ell_pN_p}\lambda)=rNk/p^{\mu_p}-2kt(\lambda)\ell_pN_p+t(\lambda)\equiv0-2t(\lambda)+t(\lambda)\equiv-t(\lambda)$ (mod $p^{\delta_p}$) for any $r$, the solution $i_p$ for $\tau^{rN/p^{\mu_p}}\tau^{-2t(\lambda)\ell_pN_p}\lambda$ will be the negative of that for $\lambda$, so $$Z(d,+)_{\tau^{rN/m}\prod_p \tau^{h_pt(\lambda)\ell_pN_p}\lambda,\tau^{rN/m}(\lambda)}=|\mathrm{stab}_{\bbZ_{m_d}}(\tau^{rN/m}\prod_p\tau^{h_pt(\lambda)\ell_pN_p}\lambda)|$$ $$=|\mathrm{stab}_{\bbZ_m}(\lambda)|=Z(d,+)_{\lambda,\tau^{rN/m}\prod_p\tau^{h_pt(\lambda)\ell_pN_p}\lambda}$$
so we get that $Z(d,+)$ is symmetric.
\end{proof}

The module categories $\mathcal{M}_{d,+}$ give $\sigma(N)$ (or $\sigma\left(\frac{N}{2}\right)$) indecomposable module categories. We obtain the remaining indecomposable module categories of $\mathcal{C}(\mathfrak{sl}_N, k)$ via more abstract means. Consider the triple $(\mathbf{1}, \mathcal{F}_{\text{charge}}, \mathbf{1})$, where $\mathcal{F}_{\text{charge}}$ is the charge-conjugation braided autoequivalence of $\mathcal{C}(\mathfrak{sl}_N, k)$ from \cite{autos}. This braided autoequivalence is non-trivial whenever $N\geq 3$ and $k\geq 3$. On objects we have
\[      \mathcal{F}_{\text{charge}}(V_\lambda) = V_\lambda^*.     \]
From \cite[Corollary 3.8]{triples}, this triple corresponds to an indecomposable $\mathcal{C}(\mathfrak{sl}_N, k)$ module category.
\begin{defn}
    Let $N,k \geq 3$. We write $\mathcal{M}_{1,-}$ for the indecomposable $\mathcal{C}(\mathfrak{sl}_N, k)$ module category corresponding to the triple $(\mathbf{1}, \mathcal{F}_{\text{charge}}, \mathbf{1})$.
\end{defn}
From Theorem~\ref{thm:nik} the modular invariant of $\mathcal{M}_{1,-}$ is given as follows.
\[    Z(1,-)_{\lambda, \nu} := [ \mathcal{Z}(\mathcal{M}_{1,-})]_{V_\lambda,V_\nu}  = \delta_{\lambda, \nu^*}.   \]
Note that $\mathcal{M}_{1,-}$ is an invertible module category (of order 2), via \cite[Remark 3.8]{triples}. We can thus obtain indecomposable module categories via relative tensor products. 
\begin{defn}
    For $d$ a divisor of $N$ (or $\frac{N}{2}$), we define
    \[    \mathcal{M}_{d,-} :=  \mathcal{M}_{1,-} \boxtimes_{\mathcal{C}(\mathfrak{sl}_N, k)}  \mathcal{M}_{d,+} .    \]
\end{defn}
From the formula given in Subsection~\ref{sub:modinv} we have the following modular invariant for the module categories.
\[    Z(d,-) = Z(1,-) Z(d,+).   \]
The module categories $\mathcal{M}_{d,\pm}$ give $2\sigma(N)$ (or $2\sigma(\frac{N}{2})$) indecomposable module categories over $\mathcal{C}(\mathfrak{sl}_N, k)$. However, we must make sure that these module categories are distinct to verify we have a complete list of indecomposable module categories. We will do this by showing that the corresponding modular invariants are distinct, as these are invariants of the module category. We will in fact show the stronger result that the modular invariants are linearly independent over $\mathbb{Z}$. This stronger result will be required later in the section.

We first need the following technical lemma. 
\begin{lem}\label{lem:techGan}
We have the following:
\begin{enumerate}[a)]
    \item For all $N\ge2$, $k\ge 1$, and $0\le a<N$ except for $(N,k,a)=(2,2,1)$, there exists an object $V_\lambda\in\mathcal{C}(\mathfrak{sl}_N,k)$ with trivial $\mathbb{Z}_N$-stabiliser and $t(\lambda)\equiv a$ (mod $N$).

\item For all $N\ge3$, $k\ge 3$, and $0\le a<N$  except for $$(N,k,a)\in\{(3,3,0),(3,6,0),(4,4,0),(4,4,2),(5,5,0),(6,3,3),(6,3,0)\}\,,$$ there exists an object $V_\lambda\in\mathcal{C}(\mathfrak{sl}_N,k)$ with trivial $D_N$-stabiliser and $t(\lambda)\equiv a$ (mod $N$).
\end{enumerate}
\end{lem}
\begin{proof}For part (a), $\lambda=\Lambda_a$ works unless both $a=N/2$ and $k=2$, in which case take $\lambda=\Lambda_1+\Lambda_{\frac{N}{2}-1}$. 

For part (b), again $\lambda=\Lambda_a$ works, unless $a=\frac{N}{2}$ or $a=N$.   

For $a=\frac{N}{2}$, take $\lambda=\Lambda_1+\Lambda_{\frac{N}{2}-1}$.   To see it has trivial $D_N$-stabiliser, for $N>6$ such a weight has two strings of 0's of differing lengths; for $N=6$ with $k>3$, use $\lambda_0>\mathrm{max}\{\lambda_i\}$; and for $N=4$ with $k>4$ again use   $\lambda_0>\mathrm{max}\{\lambda_i\}$.

 For $a=0$, take $\lambda=2\Lambda_1+\Lambda_{N-2}$. This object is shown to have trivial $D_N$ stabiliser in Proposition~\ref{prop:fix}.
 \end{proof}

With this technical lemma in hand, we can obtain linear independence.

 \begin{lem}
     Let $N,k\geq 3$, then the modular invariants 
     of the non-exceptional module categories of  $\mathcal{C}(\mathfrak{sl}_N,k)$ are linearly independent over $\mathbb{Z}$, except for $$(N,k)\in\{(3,3),(3,6),(6,3),(4,4),(5,5)\}\,.$$
 For  $(N,k)=(3,3),(3,6),(5,5)$, the only linear dependencies are $Z(N,+)=Z(N,-)$; for $(N,k)=(6,3)$ the only is $Z(3,+)=Z(3,-)$; and for $(N,k)=(4,4)$ they are generated by $Z(2,+)=Z(2,-)$ and $Z(4,+)=Z(4,-)$.
 \end{lem}

\begin{proof}

We avoid for now the values listed in the statement of the theorem. Given any divisor $d$ of $N$ or $N/2$ as appropriate, define $m_d=\mathrm{gcd}\{d,N\hat{k}/(2d)\}$ as before and $0<a_d< N$ by $a_d\equiv -m_d$ (mod $N$). Let $\lambda$ be the object promised by Lemma~\ref{lem:techGan} (b). 
 
 Suppose we have $\sum_dc_dZ(d,+)+\sum_dc'_dZ(d,-)=0$, where the sums are over divisors $d$ of $N$ or $N/2$ as appropriate. Let $d'$ be the largest such divisor (i.e. $d'=N$ or $N/2$). Note that $Z(d,+)_{{\lambda},{\tau^{N/d'}(\lambda)}}\ne 0$ for any divisor $d\le d'$ iff $d=d'$, and $Z(d,-)_{{\lambda},{\tau^{N/d'}(\lambda)}}\ne 0$ for all divisors $d$. Thus $c_{d'}=0$. Moreover,  $Z(d.-)_{{\lambda},{\tau^{-N/d'}(\lambda^*)}}\ne 0$ for any divisor $d\le d'$ iff $d=d'$, and $Z(d,+)_{{\lambda},{\tau^{-N/d'}(\lambda^*)}}\ne 0$ for all divisors $d$.  Thus $c'_{d'}=0$. Repeating this argument with the next biggest value of $d$, etc, we get inductively that all coefficients $c_d,c'_d$ must vanish. 
 
 Because of Lemma 1, the only possible linear dependencies are:
 $$Z(N,+)=Z(N,-)\eqno(a)$$ 
 for $(N,k)=(3,3),(3,6),(5,5)$;
 $$Z\left(\frac{N}{2},+\right)=Z\left(\frac{N}{2},-\right)\eqno(b)$$
 for $(N,k)=(6,3)$; and 
  $$c_NZ(N,+)+c_N'Z(N,-)+c_{\frac{N}{2}}Z\left(\frac{N}{2},+\right)+c_{\frac{N}{2}}'Z\left(\frac{N}{2},-\right)=0\eqno(c)$$ 
 for $(N,k) = (4,4)$. There is nothing more to say about (a) and (b), other than that they exist. For (c), $m_4=m_2=2$. We find that the $({\Lambda_2},{\Lambda_2})$ entries require $c_2=-c'_2$, and $({\Lambda_2},{\tau(\Lambda_2)})$ requires $c_4=-c_4'$.
 \end{proof}
 \begin{rmk}
     Note that in general the modular invariant is not a complete invariant of a module category (though our main theorem does show that for generic $k$ it is a complete invariant). While we have the above identities of modular invariants in some special cases, we expect that the corresponding module categories are still distinct. 
 \end{rmk}

 The linear independence of the modular invariants $Z(d,\pm)$ implies that the corresponding module categories $\mathcal{M}_{d,\pm}$ are inequivalent. There are precisely as many of these indecomposable module categories, as abstractly classified by Theorem~\ref{thm:genabs}. Hence we obtain the following.
 \begin{thm}\label{thm:list}
   Let $N,k\geq 4$ be such that the only connected \textit{\'etale} algebra objects in $\mathcal{C}(\mathfrak{sl}_N, k)$ are pointed. Then
   \[     \begin{cases} 
   \left\{  \mathcal{M}_{d,\varepsilon} : d \mid \frac{N}{2}, \varepsilon\in \{-,+\}      \right\} \quad &\text{ if $N$ and $k$ are both odd}\\
   \left\{  \mathcal{M}_{d,\varepsilon} : d \mid N, \varepsilon\in \{-,+\}      \right\} \quad &\text{ otherwise}
   \end{cases}     \]
   is a complete set of representatives for the indecomposable module categories over $\mathcal{C}(\mathfrak{sl}_N, k)$.
 \end{thm}

Now that we have a complete set of representatives for the indecomposable module categories, and as we know that that their modular invariants are linearly independent, we can compute the relative tensor products of these module categories by determining relations between the corresponding modular invariants.

\begin{prop}
We have the following identities of modular invariants.
\begin{enumerate}
    \item $Z(1,-)^2 = Z(1,+)$,
    \item $Z(1,-)Z(d,+)=Z(d,+)Z(1,-)$,
    \item $Z(d',+)Z(d'',+)=Z(d'',+)Z(d',+)=rZ(d,+)$  
\end{enumerate}
 where $r=\mathrm{gcd}\{m_{d'},m_{d''}\}$, and $d$ is defined under the bijection from Remark~\ref{rmk:bij} by $m_d=\mathrm{lcm}\{m_{d'},m_{d''}\}$ and $\vv{a}(d)=\vv{a}(d')\,\vv{a}(d'')$.   
\end{prop}

\begin{proof}
We have (1) by simple computation. 

Since $Z(1,-)=S^2$ (where $S$ is the normalised $S$ matrix for $\mathcal{C}(\mathfrak{sl}_N,k)$, we obtain $Z(1,-)Z(d)=Z(d)Z(1,-)$ for all $d$, as any modular invariant commutes with the $S$ matrix. This shows (2).

Note that, for any divisors $d_i$ of $N$, we have
\begin{equation}\label{eq:ss}\langle \tau^{N/d_1}(\mathbf{1})\rangle\cap\langle \tau^{N/d_2}(\mathbf{1})\rangle=\langle \tau^{N/\mathrm{gcd}\{d_1,d_2\}}(\mathbf{1})\rangle
\end{equation}
or symbolically, $\bbZ_{d_1}\cap\bbZ_{d_2}=\bbZ_{\mathrm{gcd}\{d_1,d_2\}}$. Thus if stab$_{\bbZ_N}(X)=\bbZ_s$ then stab$_{\bbZ_m}(X)=\bbZ_{\mathrm{gcd}\{s,m\}}$; just as the orbit $\bbZ_N(X)$ equals $\{X,\tau(X),...,\tau^{N/s-1}(X)\}$, the orbit $\bbZ_mX$ equals $\{X,\tau^{N/m}(X),...,\tau^{N/\mathrm{gcd}\{s,m\}-N/m}(X)\}$.

Let $Z=Z(d',+)Z(d'',+)$, and set $m=\mathrm{lcm}\{m_{d'},m_{d''}\}$ and $r=\mathrm{gcd}\{m_{d'},m_{d''}\}$. Suppose $Z_{X,Y}\ne 0$, where stab$_{\bbZ_N}(X)=\bbZ_s$. Then there is some $W$ such that both $Z(d')_{X,W}\ne0,Z(d'')_{W,Y}\ne0$. Thus (using Proposition~\ref{prop:modSc})
$$Z_{X,Y}=\sum_{i=0}^{s-1}Z(d')_{X,\tau^iW}Z(d'')_{Y,\tau^iW}=\mathrm{gcd}\{s,m_{d'}\}\,\mathrm{gcd}\{s,m_{d''}\}\,\mathrm{gcd}\left\{\frac{m_{d'}}{\mathrm{gcd}\{s,m_{d'}\}},\frac{m_{d''}}{\mathrm{gcd}\{s,m_{d''}\}}\right\}$$
We claim that that equals $r\,\mathrm{gcd}\{s,m\}$. To see that, let $p^\sigma,p^{\mu'},p^{\mu''},p^\mu,p^\rho$ be the exact powers of a prime $p$ dividing $s,m_{d'},m_{d''},m,r$ respectively. We may assume without loss of generality that $\mu'\le\mu''$, so $\mu=\mu''$ and $\rho=\mu'$. If $\sigma\le\mu'$, then the $p$th power of both expressions equals $\sigma+\sigma+(\mu''-\sigma)=\sigma+\mu''$. If instead $\mu'\le\sigma\le\mu''$, then the $p$th power becomes $\mu'+\sigma+0=\mu'+\sigma$. And finally, if $\mu''\le\sigma$, then the $p$th power becomes $\mu'+\mu''+0+0=\mu'+\mu'' $. Note that this is $r$ times the entry $Z(m)_{X,X}$.

Next, suppose $X$ has $Z_{X,Y}\ne 0$ for some $Y$. Again, that means $Z(d')_{X,W}\ne0,Z(d'')_{W,Y}\ne0$ for some $W$. We want to show  that this can only happen if $m$ divides $t(X)$, i.e. iff both $m_{d'}$ and $m_{d''}$ divide $t(X)$. But  $Z(d',+)_{X,W}\ne0$ requires $m_{d'}$ to divide $t(X)$, and $Z(d'',+)_{W,Y}\ne0$ requires $m_{d''}$ to divide $t(W)$. We know that $W=\prod_p\tau^{h'_pt(X)\ell_pN_p}X$, so we know $m_{d''}$ must divide $kh'_pt(X)\ell_pN_p+t(X)\equiv h_p't(X)+t(X)\equiv \vv{a}(d')_pt(X)$ (mod $p^{\nu_p}$), and we're done. 

Conversely, suppose $m$ divides $t(X)$. Then $m_{d'}$ divides $t(X)$ so $Z(d',+)_{X,W}\ne0$ where $W=\prod_p\tau^{h'_pt(X)\ell_pN_p}X$. As before, $t(W)\equiv \vv{a}(d')_pt(X)$ (mod $p^{\nu_p}$), so Proposition 1 gives us $Z(d'')_{W,Y}\ne0$ where $$Y=\prod_p \tau^{h''_pt(W)\ell_pN_p}W=\prod_p \tau^{h''_p\vv{a}(d')_pt(X)\ell_pN_p+h'_pt(X)\ell_pN_p}X\,.$$
Let $d$ be the unique divisor of $N$ (or $N/2$ if $N$ is even but $k$ is odd) with $m_d=m$ and $\vv{a}(d)_p=\vv{a}(d')_p\vv{a}(d'')_p$. If $\vv{a}(d')_p=1$ then $h''_p=h_p$ and $h''_p\vv{a}(d')_p+h'_p=h_p$. If  $\vv{a}(d')_p=-1$ then $h''_p\vv{a}(d')_p+h'_p$ equals $-2$ (if $\vv{a}(d'')_p=1$) or 0 (if $\vv{a}(d'')_p=-1$). In all cases, we get $Y=\prod_p \tau^{h_pt(X)\ell_pN_p}X$, in agreement with what we'd get for $Z(d,+)$.

Finally, we need to show $Z_{X,\tau^{Nq/m}Y}=Z_{X,Y}$ for any $q\in\bbZ$, where $X,Y$ are as in the previous paragraph. Write $\tau^{Nq/m}=\tau^{Nq'/m_{d'}}\tau^{Nq''/m_{d''}}=(\prod_p \tau^{Nq'_p/p^{\mu'_p}})\tau^{Nq''/m_{d''}}$ and define $\bar{q}'$ by  $\bar{q}'N/m_{d'}=\sum_p\vv{a}(d'')_pq'_pN/p^{\mu'_p}$. Then  $Z(d',+)_{X,\tau^iW}=Z(d')_{X,\tau^{\bar{q}'N/m_{d'}}\tau^iW}$ and $Z(d'',+)_{\tau^{\bar{q}'N/m_{d'}}\tau^iW,\tau^{Nq/m}Y}=Z(d'',+)_{\tau^{\bar{q}'N/m_{d'}}\tau^iW,\tau^{Nq'/m_{d'}}Y}$  using the $\bbZ_{m_{d'}}$ and $\bbZ_{m_{d''}}$ invariance of $Z(d',+)$ resp.\ $Z(d'',+)$ promised by Proposition~\ref{prop:modSc}. But
$$(\sum_ph''_pt(\tau^{\bar{q}'N/m_{d'}})\ell_pN_p)+{\bar{q}'N/m_{d'}}=\sum_p(h''_p\vv{a}(d'')_pq_p'N/p^{\mu_p'})+(\vv{a}(d'')_p{q}'_pN/p^{\mu'_p})=\sum_p{q}'_pN/p^{\mu'_p}=q'N/m_{d'}$$
Thus $Z(d'',+)_{\tau^{\bar{q}'N/m_{d'}}\tau^iW,\tau^{q'N/m_{d'}}}=Z(d'',+)_{\tau^iW,Y}$. Putting all this together, $$Z_{X,\tau^{qN/m}Y}=\sum_{i=0}^{s-1}Z(d',+)_{X,\tau^{\bar{q}'N/m_{d'}}\tau^iW}Z(d'',+)_{\tau^{\bar{q}'N/m_{d'}}\tau^iW,\tau^{q'N/m_{d'}}Y}=\sum_{i=0}^{s-1}Z(d',+)_{X,\tau^iW}Z(d'',+)_{\tau^iW,Y}=Z_{X,Y}$$
and $Z=r\,Z(d,+)$, using Proposition~\ref{prop:modSc}.

\end{proof}
Combining the results of this section, we obtain the proof of our main result, Theorem~\ref{thm:bigmain}.



\section{Small $N$ Classification}\label{sec:spec}
In this section we give a full classification (i.e. for all levels $k$) of indecomposable module categories over $\mathcal{C}(\mathfrak{sl}_N,k)$ where $N\in \{3,4,5,6,7\}$ and all $k\in \mathbb{N}$. We can achieve this due to the work of second author, which classifies all connected \textit{\'etale} algebra objects in $\mathcal{C}(\mathfrak{sl}_N,k)$ for these parameter values \cite{Gannon,newGan}. Hence we can classify such modules by classifying all braided equivalences between the corresponding categories of local modules for these \textit{\'etale} algebra objects. The classification of these braided equivalences was achieved by the first author in \cite{ModPt1}. In this section we work through the details of combining these results to obtain the desired classification results.

\subsection{The $\mathfrak{sl}_3$ case}

We begin with the case of $\mathfrak{sl}_3$. While classification in this case is certainly known (see \cite{Ocneanu}), we are unaware of a reference for the result of Theorem~\ref{thm:spec} in the $\mathfrak{sl}_3$ case. Hence we include a proof here for completeness.

\begin{proof}[Proof of Theorem~\ref{thm:spec} in the case of $\mathfrak{g} = \mathfrak{sl}_3$]

From both \cite{Gannon} and \cite{ModPt1} we have the following table listing all of the \textit{\'etale} algebra objects in $\mathcal{C}(\mathfrak{sl}_3,k)$ for $k\in \{5,9,21\}$, along with the groups of braided autoequivalences of their categories of local modules, and the number of double cosets in these autoequivalence groups with respect to the image of the automorphism group of the corresponding \textit{\'etale} algebra. 

\begin{center}\begin{tabular}{c|c| c c c}
    $\mathcal{C}$ & \textit{\'Etale} Algebras $A$  & $\operatorname{EqBr}(\mathcal{C}_A^0)$ & $|\operatorname{Aut}(A)\backslash\operatorname{EqBr}(\mathcal{C}_A^0)/\operatorname{Aut}(A)|$ \\\hline
     $\mathcal{C}(\mathfrak{sl}_3,5)$ & $\mathbf{1}$ &  $\mathbb{Z}_2^2$ & $ 4$\\
                                      & $A_{\mathfrak{sl}_6}$ &  $\mathbb{Z}_2$ & $ 2$\\\hline 
    $\mathcal{C}(\mathfrak{sl}_3,9)$ & $\mathbf{1}$ &  $\mathbb{Z}_2$ & $ 2$\\
                                      & $A_3$ &  $S_4$ & 4\\
                                       & $A_{\mathfrak{e}_6}$ &  $\mathbb{Z}_2$ &   $1$ or $2$\\ \hline
    $\mathcal{C}(\mathfrak{sl}_3,21)$ & $\mathbf{1}$ &  $\mathbb{Z}_2$ & $ 2$\\
                                      & $A_3$ &  $D_3$ & $2$\\
                                      & $A_{\mathfrak{e}_7}$ &  $\{e\}$ & $ 1$\\
\end{tabular}\end{center}
The second column follows from \cite{Gannon} and the third column follows from \cite[Theorem 1.3 and Theorem 1.4]{ModPt1}. Hence the last column is the only one that requires justification.

For the cases where $A = \mathbf{1}$, we have that $\operatorname{Aut}(\mathbf{1})$ is trivial, and hence every element of $\operatorname{EqBr}(\mathcal{C}_A^0)$ belongs to a distinct double coset. This justifies the first, third, and sixth rows. 

For the case where $A = A_3$ we can apply Lemma~\ref{lem:cosets} when $k\neq 9$ to justify the seventh row. When $k=9$ we have an exceptional autoequivalence group of $\mathcal{C}(\mathfrak{sl}_3,9 )^0_{A_3}$ as seen in the above table. From Subsection~\ref{sec:braidauto} the image of $\mathbb{Z}_3 \cong \operatorname{Aut}(A_3)$ is the subgroup of $S_4$ generated by the permutation $(234)$. A direct computation then shows that there are four double cosets of $S_4$ with respect to the subgroup $\langle (234) \rangle$. This justifies the fourth row.

The non-trivial braided autoequivalence of $\mathcal{C}(\mathfrak{sl}_3,5 )^0_{A_{\mathfrak{sl}_6}}\simeq \mathcal{C}(\mathfrak{sl}_6,1)$ maps $V^{\Lambda_i} \mapsto V^{\Lambda_{-i}}$. From Subsection~\ref{subsub:n+2} we have that the forgetful functor $\operatorname{For}:\mathcal{C}(\mathfrak{sl}_3,5 )^0_{A_{\mathfrak{sl}_6}}\to \mathcal{C}(\mathfrak{sl}_3,5 )$ maps
\begin{align*}
    \operatorname{For}(V^{\Lambda_1}) &\cong V_{\ydiagram{2}} \oplus V_{\ydiagram{5,3}}  \quad & 
   \operatorname{For}(V^{\Lambda_5}) &\cong  V_{\ydiagram{2,2}} \oplus V_{\ydiagram{5,2}}.
\end{align*}
The non-trivial braided autoequivalence of $\mathcal{C}(\mathfrak{sl}_3,5 )^0_{A_{\mathfrak{sl}_6}}$ exchanges $V^{\Lambda_1} \leftrightarrow V^{\Lambda_5}$, but does not preserve the forgetful functor. It follows from Corollary~\ref{cor:ind} that this braided autoequivalence is not induced from an element of $\operatorname{Aut}(A_{\mathfrak{sl}_6})$. Therefore the image of $\operatorname{Aut}(A_{\mathfrak{sl}_6})$ in $\operatorname{EqBr}\left(\mathcal{C}(\mathfrak{sl}_3,5 )^0_{A_{\mathfrak{sl}_6}}\right)$ is trivial, which implies that
\[|\operatorname{Aut}(A_{\mathfrak{sl}_6})\backslash\operatorname{EqBr}(\mathcal{C}(\mathfrak{sl}_3,5)_{A_{\mathfrak{sl}_6}}^0)/\operatorname{Aut}(A_{\mathfrak{sl}_6})| = |\operatorname{EqBr}(\mathcal{C}(\mathfrak{sl}_3,5)_{A_{\mathfrak{sl}_6}}^0)|=2.\]
This justifies the second row.

For the algebra $A_{\mathfrak{e}_6}$ there is a non-trivial braided autoequivalence of $\mathcal{C}(\mathfrak{sl}_3,9 )^0_{A_{\mathfrak{e}_6}}\simeq \mathcal{C}(\mathfrak{e}_6, 1)$ which exchanges the two non-trivial simple objects $g$ and $g^2$. We have that $\operatorname{For}(g) = \operatorname{For}(g^2)$, and hence we can not use Corollary~\ref{cor:ind} to argue that the image of $\operatorname{Ind}$ is trivial in this case. Hence we have an ambiguity in that 
\[\operatorname{Aut}(A_{\mathfrak{e}_6})\backslash\operatorname{EqBr}(\mathcal{C}(\mathfrak{sl}_3,9)_{A_{\mathfrak{e}_6}}^0)/\operatorname{Aut}(A_{\mathfrak{e}_6})\]
may contain either 1 or 2 double cosets. This justifies the fifth row. This ambiguity will be resolved at the end of the proof.

For the case of the algebra $A_{\mathfrak{e}_7}$, there is a unique element of $\operatorname{EqBr}(\mathcal{C}(\mathfrak{sl}_3,21 )^0_{A_{\mathfrak{e}_7}})$, and so there can be only one double coset with respect to the image of $\operatorname{Aut}(A_{\mathfrak{e}_6})$.

For each $\mathcal{C}(\mathfrak{sl}_3,k)$ under consideration, all the \textit{\'etale} algebra objects in $\mathcal{C}$ have distinct dimensions, thus there are no braided equivalences $\mathcal{C}_{A_1}^0 \to\mathcal{C}_{A_2}^0 $ for $A_1 \neq A_2$. Hence the information in the above table counts all triples $(A_1, \mathcal{F}, A_2)$ up to equivalence, and hence by \cite[Corollary 3.8]{triples} classifies indecomposable semisimple modules categories over each of the categories in question, up to the ambiguity in the $k=9$ case. 

We now deal with the ambiguity in the $k=9$ case. Our computations show that there are at most 8 (but potentially 7) indecomposable module categories over $\mathcal{C}(\mathfrak{sl}_3,9 )$ up to equivalence. We find 8 module categories over $\mathcal{C}(\mathfrak{sl}_3,9 )$ constructed in the literature \cite{SU3, lance}. To show that these module categories are all distinct, we note that the fusion graph for action by $V_{\ydiagram{1}}$ is an invariant of an $\mathcal{C}(\mathfrak{sl}_3,9 )$ module. These fusion graphs are pairwise distinct for these 8 modules, and so they are non-equivalent.

This completes the proof of Theorem~\ref{thm:spec} in the case of $\mathfrak{g} = \mathfrak{sl}_3$.
\end{proof}




\subsection{The $\mathfrak{sl}_4$ case}

We now deal the case of $\mathfrak{sl}_4$. The classification in this case was claimed in \cite{Ocneanu}. Here we give a proof for completeness. Note that in \cite{dan}, the explicit structure for these module categories is determined.

\begin{proof}[Proof of Theorem~\ref{thm:spec} in the case of $\mathfrak{g} = \mathfrak{sl}_4$]
We have the following information from \cite{Gannon, ModPt1}:
\begin{center}\begin{tabular}{c|c| c c c}
    $\mathcal{C}$ & \textit{\'Etale} Algebras $A$  & $\operatorname{EqBr}(\mathcal{C}_A^0)$ & $|\operatorname{Aut}(A)\backslash\operatorname{EqBr}(\mathcal{C}_A^0)/\operatorname{Aut}(A)|$ \\\hline
     $\mathcal{C}(\mathfrak{sl}_4,4)$ & $\mathbf{1}$ &  $\mathbb{Z}_2$ & $2$\\
                                        & $A_2$ &  $\mathbb{Z}_2^3$ & $4 $\\
                                      & $A_{\mathfrak{so}_{15}}$ &  $\{e\}$ & $1$\\\hline 
    $\mathcal{C}(\mathfrak{sl}_4,6)$ & $\mathbf{1}$ &  $\mathbb{Z}_2^2$ & $ 4$\\
                                      & $A_2$ &  $\mathbb{Z}_2^2$ & $2$\\
                                       & $A_{\mathfrak{sl}_{10}}$ &  $\mathbb{Z}_2$ & $2$\\ \hline
    $\mathcal{C}(\mathfrak{sl}_4,8)$ & $\mathbf{1}$ &  $\mathbb{Z}_2$ & $ 2$\\
                                      & $A_2$ &  $\mathbb{Z}_2^2$ & $2$\\
                                       & $A_4$ &  $S_4$ & $3$\\
                                      & $A_{\mathfrak{so}_{20}}$ &  $\mathbb{Z}_2$ &  $1$ or $2$\\
\end{tabular}\end{center}
As in the $\mathfrak{sl}_3$ case, we just have to justify the final column.

The cases where $A = \mathbf{1}$ and $A =A_m$ (with $k\neq 8$ and $m\neq 4$) follow from Lemma~\ref{lem:cosets}.

The group $\operatorname{EqBr}(\mathcal{C}(\mathfrak{sl}_4,4 )_{A_{\mathfrak{so}{15}}})$ is trivial, hence there is a unique double coset in 
\[\operatorname{Aut}(A_{\mathfrak{so}_{15}})\backslash\operatorname{EqBr}(\mathcal{C}(\mathfrak{sl}_4,4 )_{A_{\mathfrak{so}_{15}}}^0)/\operatorname{Aut}(A_{\mathfrak{so}_{15}}).\]

The non-trivial element of $\operatorname{EqBr}(\mathcal{C}(\mathfrak{sl}_4,6 )_{A_{\mathfrak{sl}_{10}}}^0)\simeq \mathcal{C}(\mathfrak{sl}_{10},1)$ is defined on objects by $V^{\Lambda_i} \mapsto V^{\Lambda_{-i}}$. From Subsection~\ref{subsub:n+2} we have the branching rules
\[\operatorname{For}(V^{\Lambda_1}) = V_{\ydiagram{2}} \oplus  V_{\ydiagram{6,6,2}}   \oplus  V_{\ydiagram{5,3,2}} \qquad \qquad   \operatorname{For}(V^{\Lambda_9}) = V_{\ydiagram{2,2,2}} \oplus  V_{\ydiagram{6,4}}   \oplus  V_{\ydiagram{5,3,2}}. \]
The forgetful functor is not preserved by the non-trivial braided autoequivalence of $\mathcal{C}(\mathfrak{sl}_{10},1 )$, and so the image of $\operatorname{Aut}(A_{\mathfrak{sl}_{10}})$ in $\operatorname{EqBr}(\mathcal{C}(\mathfrak{sl}_4,6 )_{A_{\mathfrak{sl}_{10}}}^0)$ is trivial. It follows that 
\[|\operatorname{Aut}(A_{\mathfrak{sl}_{10}})\backslash\operatorname{EqBr}(\mathcal{C}(\mathfrak{sl}_4,6 )_{A_{\mathfrak{sl}_{10}}}^0)/\operatorname{Aut}(A_{\mathfrak{sl}_{10}})| = 2.\]

For the algebra $A_{\mathfrak{so}_{20}}$ we have $\operatorname{EqBr}(\mathcal{C}(\mathfrak{sl}_4,8 )_{A_{\mathfrak{so}_{20}}}^0) \cong \mathbb{Z}_2$. In this case Corollary~\ref{cor:ind} is not strong enough to show that the image of $\operatorname{Aut}(A_{\mathfrak{so}_{20}})$ is trivial. Depending on if this image is trivial or not, we have the two possibilities \[|\operatorname{Aut}(A_{\mathfrak{so}_{20}})\backslash\operatorname{EqBr}(\mathcal{C}(\mathfrak{sl}_4,8 )_{A_{\mathfrak{so}_{20}}}^0)/\operatorname{Aut}(A_{\mathfrak{so}_{20}})| \in \{1,2\}.\]
We will resolve this ambiguity at the end of the proof.

For the algebra $A_4\in \mathcal{C}(\mathfrak{sl}_4,8 )$, we have the exceptional braided autoequivalence group 
\[\operatorname{EqBr}(\mathcal{C}(\mathfrak{sl}_4,8 )_{A_4}^0) \cong S_4.\]
From Subsection~\ref{sec:braidauto} and Proposition~\ref{prop:imgpt} we have that the image of $\operatorname{Aut}(A_4)$ in $S_4$ is the subgroup $\langle (1234) \rangle$. A direct computation shows that there are three double cosets in $S_4$ with respect to the subgroup $\langle (1234) \rangle$. Hence
\[|\operatorname{Aut}(A_4)\backslash\operatorname{EqBr}(\mathcal{C}(\mathfrak{sl}_4,8 )_{A_4}^0)/\operatorname{Aut}(A_4)| = 3.\]

By considering global dimensions, we have that the categories $\mathcal{C}(\mathfrak{sl}_4, k)^0_A$ are pairwise distinct for the algebras in question. There every triple is of the form $(A, \mathcal{F}, A)$. The above table shows that there are 7 distinct triples for $\mathcal{C}(\mathfrak{sl}_4, 4)$, 8 for $\mathcal{C}(\mathfrak{sl}_4, 6)$ and 8 or 9 for $\mathcal{C}(\mathfrak{sl}_4, 8)$. The statement of Theorem~\ref{thm:spec} in the $\mathfrak{sl}_4$ case now follows from \cite[Corollary 3.8]{triples}, up to the ambiguity in the $k=8$ case.

In the case of $k=8$, we have deduced that there are either 8 or 9 equivalence classes of indecomposable module categories over $\mathcal{C}(\mathfrak{sl}_4,8 )$. In \cite{dan} 9 non-pairwise equivalent module categories over $\mathcal{C}(\mathfrak{sl}_4,8 )$ are constructed. This resolves the ambiguity.
\end{proof}



\subsection{The $\mathfrak{sl}_5$ case}
The case of $\mathfrak{sl}_5$ follows in a somewhat similar manner to the two previous cases. As usual, the complications arise from not knowing the image of $\operatorname{Aut}(A)$ in $\operatorname{EqBr}\left( \mathcal{C}(\mathfrak{sl}_5, k)^0_A             \right)$ in general. In particular the algebra $A = A_{\mathfrak{so}_{24}}$ causes issues in this case. Before we get into the proof of Theorem~\ref{thm:spec}, we will study this algebra.

To future-proof, we will work with the family of \textit{\'etale} algebras $A_{\mathfrak{so}_{N^2-1}}\in \mathcal{C}(\mathfrak{sl}_N, N)$. We have the equivalence
\[\mathcal{C}(\mathfrak{sl}_N, N)_{A_{\mathfrak{so}_{N^2-1}}}^0 \simeq \mathcal{C}(\mathfrak{so}_{N^2-1}, 1).\]
In the case that $N$ is odd, the category $\mathcal{C}(\mathfrak{so}_{N^2-1}, 1)$ has the two spinor representation $S^{\pm}$, and there is a braided autoequivalance exchanging $S^+ \leftrightarrow S^-$. As described in Subsection~\ref{subsub:n}, we have that $\operatorname{For}(S^+) \cong \operatorname{For}(S^-)$, and so Corollary~\ref{cor:ind} is not sufficient to show that this non-trivial braided autoequivalence is not induced from an element of $\operatorname{Aut}\left(A_{\mathfrak{so}_{N^2-1}}\right)$. To show this braided autoequivalence $S^+ \leftrightarrow S^-$ is not induced from an element of $\operatorname{Aut}\left(A_{\mathfrak{so}_{N^2-1}}\right)$, we can use the presentation of the tensor category $\mathcal{C}(\mathfrak{sl}_N, N)_{A_{\mathfrak{so}_{N^2-1}}}$ obtained in \cite{pres}.

The generators and relations presentation for $\mathcal{C}(\mathfrak{sl}_N, N)_{A_{\mathfrak{so}_{N^2-1}}}$ is given in terms of tensor powers of the simple object $X:= \mathcal{F}_{A_{\mathfrak{so}_{N^2-1}}}(V_\square)$. The generating morphisms of this presentation are
\[   \att{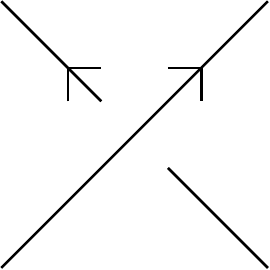}{.2} \in \operatorname{End}(X\otimes X), \quad  \att{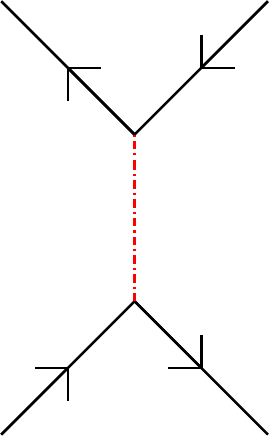}{.2} \in \operatorname{End}(X\otimes X^*) , \quad  \att{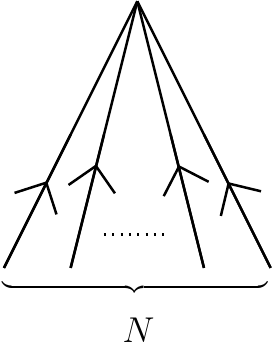}{.25} \in \operatorname{Hom}(X^{\otimes N} \to \mathbf{1}).    \]
We have the decomposition
\[     X\otimes X^* \cong \mathbf{1} \oplus g \oplus Y_1 \oplus Y_2     \]
where $g$ is an invertible object, and the $Y_i$ are simple objects with 
\[ \dim(Y_i) = \frac{-1 - e^{2\pi i\frac{1}{N}} }{(1 - e^{2\pi i\frac{1}{2N}})^2}.\] 
Furthermore, the morphism $\att{splittingEq}{.2}$ is the unique projection onto the $g$ summand of $X\otimes X^*$.
\begin{thm}\label{thm:so}
    Let $N\geq 4$, and $\mathcal{F} \in \operatorname{Eq}\left(    \mathcal{C}(\mathfrak{sl}_N, N)_{A_{\mathfrak{so}_{N^2-1}}}\right)$ be a pivotal functor such that $\mathcal{F}(X) \cong X$. Then $\mathcal{F}\cong \operatorname{Id}$.
\end{thm}
\begin{proof}
    We first observe that as $\mathcal{F}(X) \cong X$, we have that 
    \[   \mathcal{F}\left(\att{braidX}{.2}\right) \in \operatorname{End}(X\otimes X), \quad \mathcal{F}\left( \att{splittingEq}{.2}\right) \in \operatorname{End}(X\otimes X^*) , \quad  \mathcal{F}\left(\att{kw}{.2}\right) \in \operatorname{Hom}(X^{\otimes N} \to \mathbf{1}).    \]
    Supposing $N\geq 4$, then we have that $\dim(Y_i) \neq 1$ using the above dimension formula. As $\mathcal{F}(X\otimes X^*) \cong X\otimes X^*$ it follows that $\mathcal{F}(g) \cong g$. The multiplicity of $g$ in $X\otimes X^*$ is 1, and so 
    \[ \mathcal{F}\left( \att{splittingEq}{.2}\right)  =     \att{splittingEq}{.2} \]
    where we recall that $\att{splittingEq}{.2}$ is the unique projection onto $g$ in $\operatorname{End}(X\otimes X^*)$.

    We have that $ \operatorname{Hom}(X^{\otimes N} \to \mathbf{1})$ is 1-dimensional, and thus 
    \[  \mathcal{F}\left(\att{kw}{.2}\right) = \omega \att{kw}{.2}   \]
    for some $\omega \in \mathbb{C}^\times$. Applying a natural isomorphism to $\mathcal{F}$ allows us to assume that $\omega = 1$.

    From \cite[Corollary 5.7]{pres}, we have that the set
    \[   \left\{\att{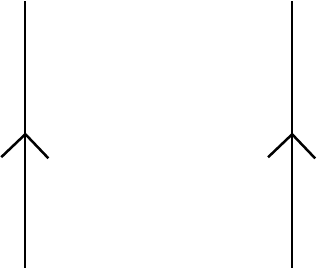}{.2}, \quad \att{braidX}{.2},\quad  \att{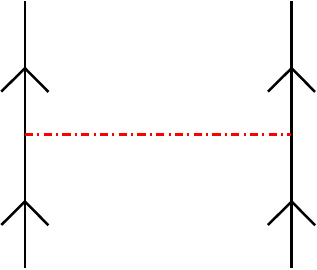}{.2} ,\quad \att{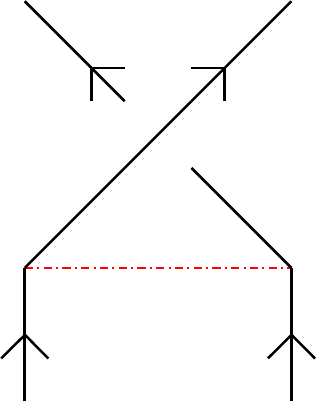}{.2} \right\}          \]
    is a basis for $\operatorname{End}(X\otimes X)$, where
    \[      \att{splittingEnd2}{.2}:= \frac{2  i}{q-q^{-1}}\att{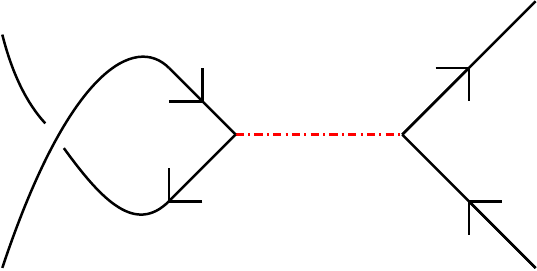}{.2}.       \]
    We thus have
    \[     \mathcal{F}\left(\att{braidX}{.2}\right) = \alpha \att{idid}{.2}+\beta \att{braidX}{.2}+\gamma \att{splittingEnd2}{.2} +\lambda \att{basis4}{.2}         \]
    for some $\alpha, \beta, \gamma, \lambda \in \mathbb{C}$. Using the relations of \cite[Definition 1.3 and Lemma 5.2]{pres}, we can solve to find that $\alpha = \gamma = \lambda = 0$ and $\beta = 1$ is the only solution. Indeed one sees that the element $\att{splittingEnd2}{.2}$ is mapped to a $\beta + \gamma + \lambda \frac{\mathbf{i}(q-q^{-1})}{2}$ scalar multiple of itself. From the relation (Stack), this scalar must be $\pm 1$. In the $-1$ case one solves the (Commute) relation to find $\beta = 0$ and $\gamma = \mathbf{i} \alpha$. Solving the (Hecke) relation reveals no solutions in this case. In the $+1$ case, one solves (Commute) to find $\lambda = 0$ and $\gamma  = -\mathbf{i}\alpha$. Solving for (Hecke) in this case then shows that either we have the trivial solution, or the level-rank duality solution. That is, the map which sends the braid to the negative of its inverse. The level-rank duality solution does not preserve the (Half-Braid) relation, and hence we only have the trivial solution.
    
    Hence $\mathcal{F}$ restricts to the identity functor on the full subcategory of $\mathcal{C}(\mathfrak{sl}_N, N)_{A_{\mathfrak{so}_{N^2-1}}}$ generated by tensor powers of $X$. As $X$ Karoubi generates $\mathcal{C}(\mathfrak{sl}_N, N)_{A_{\mathfrak{so}_{N^2-1}}}$, it follows from \cite[Equations 7 and 9]{comes} that $\mathcal{F} \cong \operatorname{Id}$ as claimed.
\end{proof}

We now obtain the desired corollary which deduces the image of $\operatorname{Aut}\left(A_{\mathfrak{so}_{N^2-1}}\right)$.
\begin{cor}\label{cor:indso}
    Let $N\geq 4$. Then the image of $\operatorname{Aut}\left(A_{\mathfrak{so}_{N^2-1}}\right)$ in $\operatorname{EqBr}\left(     \mathcal{C}(\mathfrak{sl}_N, N)_{A_{\mathfrak{so}_{N^2-1}}}^0        \right)$ is trivial.
\end{cor}
\begin{proof}
Let $\mathcal{F}$ be a braided autoequivalence of $ \mathcal{C}(\mathfrak{sl}_N, N)_{A_{\mathfrak{so}_{N^2-1}}}^0 $ in the image of $\operatorname{Ind}$. It follows from Lemma~\ref{lem:induce} and Corollary~\ref{cor:ind} that $\mathcal{F}$ extends to a pivotal autoequivalence of $ \mathcal{C}(\mathfrak{sl}_N, N)_{A_{\mathfrak{so}_{N^2-1}}} $ which preserves the forgetful functor.

We define $Y := \mathcal{F}(  X  )$. As $\mathcal{F}$ preserves the forgetful functor, we have that $\operatorname{For}(Y) \cong \operatorname{For}(X) \cong A_{\mathfrak{so}_{N^2-1}} \otimes V_\square$. By the adjunction between the free module functor and the forgetful functor we get
\[  \operatorname{dimHom}_{ \mathcal{C}(\mathfrak{sl}_N, N)_{A_{\mathfrak{so}_{N^2-1}}}}(Y \to X  )  = \operatorname{dimHom}_{ \mathcal{C}(\mathfrak{sl}_N, N)}(A_{\mathfrak{so}_{N^2-1}} \otimes V_\square \to V_\square ) \geq 1.       \]
 As $X$ is simple and $\mathcal{F}$ is an equivalence, we get that $Y$ is simple. Thus $X\cong Y$ and the statement follows immediately from Theorem~\ref{thm:so}.   
\end{proof}

With the above Corollary, we are now in place to prove Theorem~\ref{thm:spec} in this case.

\begin{proof}[Proof of Theorem~\ref{thm:spec} in the case of $\mathfrak{g} = \mathfrak{sl}_5$]

We have the following information from \cite{Gannon,ModPt1}:
\begin{center}\begin{tabular}{c|c| c c c}
    $\mathcal{C}$ & \textit{\'Etale} Algebras $A$  & $\operatorname{EqBr}(\mathcal{C}_A^0)$ & $|\operatorname{Aut}(A)\backslash\operatorname{EqBr}(\mathcal{C}_A^0)/\operatorname{Aut}(A)|$ \\\hline
      $\mathcal{C}(\mathfrak{sl}_5,3)$ & $\mathbf{1}$ &  $\mathbb{Z}_2^2$ & $ 4$\\
                                       & $A_{\mathfrak{sl}_{10}}$ &  $\mathbb{Z}_2$ & $2$\\ \hline
     $\mathcal{C}(\mathfrak{sl}_5,5)$ & $\mathbf{1}$ &  $\mathbb{Z}_2$ & $2$\\
                                        & $A_{5}$ &  $\operatorname{Alt}(5)$ & $4 $\\
                                      & $A_{\mathfrak{so}_{24}}$ &  $S_3$ & $6$\\\hline 
  
    $\mathcal{C}(\mathfrak{sl}_5,7)$ & $\mathbf{1}$ &  $\mathbb{Z}_2^2$ & $ 4$\\
                                      & $A_{\mathfrak{sl}_{15}}$ &  $\mathbb{Z}_2^2$ & $4$\\
\end{tabular}\end{center}
As usual, we only need to justify the last column. 

When $A$ is pointed (apart from the case of $A = A_5$ when $k=5$), we have the last column from Lemma~\ref{lem:cosets}.

When $A = A_5$ and $k=5$ we have an exceptional braided autoequivalence group $\operatorname{EqBr}(\mathcal{C}(\mathfrak{sl}_5,5 )^0_{A_5}) \cong \operatorname{Alt}(5)$. From Subsection~\ref{sec:braidauto} and Proposition~\ref{prop:imgpt} we have that the image of $\operatorname{Aut}(A_5)$ in $\operatorname{Alt}(5)$ is the subgroup $\langle (12345)\rangle$. A direct computation then shows that
\[|\operatorname{Aut}(A_5)\backslash\operatorname{EqBr}(\mathcal{C}(\mathfrak{sl}_5,5 )_{A_5}^0)/\operatorname{Aut}(A_5)| = 4.\]

When $A = A_{\mathfrak{sl}_{10}}$, we have the non-trivial braided autoequivalence of $\mathcal{C}(\mathfrak{sl}_5,3)^0_{A_{\mathfrak{sl}_{10}}}\simeq \mathcal{C}(\mathfrak{sl}_{10}, 1)$ defined by $V^{\Lambda_i}\mapsto V^{\Lambda_{-i}}$. We have the following branching rules from Subsection~\ref{subsub:n-2}:
\[ \operatorname{For}(V^{\Lambda_1}) \cong V_{\ydiagram{1,1}} \oplus V_{\ydiagram{3,2,2}}\qquad \qquad \operatorname{For}(V^{\Lambda_9}) \cong V_{\ydiagram{1,1,1}} \oplus V_{\ydiagram{3,3,1,1}}.\]
Hence the non-trivial braided autoequivalence of $\mathcal{C}(\mathfrak{sl}_5,3)^0_{A_{\mathfrak{sl}_{10}}}$ does not preserve the forgetful functor, and so Corollary~\ref{cor:ind} gives that this braided autoequivalence is not induced by an element of $\operatorname{Aut}(A_{\mathfrak{sl}_{10}})$. It follows that 
\[|\operatorname{Aut}(A_{\mathfrak{sl}_{10}})\backslash\operatorname{EqBr}(\mathcal{C}(\mathfrak{sl}_5,3)_{A_{\mathfrak{sl}_{10}}}^0)/\operatorname{Aut}(A_{\mathfrak{sl}_{10}})| = 2.\]

When $A = A_{\mathfrak{so}_{24}}$ we can apply Corollary~\ref{cor:indso} to see that the image of $\operatorname{Aut}\left( A_{\mathfrak{so}_{24}} \right)$ in $\operatorname{EqBr}(\mathcal{C}(\mathfrak{sl}_5,5 )^0_{A_{\mathfrak{so}_{24}}})$ is trivial. Hence
\[\operatorname{Aut}(A_{\mathfrak{so}_{24}})\backslash\operatorname{EqBr}(\mathcal{C}(\mathfrak{sl}_5,5 )^0_{A_{\mathfrak{so}_{24}}})/\operatorname{Aut}(A_{\mathfrak{so}_{24}}) = |\operatorname{EqBr}(\mathcal{C}(\mathfrak{sl}_5,5 )^0_{A_{\mathfrak{so}_{24}}})| = 6.\]

When $A = A_{\mathfrak{sl}_{15}}$ we have three non-trivial braided autoequivalences of $\mathcal{C}(\mathfrak{sl}_5,7)^0_{A_{\mathfrak{sl}_{15}}}\simeq \mathcal{C}(\mathfrak{sl}_{15}, 1)$. These three braided autoequivalences map
\begin{align*}
    V^{\Lambda_1} &\mapsto V^{\Lambda_4} \\
    V^{\Lambda_1} &\mapsto V^{\Lambda_{11}} \\
    V^{\Lambda_1} &\mapsto V^{\Lambda_{14}}
    \end{align*}
    respectively. From Subsection~\ref{subsub:n+2} we have the following branching rules:
\begin{align*}
    \operatorname{For}(V^{\Lambda_1}) &\cong V_{\ydiagram{2}}\oplus V_{\ydiagram{7,6,4}}\oplus V_{\ydiagram{7,4,4,2}}\oplus V_{\ydiagram{5,3,2,2}}\oplus V_{\ydiagram{6,6,3,2}} \\
    \operatorname{For}(V^{\Lambda_4}) &\cong  V_{\ydiagram{4,3,1}}\oplus V_{\ydiagram{5,1,1,1}}\oplus V_{\ydiagram{7,5,3,3}}\oplus V_{\ydiagram{6,6,4,2}}\oplus V_{\ydiagram{7,7,7,2}} \\
    \operatorname{For}(V^{\Lambda_{11}}) &\cong V_{\ydiagram{7,5}} \oplus V_{\ydiagram{6,4,2}}\oplus V_{\ydiagram{4,4,3,1}}\oplus V_{\ydiagram{5,4,4,4}}\oplus V_{\ydiagram{7,4,4,2}} \\
    \operatorname{For}(V^{\Lambda_{14}}) &\cong V_{\ydiagram{2,2,2,2}} \oplus V_{\ydiagram{6,4,3}}\oplus V_{\ydiagram{7,7,3,1}}\oplus V_{\ydiagram{7,5,3,3}}\oplus V_{\ydiagram{5,3,3,2}}.
\end{align*}
It is clearly seen that none of the non-trivial braided autoequivalences of $\mathcal{C}(\mathfrak{sl}_{15}, 1)$ preserve the forgetful functor. Hence Corollary~\ref{cor:ind} gives that the image of $\operatorname{Aut}(A_{\mathfrak{sl}_{15}})$ in $\operatorname{EqBr}(\mathcal{C}(\mathfrak{sl}_{15}, 1))$ is trivial. We thus have
\[|\operatorname{Aut}(A_{\mathfrak{sl}_{15}})\backslash\operatorname{EqBr}(\mathcal{C}(\mathfrak{sl}_5,7 )^0_{A_{\mathfrak{sl}_{15}}})/\operatorname{Aut}(A_{\mathfrak{sl}_{15}})| = |\operatorname{EqBr}(\mathcal{C}(\mathfrak{sl}_5,7 )^0_{A_{\mathfrak{sl}_{15}}})|=4.\]

By considering global dimensions, we have that the categories $\mathcal{C}(\mathfrak{sl}_5, k)^0_A$ are pairwise distinct for the algebras in question. There every triple is of the form $(A, \mathcal{F}, A)$. The above table shows that there are 6 distinct triples for $\mathcal{C}(\mathfrak{sl}_5, 3)$, 12 for $\mathcal{C}(\mathfrak{sl}_5, 5)$ and 8 for $\mathcal{C}(\mathfrak{sl}_5, 7)$. The claimed result is now a direct application of \cite[Corollary 3.8]{triples}.

\end{proof}

\subsection{The $\mathfrak{sl}_6$ case}\label{sub:6}

The case of $\mathfrak{sl}_6$ is somewhat special, as we have the first encounter with \textit{\'etale} algebra objects which are not pointed, nor of the form $A_{\mathfrak{g}}$. However these new algebra object are not typically considered exotic, as they can be constructed via well-understood methods. We begin this subsection by constructing these two algebras.

From the conformal embedding 
\[      \mathcal{V}(\mathfrak{sl}_6, 6) \subset   \mathcal{V}(\mathfrak{sp}_{20}, 1)\]
there exists an \textit{\'etale} algebra object $A_{\mathfrak{sp}_{20}}\in  \mathcal{C}(\mathfrak{sl}_6, 6)$. We then have that
\[  \mathcal{C}(\mathfrak{sl}_6, 6)_{A_{\mathfrak{sp}_{20}}}^0 \simeq   \mathcal{C}(\mathfrak{sp}_{20}, 1).\]
The category $\mathcal{C}(\mathfrak{sp}_{20}, 1)$ itself has an  \textit{\'etale} algebra object $\mathbf{1}\oplus V^{\Lambda_6}$\footnote{The existence of this algebra follows from the braided equivalence $\mathcal{C}(\mathfrak{sp}_{20}, 1)^\textrm{ad} \simeq \mathcal{C}(\mathfrak{sl}_{2}, 10)^{\textrm{ad},\textrm{rev}}$, and from the fact that the category $\mathcal{C}(\mathfrak{sl}_{2}, 10)^{\textrm{ad}}$ has the \textit{\'etale} algebra object $\mathbf{1}\oplus V^{6\Lambda_1}$ from \cite{OstMod}.}. As the forgetful functor \[\operatorname{For}: \mathcal{C}(\mathfrak{sp}_{20}, 1)\to \mathcal{C}(\mathfrak{sl}_{6}, 6)\]
is lax-monoidal and braided, we have that $   \operatorname{For}(  \mathbf{1}\oplus V^{\Lambda_6}  ) $ is an \textit{\'etale} algebra object in $\mathcal{C}(\mathfrak{sl}_{6}, 6)$.

\begin{defn}
    We will write $A_{\mathfrak{sp}_{20}}^{\textrm{ext}}:= \operatorname{For}(  \mathbf{1}\oplus V^{\Lambda_6}  )\in \mathcal{C}(\mathfrak{sl}_{6}, 6)$.
\end{defn}

We also obtain an additional \textit{\'etale} algebra object via level-rank-duality. Recall from \cite{MirrorXu, LROst} that there is a braid-reversing equivalence 
\[\mathcal{T}:\mathcal{C}(\mathfrak{sl}_N, k)^\text{ad} \simeq \mathcal{C}(\mathfrak{sl}_k, N)^\text{ad}.\]
By inspection, the \textit{\'etale} algebra $A_{\mathfrak{sp}_{20}}$ lives in the subcategory $\mathcal{C}(\mathfrak{sl}_6, 6)^\text{ad}$. It follows that $\mathcal{T}( A_{\mathfrak{sp}_{20}}  )$ is an  \textit{\'etale} algebra object in $\mathcal{C}(\mathfrak{sl}_6, 6)$.

\begin{defn}
    We will write $A_{\mathfrak{sp}_{20}}^{\textrm{tr}}:= \mathcal{T}(A_{\mathfrak{sp}_{20}})\in \mathcal{C}(\mathfrak{sl}_{6}, 6)$.
\end{defn}

Using the explicit description of the functor $\mathcal{T}$ from \cite[Section 2.1]{LROst} we have
\[      A_{\mathfrak{sp}_{20}}^{\textrm{tr}}\cong \left[ V_\emptyset  \right]_{\mathbb{Z}_3} \oplus    \left[ V_{\ydiagram{6,3,3}} \right]_{\mathbb{Z}_3}\oplus V_{\ydiagram{4,4,2,2}}.  \]

We have from \cite{newGan} (obtained by modular data considerations) that the category $\mathcal{C}(\mathfrak{sl}_{6}, 6)^0_{A_{\mathfrak{sp}_{20}}^{\textrm{tr}}}$ has eleven simple objects, which we label by the set $\{X_i : 0 \leq i \leq 10\}$. This category has the same fusion rules as $\mathcal{C}(\mathfrak{sl}_2, 10)$. We have the following branching rules from \cite{newGan}:
\[
    X_1 \mapsto \left[ V_{\ydiagram{6,3  }}  \right]_{\mathbb{Z}_3}\oplus  \left[ V_{\ydiagram{4,3,2  }}  \right]_{\mathbb{Z}_3}\qquad X_2 \mapsto \left[ V_{\ydiagram{6,4,2  }}  \right]_{\mathbb{Z}_3}\oplus  \left[ V_{\ydiagram{5,3,2,2 }}  \right]_{\mathbb{Z}_3}\qquad X_{9} \mapsto \left[ V_{\ydiagram{3  }}  \right]_{\mathbb{Z}_3}\oplus  \left[ V_{\ydiagram{6,4,3,2 }}  \right]_{\mathbb{Z}_3}.
\]

With these additional algebra objects, we can now prove Theorem~\ref{thm:spec} in this case.

\begin{proof}[Proof of Theorem~\ref{thm:spec} in the case of $\mathfrak{g} = \mathfrak{sl}_6$]

We have the following table from the results of \cite{newGan,ModPt1}. Here we add a $\dag$ label to the rows where $\mathcal{C}_{A}^0 \simeq  \mathcal{C}_{A'}^0$. The $\mathfrak{sl}_6$ case is the first time we encounter this phenomenon. This coincidence of local modules will result in additional module categories over $\mathcal{C}(\mathfrak{sl}_6,6)$.

\begin{center}\begin{tabular}{c|c| c c c c}
    $\mathcal{C}$ & \textit{\'Etale} Algebras $A$  & $\operatorname{EqBr}(\mathcal{C}_A^0)$ & $|\operatorname{Aut}(A)\backslash\operatorname{EqBr}(\mathcal{C}_A^0)/\operatorname{Aut}(A)|$ \\\hline
      $\mathcal{C}(\mathfrak{sl}_6,4)$ & $\mathbf{1}$ &  $\mathbb{Z}_2^2$ & 4\\
                                      &  $A_2$ &  $\mathbb{Z}_2^3$ & 4\\
                                       & $A_{\mathfrak{sl}_{15}}$ &  $\mathbb{Z}_2^2$ & 4\\ \hline
     $\mathcal{C}(\mathfrak{sl}_6,6)$ & $\mathbf{1}$ &  $\mathbb{Z}_2^2$ & 4\\
                                        &$A_3$ &  $D_3\times \mathbb{Z}_2$ & 4\\
                                         & $A_{\mathfrak{so}_{35}}$ &  $\{e\}$ &1 & $\dag$\\         
                                          & $A_{\mathfrak{sp}_{20}}$ &  $\mathbb{Z}_2$ & 2\\         
                                           & $A_{\mathfrak{sp}_{20}}^{\textrm{ext}}$ &  $\{e\}$ & 1&$\dag$\\    
                                   &$A_{\mathfrak{sp}_{20}}^{\textrm{tr}}$ &  $\mathbb{Z}_2$ & 2&\\ \hline 
  
    $\mathcal{C}(\mathfrak{sl}_6,8)$ & $\mathbf{1}$ &  $\mathbb{Z}_2^2$  & 4\\
                                      &  $A_2$ &   $\mathbb{Z}_2^2$ & 4\\
                                       & $A_{\mathfrak{sl}_{21}}$ &  $\mathbb{Z}_2^2$& 4\\ \hline
                                        
\end{tabular}\end{center}

We first justify the claim that $\mathcal{C}(\mathfrak{sl}_6,6)_{A_{\mathfrak{so}_{35}}}^0 \simeq \mathcal{C}(\mathfrak{sl}_6,6)_{A^\textrm{ext}_{\mathfrak{sp}_{20}}}^0 $. We have that $\mathcal{C}(\mathfrak{sl}_6,6)_{A_{\mathfrak{so}_{35}}}^0 \simeq\mathcal{C}(\mathfrak{so}_{35},1)$, which is an Ising category \cite[Appendix B]{braidedI} with central charge $c = \frac{3}{2} \pmod 8$. On the other hand $\mathcal{C}(\mathfrak{sl}_6,6)_{A_{\mathfrak{sp}^{\textrm{ext}}_{20}}}^0 \simeq \mathcal{C}(\mathfrak{sp}_{20},1)_{\mathbf{1}\oplus V^{\Lambda_6}}^0$. There is a moniodal (but not braided) equivalence $\mathcal{C}(\mathfrak{sp}_{20},1)\simeq \mathcal{C}(\mathfrak{sl}_{2},10)$. Hence we have a monoidal equivalence $\mathcal{C}(\mathfrak{sp}_{20},1)_{\mathbf{1}\oplus V^{\Lambda_6}}^0\simeq \mathcal{C}(\mathfrak{sl}_{2},10)_{\mathbf{1}\oplus V^{6\Lambda_1}}^0$. It is well known that $\mathcal{C}(\mathfrak{sl}_{2},10)_{\mathbf{1}\oplus V^{6\Lambda_1}}^0$ is an Ising category, and hence so is $\mathcal{C}(\mathfrak{sp}_{20},1)_{\mathbf{1}\oplus V^{\Lambda_6}}^0$. The braided equivalence $\mathcal{C}(\mathfrak{sl}_6,6)_{A_{\mathfrak{sp}^{\textrm{ext}}_{20}}}^0 \simeq \mathcal{C}(\mathfrak{sp}_{20},1)_{\mathbf{1}\oplus V^{\Lambda_6}}^0$ shows that the central charge of $\mathcal{C}(\mathfrak{sp}_{20},1)_{\mathbf{1}\oplus V^{\Lambda_6}}^0$ is equal to the central charge of $\mathcal{C}(\mathfrak{sl}_6,6)$ \cite[Corollary 3.30]{LagrangeUkraine}. This is computed to be $\frac{3}{2} \pmod 8$. As the central charge (mod 8) is a complete invariant of Ising categories, we have the braided equivalence 
\begin{equation} \label{eq:equiv6}
  \mathcal{C}(\mathfrak{sl}_6,6)_{A_{\mathfrak{so}_{35}}}^0 \simeq \mathcal{C}(\mathfrak{sl}_6,6)_{A^\textrm{ext}_{\mathfrak{sp}_{20}}}^0.    
\end{equation}
Furthermore, as the braided autoequivalence group of an Ising category is trivial \cite[Theorem 1.2]{ADEaut}, there is a unique equivalence between these two categories.

Next we prove that there is no braided equivalence between $\mathcal{C}(\mathfrak{sl}_6,6)_{A_{\mathfrak{sp}_{20}}}^0$ and $\mathcal{C}(\mathfrak{sl}_6,6)_{A_{\mathfrak{sp}_{20}}^\textrm{tr}}^0$. These categories both have $\mathcal{C}(\mathfrak{sl}_2, 10)$ fusion rules. Braided categories with these fusion rules are completely classified in \cite{fro} by $\ell \in \mathbb{Z}_{48}$ such that $\gcd(\ell, 12) =1$. These categories are labeled $\mathcal{C}^\textrm{br}_{10, \ell}$. There are two pivotal structures on each of these categories, and we use the notation $\mathcal{C}^\textrm{br}_{10, \ell,\pm}$ to distinguish these. The numerical data for these categories has been collated in \cite{surv}. In particular we have the following twist values in $\mathcal{C}^\textrm{br}_{10, \ell,\pm}$:
\[   \theta_{X_1} = \mp e^{6 \pi i \frac{\ell + 12}{48}}  \qquad \text{ and }   \theta_{X_2} =  e^{16 \pi i \frac{\ell + 12}{48}}.    \]

From the branching rules at the beginning of this subsection, we have the following twists in $\mathcal{C}(\mathfrak{sl}_6,6)_{A_{\mathfrak{sp}_{20}}^\textrm{tr}}^0$:
\[      \theta_{X_1} =  e^{2 \pi i \frac{15}{16}}     \qquad \text{ and }   \theta_{X_2} =  e^{2 \pi i \frac{1}{6}} .       \]
Hence as a braided pivotal category we have that $\mathcal{C}(\mathfrak{sl}_6,6)_{A_{\mathfrak{sp}_{20}}^\textrm{tr}}^0$ is equivalent to either $\mathcal{C}^\textrm{br}_{10, 13,-}$ or $\mathcal{C}^\textrm{br}_{10, 43,+}$.

On the other hand, we have a braided equivalence $\mathcal{C}(\mathfrak{sl}_6,6)_{A_{\mathfrak{sp}_{20}}}^0\simeq \mathcal{C}(\mathfrak{sp}_{20},1)$. We compute the twists:
\[  \theta_{V_{\Lambda_1}} =            e^{2 \pi i \frac{7}{16}}     \qquad \text{ and }   \theta_{V_{\Lambda_2}} =  e^{2 \pi i \frac{5}{6}} .       \]
This allows us to see that $\mathcal{C}(\mathfrak{sl}_6,6)_{A_{\mathfrak{sp}_{20}}}^0$ is equivalent as a braided pivotal category to either $\mathcal{C}^\textrm{br}_{10, 5,-}$ or $\mathcal{C}^\textrm{br}_{10, 35,+}$.

The value $\ell$ is a complete invariant of braided categories with $\mathcal{C}(\mathfrak{sl}_2, 10)$ fusion rules. By forgetting the pivotal structure, this allows us to see that there is no braided equivalence between $\mathcal{C}(\mathfrak{sl}_6,6)_{A_{\mathfrak{sp}_{20}}}^0$ and $\mathcal{C}(\mathfrak{sl}_6,6)_{A_{\mathfrak{sp}_{20}}^\textrm{tr}}^0$.

For all remaining \textit{\'etale} algebra objects, we have that $ \mathcal{C}(\mathfrak{sl}_6, k)^0_{A} \not\simeq \mathcal{C}(\mathfrak{sl}_6, k)^0_{B} $ for $A \not\cong B$
by global dimension considerations.

We now determine the number of double cosets in the braided autoequivalence group of $\mathcal{C}(\mathfrak{sl}_6, k)^0_{A}$ for each \textit{\'etale} algebra object $A$. For the pointed algebras $A_m$, and for the algebras $A_{\mathfrak{sl}_{ 15   }}$, $A_{\mathfrak{sl}_{ 21   }}$, and $A_{\mathfrak{so}_{ 35   }}$ we can apply the same logic used in the previous subsections for smaller $N$.

In the case of the algebra $A_{\mathfrak{sp}_{ 20 }}$ we have from \cite[Theorem 1.3]{ModPt1} that $\operatorname{EqBr}\left(  \mathcal{C}(\mathfrak{sp}_{20},1) \right) \cong \mathbb{Z}_2$, and that the non-trivial braided autoequivalence maps $V^{\Lambda_1} \leftrightarrow V^{\Lambda_{9}}$. This autoequivalence does not preserve the forgetful functor from Subsection~\ref{sub:spor}, and hence is not induced by an element of $\operatorname{Aut}\left(A_{\mathfrak{sp}_{ 20 }}\right)$. Hence the image of $\operatorname{Aut}\left(A_{\mathfrak{sp}_{ 20 }}\right)$ in $\operatorname{EqBr}\left(  \mathcal{C}(\mathfrak{sl}_{6},6)^0_{A_{\mathfrak{sp}_{ 20 }}} \right)$ is trivial, and so 
\[|\operatorname{Aut}\left(A_{\mathfrak{sp}_{ 20 }}\right)\backslash\operatorname{EqBr}\left(\mathcal{C}(\mathfrak{sl}_6,6 )^0_{A_{\mathfrak{sp}_{ 20 }}}\right)/\operatorname{Aut}\left(A_{\mathfrak{sp}_{ 20 }}\right)| = 2.\]

For the algebra $A^\textrm{tr}_{\mathfrak{sp}_{ 20 }}$ we have from an earlier discussion that $\mathcal{C}(\mathfrak{sl}_6, k)^0_{^\textrm{tr}_{\mathfrak{sp}_{ 20 }}}$ is a category with $\mathcal{C}(\mathfrak{sl}_2, 10)$ fusion rules. The braided autoequivalence group of such a category is $\mathbb{Z}_2$, where the non-trivial autoequivalence exchanges $X_1\leftrightarrow X_{9}$ and $X_3 \leftrightarrow X_7$. From the branching rules at the beginning of this subsection, we can see this autoequivalence does not preserve the forgetful functor. Hence 
\[|\operatorname{Aut}\left(A^\textrm{tr}_{\mathfrak{sp}_{ 20 }}\right)\backslash\operatorname{EqBr}\left(\mathcal{C}(\mathfrak{sl}_6,6 )^0_{A^\textrm{tr}_{\mathfrak{sp}_{ 20 }}}\right)/\operatorname{Aut}\left(A^\textrm{tr}_{\mathfrak{sp}_{ 20 }}\right)| = 2.\]

For the algebra $A^\textrm{ext}_{\mathfrak{sp}_{ 20 }}$ we have from the earlier discussion that $\mathcal{C}(\mathfrak{sl}_6,6 )^0_{A^\textrm{ext}_{\mathfrak{sp}_{ 20 }}}$ is an Ising category. The braided autoequivalence group of the Ising categories are known to be trivial. Hence 
\[|\operatorname{Aut}\left(A^\textrm{ext}_{\mathfrak{sp}_{ 20 }}\right)\backslash\operatorname{EqBr}\left(\mathcal{C}(\mathfrak{sl}_6,6 )^0_{A^\textrm{ext}_{\mathfrak{sp}_{ 20 }}}\right)/\operatorname{Aut}\left(A^\textrm{ext}_{\mathfrak{sp}_{ 20 }}\right)| = 1.\]

The above table gives 14 triples of the form $(A, \mathcal{F}, A)$ for each of the categories $\mathcal{C}(\mathfrak{sl}_6,4)$, $\mathcal{C}(\mathfrak{sl}_6,6)$, and $\mathcal{C}(\mathfrak{sl}_6,8)$. Furthermore, in the case of $\mathcal{C}(\mathfrak{sl}_6,6)$ we have the two additional triples 
\[   \left( A_{\mathfrak{so}_{ 35 }} , \mathcal{F},  A^\textrm{ext}_{\mathfrak{sp}_{ 20 }} \right)\qquad \text{ and } \qquad  \left(  A^\textrm{ext}_{\mathfrak{sp}_{ 20 }} , \mathcal{F}^{-1},  A_{\mathfrak{so}_{ 35 }}\right)       \]
where $\mathcal{F}$ is the braided equivalence from Equation~\eqref{eq:equiv6}. The result follows from \cite[Corollary 3.8]{triples}.
\end{proof}

\subsection*{The $\mathfrak{sl}_7$ case}\label{sub:7}

In the $\mathfrak{sl}_7$ case we again encounter interesting \textit{\'etale} algebra objects. There are again two of these, and they both occur in $\mathcal{C}(\mathfrak{sl}_7, 7)$.

The first is constructed in a similar fashion to the algebra $A_{\mathfrak{sp}_{20}}^\textrm{ext}$ from the previous subsection. From the conformal embedding
\[      \mathcal{V}(\mathfrak{sl}_7,7) \subset    \mathcal{V}(\mathfrak{so}_{48},1)     \]
we obtain the \textit{\'etale} algebra object $A_{\mathfrak{so}_{48}}\in \mathcal{C}(\mathfrak{sl}_7,7)$. Furthermore, we have a braided equivalence
\[    \mathcal{C}(\mathfrak{sl}_7,7)^0_{A_{\mathfrak{so}_{48}}} \simeq \mathcal{C}(\mathfrak{so}_{48},1).            \]
The braided category $\mathcal{C}(\mathfrak{so}_{48},1)$ is pointed, with underlying group $\mathbb{Z}_2\times \mathbb{Z}_2$, and quadratic form $q = (1, 1,-1,1)$. This implies that the object $\mathbf{1} \oplus S^+\in\mathcal{C}(\mathfrak{so}_{48},1)$ is an \textit{\'etale} algebra object. It follows that $\operatorname{For}(\mathbf{1} \oplus S^+)$ is an \textit{\'etale} algebra object in $\mathcal{C}(\mathfrak{sl}_7,7)$.
\begin{defn}
    We will write $A_{\mathfrak{so}_{48}}^\textrm{ext}:= \operatorname{For}(\mathbf{1} \oplus S^+)\in \mathcal{C}(\mathfrak{sl}_7,7)$.
\end{defn}

Explicitly we have that
\[    A_{\mathfrak{so}_{48}}^\textrm{ext} \cong A_{\mathfrak{so}_{48}}\oplus 4 \cdot V_{\ydiagram{5,4,3,2,1}},   \]
and that
\[       \mathcal{C}(\mathfrak{sl}_7,7)_{A_{\mathfrak{so}_{48}}^\textrm{ext}}^0 \simeq \mathcal{C}(\mathfrak{so}_{48},1)_{\mathbf{1} \oplus S^+}^0 \simeq \operatorname{Vec}.   \]
\begin{rmk}
    More generally, whenever $N \equiv \pm 1 \pmod 8$, the above construction works to give an \textit{\'etale} algebra object $A_{\mathfrak{so}_{N^2-1}}^\textrm{ext}$ in $\mathcal{C}(\mathfrak{sl}_7,7)$.
\end{rmk}

The second new \textit{\'etale} algebra object in $\mathcal{C}(\mathfrak{sl}_7,7)$ is much more exotic. This algebra comes from an exotic VOA extension of $\mathcal{V}(\mathfrak{sl}_7,7)\subset \mathcal{V}'$ \cite{A67}. This exotic VOA was first predicted by Schellekens \cite{Bert}.
\begin{defn}
    We will write $A_{\textrm{Schellekens}}$ for the \textit{\'etale} algebra object in $\mathcal{C}(\mathfrak{sl}_7,7)$ coming from the above VOA extension.
\end{defn}

We have that 
\[ A_{\textrm{Schellekens}}\cong \left[ \mathbf{1} \right]_{\mathbb{Z}_7} \oplus   \left[ V_{\ydiagram{4,4,2,2,2}} \right]_{\mathbb{Z}_7}\oplus   \left[ V_{\ydiagram{6,6,5,3,1}} \right]_{\mathbb{Z}_7}\oplus   \left[ V_{\ydiagram{4,4,4,3,1,1}} \right]_{\mathbb{Z}_7} \oplus   \left[ V_{\ydiagram{4,3,3,1}} \right]_{\mathbb{Z}_7}\oplus 3\cdot V_{\ydiagram{6,5,4,3,2,1}}, \]
and that 
\[\mathcal{C}(\mathfrak{sl}_7,7)^0_{A_{\textrm{Schellekens}}}\simeq \operatorname{Vec}.\]

With this information compiled, we are ready to finish our proof of Theorem~\ref{thm:spec}.
\begin{proof}[Proof of Theorem~\ref{thm:spec} in the case of $\mathfrak{g} = \mathfrak{sl}_7$]
We have the following data. Again, the $\dag$ represents when two algebras $A$ and $B$ have braided equivalent categories of local modules.

\begin{center}\begin{tabular}{c|c| c c c c}
    $\mathcal{C}$ & \textit{\'Etale} Algebras $A$  & $\operatorname{EqBr}(\mathcal{C}_A^0)$ & $|\operatorname{Aut}(A)\backslash\operatorname{EqBr}(\mathcal{C}_A^0)/\operatorname{Aut}(A)|$ \\\hline
      $\mathcal{C}(\mathfrak{sl}_7,5)$ & $\mathbf{1}$ &  $\mathbb{Z}_2^2$ & 4\\
                                       & $A_{\mathfrak{sl}_{21}}$ &  $\mathbb{Z}_2^2$ & 4\\ \hline
     $\mathcal{C}(\mathfrak{sl}_7,7)$ & $\mathbf{1}$ &  $\mathbb{Z}_2$ & 2\\
                                        &$A_7$ &  $D_7\times \mathbb{Z}_2$ & 2\\
                                         & $A_{\mathfrak{so}_{48}}$ &  $\mathbb{Z}_2$ &2 &  \\        
                                        & $A^\textrm{ext}_{\mathfrak{so}_{48}}$ &  $\{e\}$ &1 &\dag  \\ 
                                        & $A_{\textrm{Schellekens}}$ &  $\{e\}$ &1 &\dag \\ \hline
    $\mathcal{C}(\mathfrak{sl}_7,9)$ & $\mathbf{1}$ &  $\mathbb{Z}_2^2$  & 4\\
                                       & $A_{\mathfrak{sl}_{28}}$ &  $\mathbb{Z}_2^2$& 4\\ \hline
                                        
\end{tabular}\end{center}
\end{proof}
It is immediate that there is a braided equivalence 
\begin{equation}\label{eq:equiv7}\mathcal{C}(\mathfrak{sl}_7,7)^0_{A^\textrm{ext}_{\mathfrak{so}_{48}}}  \simeq  \operatorname{Vec} \simeq       \mathcal{C}(\mathfrak{sl}_7,7)^0_{A_{\textrm{Schellekens}}}.   \end{equation}
The remaining categories $\mathcal{C}(\mathfrak{sl}_7,k)_A^0$ are seen to be pairwise distinct by global dimension considerations.

We now move on to determining the number of double cosets in $\operatorname{EqBr}(\mathcal{C}(\mathfrak{sl}_7,k)_A^0)$ with respect to the image of $\operatorname{Aut}(A)$. The arguments for the algebras $A_m$, $A_{\mathfrak{sl}_{21}}$, $A_{\mathfrak{sl}_{28}}$, and $A_{\mathfrak{so}_{48}}$ are near identical to the previous subsections. The remaining algebras $A^\textrm{ext}_{\mathfrak{so}_{48}}$ and $A_{\textrm{Schellekens}}$ are trivial to deal with, as $\mathcal{C}(\mathfrak{sl}_7,k)_A^0\simeq \operatorname{Vec}$ in these cases, which has trivial autoequivalence group.

We thus have 8 triples of the form $(A, \mathcal{F}, A)$ for the each of the categories $\mathcal{C}(\mathfrak{sl}_7,5)$, $\mathcal{C}(\mathfrak{sl}_7,7)$ and $\mathcal{C}(\mathfrak{sl}_7,9)$. Furthermore, in the case of $\mathcal{C}(\mathfrak{sl}_7,7)$ we have the two additional triples 
\[   \left( A^\textrm{ext}_{\mathfrak{so}_{ 48 }} , \mathcal{F},  A_{\textrm{Schellekens}} \right)\qquad \text{ and } \qquad  \left(  A_{\textrm{Schellekens}} , \mathcal{F}^{-1},  A^\textrm{ext}_{\mathfrak{so}_{ 48 }}\right)       \]
where $\mathcal{F}$ is the braided equivalence from Equation~\eqref{eq:equiv7}.

The result now follows from \cite[Corollary 3.8]{triples}.

\appendix
\section{Erratum to Lemma 3.9 in part I}\label{app:fix}

This appendix is to fix a minor (but non-trivial) error in \cite{ModPt1}. We will freely use the notation of the cited paper in this appendix. The proof of Lemma 3.9 in \cite{ModPt1} is incomplete as written, as it fails to consider several special cases. The statement of this lemma needs a slight alteration to deal with one of these special cases $\mathcal{C}(\mathfrak{sl}_4,4)^{ad}_{\operatorname{Rep}(\mathbb{Z}_2)}$. The correct statement of the lemma is as follows.

\begin{lem}[Corrected Statement of Lemma 3.9]\label{lem:new}
    Apart from the case $(r,k,m) = (3,4,2)$, we have
    \[  \operatorname{EqBr}\left( \mathcal{C}(\mathfrak{sl}_{r+1},k)^{ad}_{\operatorname{Rep}(\mathbb{Z}_{m'})}  \right)  \cong \begin{cases}
        \mathbb{Z}_{m'} \quad \text{ if $r=1$ or $k=2$}\\
        D_{m'} \quad \text{ otherwise.}\\
    \end{cases} \]
    In the remaining case, we have that $\operatorname{EqBr}\left( \mathcal{C}(\mathfrak{sl}_{4},4)^{ad}_{\operatorname{Rep}(\mathbb{Z}_{2})}  \right)$ is isomorphic to either $D_2$ or $\mathbb{Z}_2$.
\end{lem}

This does not affect the statement of Theorem 3.2, as we can deal with the ambiguity of the $ \mathcal{C}(\mathfrak{sl}_{4},4)^{ad}_{\operatorname{Rep}(\mathbb{Z}_{2})}$ case by slightly altering the proof of Theorem 3.2. We will give the new proof at the end of the Erratum. All other statements and proofs are entirely unaffected by the changes to Lemma 3.9.

\subsection*{The correct proof of Lemma 3.9}

We begin by discussing the oversight in our proof of Lemma 3.9. The step in question is in showing that the braided autoequivalence of $\mathcal{C}(\mathfrak{sl}_{r+1}, k)^{ad}_{\operatorname{Rep}(\mathbb{Z}_{m'})}$ defined on objects by 
\[   (X, \chi_n) \mapsto (X^*, \chi_{-n})      \]
is non-isomorphic to any of the autoequivalences which are defined on objects by
\[      (X, \chi_n) \mapsto (X, \chi_{n+i}).      \]          
As explained in the proof of Lemma 3.9 it is a sufficient condition to find an object $X\in \mathcal{C}(\mathfrak{sl}_{r+1}, k)^{ad}$ such that $X^* \not \cong k\Lambda_{i}\otimes X$ for any $i \in \mathbb{Z}_{r+1}$. It is claimed in the paper that $X \cong (k-3)\Lambda_0 + \Lambda_1 + 2\Lambda_{r-1}$ works for all $N\geq 4$. We first point out that this is a typo, as this $X$ is not an object of $\mathcal{C}(\mathfrak{sl}_{r+1}, k)^{ad}$. The intended object was 
\[          X \cong (k-3)\Lambda_0+2\Lambda_1 + \Lambda_{r-1}.\]
More seriously, this object fails the required condition with at several specific values of $(r,k)$. These exceptions were missed in the proof of Lemma 3.9.

\begin{prop}\label{prop:fix}
 Let $r\geq 2$, $ k\geq 3$, and choose $X \cong (k-3)\Lambda_0+2\Lambda_1 + \Lambda_{r-1}$. Then apart from the cases $(r,k) \in \{(2,3),  (2,6), (3,4), (4,5), (5,3)\}$, we have
 \[   X^* \not \cong k\Lambda_{i}\otimes X   \]
 for any $i \in \mathbb{Z}_{r+1}$.
\end{prop}
\begin{proof}
    Suppose that
    \[   X^* \cong k\Lambda_{i}\otimes X   \]
 for some $i \in \mathbb{Z}_{r+1}$. To deduce that we are in one of the five special cases we will split into cases.
    \begin{trivlist}\leftskip=2em
        \item \textbf{Case $r= 2$}:

        Here we have $k\Lambda_{i}\otimes X \cong (k-3)\Lambda_i + 3\Lambda_{1+i}$, and $X^* \cong (k-3)\Lambda_0 + 3\Lambda_2$. If $k\neq 6$, then we must have the equations
        \[   i \equiv 0 \pmod 3 \quad \text{ and }\quad i+1 \equiv 2 \pmod 3.    \]
        This system clearly has no solution. If $k=6$ then we have the solution
        \[ k\Lambda_{2}\otimes X \cong 3\Lambda_0 + 3\Lambda_{2}    \cong X^*.  \]
        This gives the exception $(r,k) = (2,6)$ in the statement of the proposition.

        \vspace{1em}
        
        For the remainder of the cases we will have $r>2$, and so 
        \[  k\Lambda_{i}\otimes X \cong (k-3)\Lambda_i + 2\Lambda_{1+i} + \Lambda_{r-1+i}\quad \text{ and } \quad  X^* \cong (k-3)\Lambda_0  + \Lambda_2 + 2\Lambda_r.      \]
        
        \vspace{1em}

        \item \textbf{Case $k= 3$ and $r>2$}:

        Here we obtain the equations
        \[  2\equiv r-1 +i \pmod {r+1} \quad \text{ and } r \equiv 1+i \pmod {r+1}.   \]
        Solving this system gives that either $r = 5$ or $r = 2$. These give the two exceptions $(2,3)$ and $(5,3)$ in the statement of the proposition.

        \vspace{1em}

        \item \textbf{Case $k= 4$ and $r>2$}:

        Here we have that
        \[  r\equiv 1+i \pmod {r+1} \]
        and that either
        \[    0\equiv i \pmod {r+1} \quad \text{ and } 2 \equiv r-1+i \pmod {r+1}    \]
        or
        \[    2\equiv i \pmod {r+1} \quad \text{ and } 0 \equiv r-1+i \pmod {r+1}.    \]
        The first only has the solution $r=1$ which we can ignore, and the second has the solutions $r=1$ and $r=3$. Again we can ignore the $r=1$ solution, and the $r=3$ solution gives the $(r,k) = (3,4)$ exception in the statement of the proposition.

        \vspace{1em}

        \item \textbf{Case $k= 5$ and $r>2$}:

        Here we have that
        \[  2\equiv r-1+i \pmod {r+1} \]
        and that either
        \[    0\equiv i \pmod {r+1} \quad \text{ and } r \equiv 1+i \pmod {r+1}    \]
        or
        \[    i\equiv r \pmod {r+1} \quad \text{ and } 0 \equiv 1+i \pmod {r+1}.    \]
        The first only has the solution $r=1$ which is ignored, and the second only has the solution $r=4$. This gives the exception $(r,k) = (4,5)$ in the statement of the proposition.

        \vspace{1em}

        \item \textbf{Case $k\geq 6$ and $r>2$}:

        Here we get the equations
        \[       i \equiv 0 \pmod{r+1} \qquad  2 \equiv r-1+i \pmod{r+1}, \quad \text{ and } \quad r \equiv 1+i \pmod {r+1}.  \]
        This only has the solution $r=1$ which is ignored.
\end{trivlist}
\end{proof}
This proposition shows that the statement of Lemma 3.9 holds for all but the five special cases $(r,k) \in \{(2,3),  (2,6), (3,4), (4,5), (5,3)\}$. The case of $\mathcal{C}( \mathfrak{sl}_3, 3  )^{ad}_{\operatorname{Rep}(\mathbb{Z}_3)}$ is not included in the statement of Lemma 3.9 due to Lemma 2.15 and Remark 2.16. For the remaining four cases we have to show that the autoequivalences of $\mathcal{C}( \mathfrak{sl}_{r+1}, k  )^{ad}_{\operatorname{Rep}(\mathbb{Z}_{m'})}$
\begin{equation}\label{eq:autoc}   (X, \chi_n) \mapsto (X^*, \chi_{-n})     \end{equation}
and
\begin{equation}\label{eq:autoi}      (X, \chi_n) \mapsto (X, \chi_{n+i}).     \end{equation}
are distinct. In the special case of $m'=1$, this follows from \cite[Theorem 1.1]{autos}.

For the cases $(r,k) \in \{(2,6), (4,5), (5,3)\}$ we are left with the cases where $m'$ is $3$, $5$, and $3$ respectively. Here we consider the objects $(\Lambda_1 + \Lambda_r, \chi_1)$, and $(\Lambda_1 + \Lambda_r, \chi_{-1})$, which are distinct as $\operatorname{Stab}(\Lambda_1 + \Lambda_r) \cong \mathbb{Z}_{m'}$, and $m'>2$ in all cases. The autoequivalence in Equation~\eqref{eq:autoc} exchanges these two objects. As $m' >2$ in all cases, it is immediate that there is no choice of $i$ such that the autoequivalence in Equation~\eqref{eq:autoi} exchanges $(\Lambda_1 + \Lambda_r, \chi_1)$ and $(\Lambda_1 + \Lambda_r, \chi_{-1})$. Hence we have the desired outcome for these three cases.

In the final case of $\mathcal{C}( \mathfrak{sl}_{4}, 4  )^{ad}_{\operatorname{Rep}(\mathbb{Z}_{2})}$ we are unable to use combinatorial arguments to distinguish the map 
\[   (X, \chi_n) \mapsto (X^*, \chi_{-n})      \]
from the identity. Hence we have the ambiguity that either $\operatorname{EqBr}\left( \mathcal{C}(\mathfrak{sl}_{4},4)^{ad}_{\operatorname{Rep}(\mathbb{Z}_{2})}  \right)\cong D_2$ if the above autoequivalence is not naturally isomorphic to the identity, or $\operatorname{EqBr}\left( \mathcal{C}(\mathfrak{sl}_{4},4)^{ad}_{\operatorname{Rep}(\mathbb{Z}_{2})}  \right)\cong \mathbb{Z}_2$ if it is. All together we have the proof of Lemma~\ref{lem:new}.

\subsection*{The adjustment to the proof of Theorem 3.2}

To finish up, we have to show that the ambiguity in Lemma~\ref{lem:new} can be resolved in the proof of Theorem 3.2, without the statement of Theorem 3.2 being affected. That is, we have to prove that $\operatorname{EqBr}\left( \mathcal{C}(\mathfrak{sl}_{4},4)^{0}_{\operatorname{Rep}(\mathbb{Z}_{2})}  \right)\cong \mathbb{Z}_2\times D_2$. With the new ambiguity present in Lemma~\ref{lem:new}, we have that either $\operatorname{EqBr}\left( \mathcal{C}(\mathfrak{sl}_{4},4)^{0}_{\operatorname{Rep}(\mathbb{Z}_{2})}  \right)\cong \mathbb{Z}_2\times D_2$ or $\operatorname{EqBr}\left( \mathcal{C}(\mathfrak{sl}_{4},4)^{0}_{\operatorname{Rep}(\mathbb{Z}_{2})}  \right)\cong \mathbb{Z}_2\times \mathbb{Z}_2$. To rule out the latter possibility, we will simply find 8 non pairwise isomorphic braided autoequivalences of $\mathcal{C}(\mathfrak{sl}_{4},4)^{0}_{\operatorname{Rep}(\mathbb{Z}_{2})}$. 

The first is given by the action of $\operatorname{Rep}(\mathbb{Z}_2)$, which is defined on objects by 
\[
    (X, \chi_n)\mapsto  (X, \chi_{n+1}) . 
\]
We label this autoequivalence $\mathcal{F}_{\operatorname{Rep}(\mathbb{Z}_2)}$. We have that $\mathcal{F}_{\operatorname{Rep}(\mathbb{Z}_2)}\circ \mathcal{F}_{\operatorname{Rep}(\mathbb{Z}_2)} \cong \operatorname{Id}$.

We have another braided autoequivalence coming from simple currents as in \cite[Corollary 3.11]{ModPt1}. This autoequivalence is distinguished by the fact that it restricts to the identity on the subcategory $\mathcal{C}(\mathfrak{sl}_{4},4)^{ad}_{\operatorname{Rep}(\mathbb{Z}_{2})}$. We label this autoequivalence $\mathcal{F}_{  (4\Lambda_1, \chi_0)}$. This braided autoequivalence also has order 2.

Finally we have a braided autoequivalence which is descended from the charge conjugation autoequivalence of $\mathcal{C}(\mathfrak{sl}_{4},4)$. This autoequivalence is defined on objects by 
\[     (X,\chi_n) \mapsto   (X^*,\chi_{-n}).     \]
As for the previous examples, this braided autoequivalence has order 2. We label this autoequivalence $\mathcal{F}_{c}$.

While we previously couldn't distinguish the autoequivalence $\mathcal{F}_c$ from the identity when restricted to the subcategory $\mathcal{C}(\mathfrak{sl}_{4},4)^{ad}_{\operatorname{Rep}(\mathbb{Z}_{2})}$, it is much easier to distinguish them on the full category $\mathcal{C}(\mathfrak{sl}_{4},4)^{0}_{\operatorname{Rep}(\mathbb{Z}_{2})}$. Indeed, we consider the object $(2\Lambda_1, \chi_0)\in \mathcal{C}(\mathfrak{sl}_{4},4)^{0}_{\operatorname{Rep}(\mathbb{Z}_{2})}$. We have that the charge conjugation autoequivalence maps
\[     (2\Lambda_1, \chi_0)\mapsto (2\Lambda_3, \chi_0).   \]
As $2\Lambda_3 \not \cong (4\Lambda_2) \otimes (2\Lambda_1)$, we have that $  (2\Lambda_1, \chi_0)\not \cong (2\Lambda_3, \chi_0)$, and so the charge congugation autoequivalence is not isomorphic to the identity.

Furthermore, we have that $\mathcal{F}_{\operatorname{Rep}(\mathbb{Z}_2)}\circ \mathcal{F}_{c} \not \cong \mathcal{F}_{  (4\Lambda_1, \chi_0)}$. This can be checked at the level of the fusion ring as follows:
\[   ( \mathcal{F}_{\operatorname{Rep}(\mathbb{Z}_2)}\circ \mathcal{F}_{c})(\Lambda_1+\Lambda_3, \chi_0 ) \cong  (\Lambda_1+\Lambda_3, \chi_1)\not\cong (\Lambda_1+\Lambda_3, \chi_0) \cong \mathcal{F}_{  (4\Lambda_1, \chi_0)}(\Lambda_1+\Lambda_3, \chi_0).       \]
Here we have used that $\mathcal{F}_{  (4\Lambda_1, \chi_0)}$ restricts to the identity on the subcategory $\mathcal{C}(\mathfrak{sl}_{4},4)^{ad}_{\operatorname{Rep}(\mathbb{Z}_{2})}$.

From the above computations, we have that $\operatorname{EqBr}\left( \mathcal{C}(\mathfrak{sl}_{4},4)^{0}_{\operatorname{Rep}(\mathbb{Z}_{2})}  \right)$ contains 3 distinct order 2 elements, no two of which compose together to give the remaining element. This implies that the order of $\operatorname{EqBr}\left( \mathcal{C}(\mathfrak{sl}_{4},4)^{0}_{\operatorname{Rep}(\mathbb{Z}_{2})}  \right)$ is at least 8, and so 
\[\operatorname{EqBr}\left( \mathcal{C}(\mathfrak{sl}_{4},4)^{0}_{\operatorname{Rep}(\mathbb{Z}_{2})}  \right)\cong \mathbb{Z}_2\times D_2\]
as claimed.

\bibliography{mod}
\bibliographystyle{alpha}
\end{document}